\documentclass[a4paper,10pt]{article}
\usepackage{amsmath,amssymb,amsthm}
\theoremstyle{definition}
 \newtheorem{dfn}{Definition}[section]
 \newtheorem{remark}[dfn]{Remark}  
\theoremstyle{plain}
 \newtheorem{thm}[dfn]{Theorem}
 \newtheorem{prop}[dfn]{Proposition}
 \newtheorem{lem}[dfn]{Lemma}

\numberwithin{equation}{section}
\newcommand{\DV}{{\rm Div}\,}
\newcommand{\dv}{{\rm div}\,}

\newcommand{\BR}{{\mathbb R}}
\newcommand{\BC}{{\mathbb C}} 

\newcommand{\BN}{{\mathbb N}}

\newcommand{\BPi}{{\boldsymbol\Pi}}
\newcommand{\tg}{\tilde{g}}

\newcommand{\CA}{{\mathcal A}}
\newcommand{\CB}{{\mathcal B}}

\newcommand{\CD}{{\mathcal D}}
\newcommand{\CE}{{\mathcal E}}
\newcommand{\CF}{{\mathcal F}}

\newcommand{\CK}{{\mathcal K}}

\newcommand{\CN}{{\mathcal N}}
\newcommand{\CG}{{\mathcal G}}

\newcommand{\CR}{{\mathcal R}}

\newcommand{\CH}{{\mathcal H}}
\newcommand{\CO}{{\mathcal O}}

\newcommand{\CV}{{\mathcal V}}

\newcommand{\bA}{{\mathbf A}} 
\newcommand{\bB}{{\mathbf B}}

\newcommand{\bD}{{\mathbf D}}

\newcommand{\bK}{{\mathbf K}}
\newcommand{\bI}{{\mathbf I}}

\newcommand{\bG}{{\mathbf G}}

\newcommand{\bT}{{\mathbf T}}
\newcommand{\bH}{{\mathbf H}}

\newcommand{\bU}{{\mathbf U}}
\newcommand{\bV}{{\mathbf V}}
\newcommand{\bW}{{\mathbf W}}
\newcommand{\ba}{{\mathbf a}}
\newcommand{\bb}{{\mathbf b}}
\newcommand{\bd}{{\mathbf d}}
\newcommand{\bh}{{\mathbf h}}
\newcommand{\bk}{{\mathbf k}}

\newcommand{\bn}{{\mathbf n}}

\newcommand{\bw}{{\mathbf w}}

\newcommand{\bv}{{\mathbf v}}
\newcommand{\bu}{{\mathbf u}}
\newcommand{\bff}{{\mathbf f}}
\newcommand{\bg}{{\mathbf g}}
\newcommand{\fg}{{\frak g}}
\newcommand{\fp}{{\frak p}}
\newcommand{\fq}{{\frak q}}
\newcommand{\pd}{\partial}

\newcommand{\curl}{{\rm curl\,}}
\newcommand{\rot}{{\rm rot\,}}

\begin{document}
\title{\bf Local Well-posedness for the Magnetohydrodynamics 
in the different two liquids case}
\author{Elena Frolova\thanks{Department of Mathematics and 
Mechanics,
\endgraf St. Petersburg State Electrotechnical
University, Prof. Popova 5, 191126
St. Petersburg, Russia:
\endgraf
e-mail address: elenafr@mail.ru}
\enskip and \enskip Yoshihiro Shibata  
\thanks{Department of Mathematics,  Waseda University, 
\endgraf
Department of Mechanical Engineering and Materials Science,
University of Pittsburgh, USA \endgraf
mailing address: 
Department of Mathematics, Waseda University,
Ohkubo 3-4-1, Shinjuku-ku, Tokyo 169-8555, Japan. \endgraf
e-mail address: yshibata325@gmail.com \endgraf
Partially supported by Top Global University Project, 
JSPS Grant-in-aid for Scientific Research (A) 17H0109, and Toyota Central Research Institute Joint Research Fund }}
\date{}
\maketitle

\begin{abstract}
We prove the local well-posedness for a two phase problem of 
magnetohydrodynamics with a sharp interface.  The solution
is obtained in the maximal regularity space:
$H^1_p((0, T), L_q) \cap L_p((0, T), H^2_q)$ with 
$1 < p, q < \infty$ and $2/p + N/q < 1$, where $N$ is the 
space dimension. 
\end{abstract}

\section{Introduction}\label{sec:1} 
We consider a two phase problem govering the motion of 
two incompressible electrically conducting capillary liquids 
separated by a sharp interface. The problem is formulated as follows: 
Let $\Omega_+$ and $\Omega_-$ be two reference domains in 
 the $N$-dimensional Euclidean space
$\BR^N$ ($N \geq 2$).  Assume that the boundary of each 
$\Omega_\pm$ consists of two connected components $\Gamma$ and $S_\pm$, where
$\Gamma$ is the common boundary of $\Omega_\pm$.  Throughout the paper, we assume 
that $\Gamma$ is a compact hypersurface of $C^3$ class, that $S_\pm$ are
hypersurfaces of $C^2$ class,  and  that 
${\rm dist}\,(\Gamma, S_\pm) \geq d_\pm$ with some positive constants $d_\pm$, where
 the  ${\rm dist}(A, B)$ denotes the 
distance of any subsets $A$ and $B$ of $\BR^N$
 which is defined by setting ${\rm dist}(A, B)
= \inf\{|x-y| \mid x \in A, y \in B\}$.
Let $\Omega = \Omega_+ \cup \Gamma \cup \Omega_-$
and $\dot\Omega = \Omega_+ \cup \Omega_-$. The boundary of $\Omega$ 
	is $S_+\cup S_-$.  We may consider 
the case that one of $S_\pm$ is an empty set, or that both of $S_\pm$ are
empty sets. Let $\Gamma_t$ be an evolution of $\Gamma$ for time $t > 0$,
which is assumed to be given by
\begin{equation}\label{free-surface:1}
\Gamma_t = \{x = y + h(y, t)\bn(y) \mid y \in \Gamma\}
\end{equation}
with an unknown function $h(y, t)$.  We assume that $h|_{t=0} = h_0(y)$ is a
given function. Let $\Omega_{t \pm}$ are two connected components of 
$\Omega\setminus\Gamma_t$ such that the boundary of $\Omega_{t\pm}$
consists of $\Gamma_t$ and $S_\pm$.  
Let $\bn_t$ be the unit  outer normal to
$\Gamma_t$ oriented from $\Omega_{t+}$ into $\Omega_{t-}$, and 
let $\bn_\pm$ be  respective the unit outer normal to $S_\pm$. 
 Given any functions, $v_\pm$,  defined on $\Omega_{t\pm}$, $v$ is
defined by $v(x) = v_{\pm}(x)$ for $x \in \Omega_{t\pm}$ for $t \geq 0$,
where $\Omega_{0\pm} = \Omega_\pm$. Moreover, what
$v = v_\pm$ denotes that $v(x) = v_+(x)$ for $x \in \Omega_{t+}$ and 
$v(x) = v_{t-}(x)$ for $x \in \Omega_{t-}$.  Let
$$[[v]](x_0) = \lim_{x\to x_0 \atop x \in \Omega_{t+}}v_+(x)
- \lim_{x\to x_0 \atop x \in \Omega_{t-}}v_-(x)
$$
for every point $x_0 \in \Gamma_t$, which is the jump quantity of $v$ across
$\Gamma$. Let $\dot \Omega
= \Omega_+ \cup \Omega_-$ and
$\dot \Omega_t = \Omega_{t+} \cup \Omega_{t-}$.  
The purpose of this paper is to prove the local well-posedness of the  
magnetohydrodynamical equations with interface condition
 in the different two liquids case, which are formulated by the set of
 the following equations:
 \allowdisplaybreaks
\begin{alignat}2
\rho(\pd_t\bv + \bv\cdot\nabla\bv) - \DV(\bT(\bv, \fp) + \bT_M(\bH))  = 0,
\quad \dv\bv= 0&
&\quad &\text{in $\bigcup_{0 < t < T}\, \dot\Omega_t \times\{t\}$}, 
\nonumber \\
[[(\bT(\bv, \fp)+ \bT_M(\bH))\bn_t]] = \sigma H(\Gamma_t)\bn_t
-\fp_0\bn_t,
\quad [[\bv]]=0, \quad V_{\Gamma_t}  = \bv_+\cdot\bn_t& 
&\quad &\text{on $\bigcup_{0 < t < T}\, \Gamma_t \times\{t\}$}, 
\nonumber \\
\mu\pd_t\bH + \DV\{\alpha^{-1}\curl\bH - \mu
(\bv\otimes\bH - \bH\otimes\bv)\} = 0,
\quad \dv \bH = 0&
&\quad &\text{in $\bigcup_{0 < t < T}\, \dot\Omega_t \times\{t\}$}, \nonumber \\
[[\{\alpha^{-1}\curl\bH - \mu(\bv\otimes\bH - \bH\otimes\bv)\}
\bn_t]]  = 0&
&\quad &\text{on $\bigcup_{0 < t < T}\, \Gamma_t \times\{t\}$}, \nonumber  \\
[[\mu\bH\cdot\bn_t]]=0,
\quad [[\bH-<\bH, \bn_t>\bn_t]]=0&
&\quad &\text{on $\bigcup_{0 < t < T}\, \Gamma_t \times\{t\}$}, \nonumber \\
\bv_\pm=0, \quad \bn_\pm\cdot\bH_\pm =0, \quad
(\curl\bH_\pm)\bn_\pm = 0&
&\quad&\text{on $S_\pm\times(0, T)$}, \nonumber\\
(\bv, \bH, h)|_{t=0}  = (\bv_0, \bH_0, h_0)&
&\quad&\text{in $\dot\Omega\times \Gamma$}.\label{mhd.1}
\end{alignat}
Here, $\bv = \bv_{\pm} = (v_{\pm 1}(x, t), \ldots, v_{\pm N}(x, t))^\top$
are the velocity vector fields, where $M^\top$ stands for the 
transposed $M$, $\fp = \fp_\pm(x, t)$ the pressure fields, and 
$\bH = \bH_\pm = (H_{\pm 1}(x, t), \ldots, H_{\pm N}(x, t))^\top$
the magnetic fields.  The $\bv$, $\fp$, $\bH$ and $\Gamma_t$ are unknow,
while $\bv_0$, $\bH_0$ and $\rho_0$ are prescribed $N$-component vectors. 
As for the remaining symbols,  
$\bT(\bv, \fp) = \nu_\pm\bD(\bv_\pm)-\fp_\pm\bI$ is the viscous stress
tensor, $\bD(\bv_\pm) = \nabla \bv_\pm
+ (\nabla\bv_\pm)^\top$ is the doubled deformation tensor whose
$(i, j)$th component is $\pd_j v_{\pm i} + \pd_i v_{\pm j}$ with 
$\pd_i = \pd/\pd x_i$, $\bI$ the $N\times N$ unit matrix, 
$\bT_M(\bH) = \bT_M(\bH_\pm)= \mu_\pm(\bH_\pm\otimes\bH_\pm
- \frac12|\bH_\pm|^2\bI)$
the magnetic stress tensor, $\curl \bv = \curl \bv_\pm
= \nabla \bv_\pm - (\nabla \bv_\pm)^\top$ the doubled rotation
tensor whose $(i,j)$th component is $\pd_j v_{\pm i} - \pd_i v_{\pm j}$,
$V_{\Gamma_t}$ the velocity of the evolution of $\Gamma_t$ in the direction of 
$\bn_t$, which is given by $V_{\Gamma_t} = (\pd_t\rho)\bn\cdot\bn_t$ 
in the case of \eqref{free-surface:1}, and $\CH(\Gamma_t)$  $N-1$ fold  mean curvature of 
$\Gamma_t$ that is given by $\CH(\Gamma_t) \bn_t = \Delta_{\Gamma_t}x$
for $x
\in \Gamma_t$ with 
the Laplace Beltrami operator $\Delta_{\Gamma_t}$
on $\Gamma_t$, and $\fp_0$ the outside pressure. 
Moreover,  
$\rho = \rho_\pm$, $\mu=\mu_\pm$, $\nu=\nu_\pm$, and $\alpha = \alpha_\pm$, 
and $\rho_\pm$, $\mu_\pm$, $\nu_\pm$, and $\alpha_\pm$ are positive
constants describing respective the mass density, 
the magnetic permability,  the kinematic 
viscosity, and conductivity..  And, 
$\sigma$ is a positive constant describing the
coefficient of the surface tension.  Finally, for any
matrix field $\bK$ with $(i, j)$th component 
$K_{ij}$, the quantity $\DV K$ is an $N$-vector of functions
with the $i$th  component $\sum_{j=1}^N \pd_jK_{ij}$, 
and for any $N$-vectors of functions $\bu = (u_1, \ldots,
u_N)^\top$ and $\bw = (w_1,\ldots, w_N)^\top$,  
$\dv\bu = \sum_{j=1}^N\pd_ju_j$ and 
$\bu\cdot\nabla\bw$ is an $N$-vector of functions
with  the $i$th component  $\sum_{j=1}^N u_j\pd_jw_i$,
and $\bu\otimes\bw$ an $N\times N$ matrix with the $(i, j)$th component
 $u_iw_j$. We notice that  
\begin{equation}\label{lap:1}\begin{split}
&\Delta \bv = -\DV\curl \bv + \nabla\dv \bv, 
\quad 
\DV(\bv\otimes\bH- \bH\otimes\bv)
= \bv\dv\bH - \bH\dv \bv + \bH\cdot\nabla\bv - \bv\cdot\nabla\bH, \\
&\rot\times\rot \bH = \DV\curl\bH, \quad
\rot\times\bv\times\bH= \DV(\bv\otimes\bH - \bH\otimes\bv)  
\quad (\text{3 dimensional case}),
\end{split}\end{equation}
where $\times$ is the exterior product in the three dimensional case.
In particular, in the three dimensional case, the set of equations 
for the  magnetic field in Eq. \eqref{mhd.1} are written by
\begin{alignat*}
\mu\pd_t\bH + \rot\times(\alpha^{-1}\rot\bH - \mu\bv\times\bH)
=0, \quad\dv \bH = 0& &\quad &\text{in $\bigcup_{0 < t < T}\, \dot\Omega_t \times\{t\}$}, \nonumber \\ 
[[\bn_t\times \{\alpha^{-1}\rot\bH - \mu\bv\times\bH)\}]]  = 0, \quad 
[[\mu\bH\cdot\bn_t]]=0,
\quad [[\bH-<\bH, \bn_t>\bn_t]]=0&
&\quad &\text{on $\bigcup_{0 < t < T}\, \Gamma_t \times\{t\}$}. 
\end{alignat*}
This is a standard description, and so the set of equations for the magnetic field in Eq. \eqref{mhd.1} is 
the $N$-dimensional mathematical description for the magnetic equations with transmission 
conditions. 

In the equilibrium state, $\bv=0$, $\bH=0$, $\Gamma_t=\Gamma$, and $\fp$ is a constant state,
and so we assume that
\begin{equation}\label{initial:0}
\fp_0 = \sigma H(\Gamma).
\end{equation}

In Eq. \eqref{mhd.1}, there is one equation for the magnetic fields
$\bH_\pm$ too many, so that in the following instead of \eqref{mhd.1}, 
we consider the following equations:
\begin{alignat}2
\rho(\pd_t\bv + \bv\cdot\nabla\bv) - \DV(\bT(\bv, \fp) + \bT_M(\bH))  = 0,
\quad \dv\bv= 0&
&\quad &\text{in $\bigcup_{0 < t < T}\, \dot\Omega_t \times\{t\}$},
\nonumber \\
[[(\bT(\bv, \fp)+ \bT_M(\bH))\bn_t]] = \sigma H(\Gamma_t)\bn_t - \fp_0\bn_t,
\quad [[\bv]]=0, \quad V_{\Gamma_t}  = \bv_+\cdot\bn& 
&\quad &\text{on $\bigcup_{0 < t < T}\, \Gamma_t \times\{t\}$}, 
\nonumber\\
\mu\pd_t\bH - \alpha^{-1}\Delta\bH -\DV\mu
(\bv\otimes\bH - \bH\otimes\bv)  = 0&
&\quad &\text{in $\bigcup_{0 < t < T}\, \dot\Omega_t \times\{t\}$}, 
\nonumber\\
[[\{\alpha^{-1}\curl\bH - \mu(\bv\otimes\bH - \bH\otimes\bv)\}
\bn_t]]  = 0, \quad [[\mu\dv\bH]]  =0&
&\quad &\text{on $\bigcup_{0 < t < T}\, \Gamma_t \times\{t\}$}, 
\nonumber\\
[[\mu\bH\cdot\bn_t]]=0,
\quad [[\bH-<\bH, \bn_t>\bn_t]]=0&
&\quad &\text{on $\bigcup_{0 < t < T}\, \Gamma_t \times\{t\}$}, 
\nonumber\\
\bv_\pm=0,\quad \bn_0\cdot\bH_\pm =0, \quad
(\curl\bH_\pm)\bn_\pm = 0&
&\quad&\text{on $S_\pm\times(0, T)$}, 
\nonumber\\
(\bv, \bH)|_{t=0}  = (\bv_0, \bH_0)&
&\quad&\text{in $\dot\Omega$}.
\label{mhd.2}
\end{alignat}
Namely, two equations: $\dv \bH_\pm = 0$ in $\Omega_\pm$ is replaced with 
one boundary condition:
 $[[\mu\dv\bH]]=0$ on $\Gamma$. 
 Frolova and Shibata \cite{FS1} proved that in equations \eqref{mhd.2}
 if $\dv\bH= 0$ initially, then $\dv \bH=0$ in $\dot\Omega$ 
follows automatically for any $t > 0$ as long as solutions
exist.  Thus, the local wellposedness of equations \eqref{mhd.1} 
follows from that of equations \eqref{mhd.2} provided that
the initial  data $\bH_0$ satisfy the divergence condition:
$\dv \bH_0 = 0$, which is a compatibility condition. 
This paper devotes to proving the local wellposedness of equations
\eqref{mhd.2} in the maximal $L_p$-$L_q$ regularity framework
under the assumption that $\rho_0$ is small enough,
that is the interface $\Gamma_t$ is very close to the reference interface
$\Gamma$ initially.

 Since $\dot\Omega_t$ and $\Gamma_t$ are unknown,   
the set of equations in \eqref{mhd.2} is transformed to that of equations in 
$\dot\Omega$ and $\Gamma$ by the Hanzawa transform generated
by $\rho$ (cf. Subsect. \ref{sec:new2.0} below),   and then the main result are 
stated in Subsect. \ref{sec:new2.4} below. 
%%%%%%%%%%%%%
%%%%%%%%%%%%%%
%%%%%%%%%%%%%

The equations of magnetodydrodynamics (MHD) can be found 
in \cite{C, LL}.  The solvability of MHD equations was first 
obtained by \cite{LS}. 
The free boundary problem for MHD was first studied
by Padula and Solonnikov \cite{PS} 
in the case where $\Omega_{+t}$ is a vacuume region in
the three dimensional Euclidean space $\BR^3$. They proved
the local well-posedness in the $L_2$ framework and 
used Sobolve-Slobodetskii spaces of fractional order. 
Later on, the global well-posedness was proved by
Froloba \cite{Frolova1} and Solonnikov and Frolova \cite{SE1}.
Moreover, the $L_p$ approach to the same problem was done
by Solonnikov \cite{Sol14-1, Sol14-2}.
 %In this paper, both of $\Omega_{+t}$ and $\Omega_{-t}$ are
%filled by  two different incompressible viscous fluids and two 
%different electromagnetic fields are generated in there.
%We prove the local well-posedness in the maximal $L_p$-$L_q$ regularity classes.
 In \cite{PS}, by some technical reason,  it was 
required that regularity class of the fluid is slightly higher
	than that of the magnetic field (cf. \cite[p.331]{PS}).
	But, in this paper, we do not need this assumption, 
	that is we can solve the problem in the same regularity
	classes for the fluid and magnetic field. 
	The different point of this paper than in \cite{PS} appears
	in the iteration scheme (cf. \eqref{eq:4.1} and  \eqref{eq:4.2}).
	
	As a related topics,  in \cite{K1, K2} and references
	therein Kacprzyk proved the local and global 
	well-posedness of free boundary problem for
	the viscous non-homogeneous incompressible MHD
	in the case where an incompressible fluid is occupied 
	in a domain $\Omega_{-t}$ bounded by a free surface $\Gamma_t$ subjected to an
	electromagnetic field generated in a domain $\Omega_{+t}$
	exterior to $\Omega_{-t}$ by some currents located on a fixed
	 bounary $S_+$ of $\Omega_{+t}$. In \cite{K1, K2},  
	 it is assumed that $S_- =\emptyset$.
	 On the free surface, $\Gamma_t$, free boundary condition
	 without surface tension for the viscous fluid part and transmission 
	 conditions for electromagnetic fields part are imposed.  Since
	 the surface tension is not taken into account, the Lagrange
	 transformation was applied, and so the viscous fluid part has one
	 regularity higher than the electromagnetic fillds part.
	 An $L_2$ approach is applied and Sobolev-Slobodetskii
	 spaces of fractiona order are also used in \cite{K1, K2}. 
	
%%%%%%%%%%%%% further notation %%%%%%%%%%%%%%%
Finally, we explain some symbols used throughout the paper.
We denote the set of all natural numbers, real numbers, and
 complex numbers by  $\BN$, $\BR$, and $\BC$,
respectively.  Set $\BN_0 = \BN \cup\{0\}$.
For any multi-index $\kappa
= (\kappa_, \ldots, \kappa_N)$, $\kappa_j \in \BN_0$, we set $\pd_x^\kappa
= \pd_1^{\kappa_1}\cdots\pd_N^{\kappa_N}$ ,
$|\kappa| = \kappa_1 + \cdots+ \kappa_N$.
For scalar $f$, and $N$-vector of functions,
$\bg = (g_1, \ldots, g_N)$, we set $\nabla^n f
=(\pd_x^\kappa f\mid |\kappa|=n)$ and
$\nabla^n\bg = (\pd_x^\kappa g_j \mid |\kappa|=n, j=1, \ldots, N)$.
In particular, $\nabla^0f=f$, $\nabla^0\bg=\bg$, $\nabla^1f=\nabla f$,
and $\nabla^1\bg=\nabla\bg$.
For $1 \leq q \leq \infty$, $m \in \BN$, $s \in \BR$,
and any domain $D \subset \BR^N$, we denote by $L_q(D)$, $H^m_q(D)$, and
$B^s_{q,p}(D)$  the standard Lebesgue, Sobolev,
and Besov spaces, respectively, while 
 $\|\cdot\|_{L_q(D)}$, $\|\cdot\|_{H^m_q(D)}$,
and $\|\cdot\|_{B^s_{q,p}(D)}$ denote the norms of these spaces. 
We write $W^s_q(D) = B^s_{q,q}(D)$ and $H^0_q(D)=L_q(D)$.
What $f= f_\pm$ means that
$f(x) = f_\pm(x)$ for $x \in \Omega_\pm$. 
 For $\CH \in \{H^m_q, B^s_{q,p}\}$, 
the function spaces $\CH(\dot \Omega)$ ($\dot \Omega = \Omega_+ \cup \Omega_-$) 
and their norms   are  defined by setting
$$\CH(\dot \Omega) = \{f = f_\pm \mid f_\pm \in \CH(\Omega_\pm)\},
\quad \|f\|_{\CH(\dot D)} = \|f_+\|_{\CH(D_+)} +
\|f_-\|_{\CH(D_-)}. 
$$

For any Banach space $X$ with the
 norm $\|\cdot\|_X$,
$X^d$ denotes the $d$ product space defined by
$\{x = (x_1, \ldots, x_d) \mid x_i \in X\}$, while
 the norm of $X^d$ is simply written by $\|\cdot\|_X$, that is
$\|x\|_X = \sum_{j=1}^d\|x_j\|_X$.
For any time interval $(a, b)$,
$L_p((a, b), X)$ and $H^m_p((a, b), X)$ denote respective the standard
$X$-valued Lebesgue space and $X$-valued Sobolev space,
while $\|\cdot\|_{L_p((a, b), X)}$ and
$\|\cdot\|_{H^m_p((a, b), X)}$ denote their norms. 
Let $\CF$ and $\CF^{-1}$ be respective the Fourier transform
and the Fourier inverse transform. Let $H^s_p(\BR, X)$, $s>0$,
be the Bessel potential space of order $s$ defined
by
\begin{align*}
H^s_p(\BR, X) &= \{f \in L_p(\BR, X) \mid \|f\|_{H^s_p(\BR, X)}=
\|\CF^{-1}[(1+|\tau|^2)^{1/4}\CF[f](\tau)]\|_{L_p(\BR, X)} < \infty\}.
\end{align*}

For any $N$-vector of functions,
$\bu={}^\top(u_1, \ldots, u_N)$, sometimes $\nabla\bu$ is regarded as
an $N\times N$-matrix of functions whose $(i, j)$th component is
$\pd_ju_i$.
For any $m$-vector $V=(v_1, \ldots, v_m)$ and $n$-vector
$W=(w_1, \ldots, w_n)$, $V\otimes W$ denotes an $m\times n$ matrix 
whose $(i, j)$th component is $V_iW_j$.  For any 
$(mn\times N)$-matrix $A=(A_{ij, k} \mid i=1, \ldots, m, j=1, \ldots, n,
k=1, \ldots, N)$, 
$AV\otimes W$ denotes an $N$-column vector whose $k^{\rm th}$ component is
the quantity: $\sum_{j=1}^m\sum_{j=1}^n A_{ij, k}v_iw_j$.  Moreover, we define 
$AV\otimes W \otimes Z = (AV\otimes W)\otimes Z$. Inductively, we 
define $AV_1\otimes \cdots \otimes V_n$ by setting $AV_1\otimes \cdots \otimes V_n
=(AV_1\otimes\cdots\otimes V_{n-1})\otimes V_n$ for $n \geq 4$.

Let $\ba\cdot \bb =<\ba, \bb>= \sum_{j=1}^Na_jb_j$
for any $N$-vectors $\ba=(a_1, \ldots, a_N)$  and $\bb
=(b_1, \ldots, b_N)$.  
For any $N$-vector $\ba$, let $\ba_\tau
: = \ba - <\ba, \bn>\bn$.  For any two $N\times N$-matrices 
$\bA=(A_{ij})$ and $\bB=(B_{ij})$, the quantity $\bA:\bB$ is defined by
$\bA:\bB 
= \sum_{i,j=1}^NA_{ij}B_{ji}$. 
For any domain $G$ with boundary $\pd G$, 
we set
$$(\bu, \bv)_G = \int_G\bu(x)\cdot\overline{\bv(x)}\,dx, 
\quad (\bu, \bv)_{\pd G} = \int_{\pd G}\bu\cdot\overline{\bv(x)}\,d\sigma,
$$
where $\overline{\bv(x)}$ is the complex conjugate of $\bv(x)$ and 
$d\sigma$ denotes the surface element of $\pd G$. 
Given  
$1 < q < \infty$, let $q' = q/(q-1)$. Throughout the paper, the letter $C$ denotes
generic constants and   $C_{a,b, \cdots}$ 
the constant  which depends on $a$, $b$, $\cdots$.
 The values of constants $C$, $C_{a,b,\cdots}$
may be changed from line to line.

%%%%%%%%%%%%%%%%%%%
%%%%%%%%%%%%%%%%%%%%
\section{Hanzawa transform and statment of main result}\label{sec:new2}

\subsection{Hanzawa transform}\label{sec:new2.0}
Let $\bn$ be the unit normal to $\Gamma$ oriented from $\Omega_+$ into $\Omega_-$.
Since $\Gamma_t$ is unknown, we assume that 
the $\Gamma_t$ is represented by \eqref{free-surface:1} 
Our task is to find not only $\bv$, $\fp$ and $\bH$ but also $h$.
We know  the existence of 
an $N$-vector, $\tilde\bn$, of $C^2$ functions  defined on 
$\BR^N$ such that 
\begin{equation}\label{norm:1}\begin{split}
&\bn = \tilde\bn \quad\text{on $\Gamma$}, \quad 
{\rm  supp}\, \tilde\bn \subset U_\Gamma, 
\quad 
\|\tilde \bn\|_{H^2_\infty(\BR^N)} \leq C
\end{split}\end{equation}
with some constant $C$, 
where we have set $U_\Gamma = \bigcup_{x_0 \in \Gamma} \{x \in \BR^N \mid |x-x_0| < \alpha\}$
with some small constant $\alpha > 0$. 
We will construct $\tilde\bn$ in Subsec \ref{sec:new2.2} below. 
We may assume that  
$${\rm dist}\,({\rm supp}\, \tilde\bn,  S_\pm) \geq 3d_\pm/4.$$
In the following we write $\dot\Omega = \Omega_+ \cup \Omega_-$ and 
$\Omega = \dot\Omega \cup\Gamma$. 
Let $H_h$ be an extension function of $h$ such that
$h = H_h$ on $\Gamma$.   In fact, we take $H_h$ as a solution of the 
harmonic equation:
\begin{equation}\label{eq:harmonic}
(-\Delta + \lambda_0)H_h = 0 \quad\text{in $\dot\Omega$},
\quad H_h|_\Gamma = h
\end{equation}
with some large positive number $\lambda_0$ which guarantees the unique solvability
of \eqref{eq:harmonic}.  In this case, if $h$ satisfies the regularity condition:
\begin{equation}\label{height:1}
h \in H^1_p((0, T), W^{2-1/q}_q(\Gamma)) \cap L_p((0, T), W^{3-1/q}_q(\Gamma)),
\end{equation}
then $H_h$ satisfies the regularity condition:
\begin{equation}\label{height:2}H_h \in H^1_p((0, T), H^2_q(\Omega)) \cap L_p((0, T), H^3_q(\Omega)),
\end{equation}
and possesses the estimate:
\begin{equation}\label{1.3.2}\begin{aligned}
\|\pd_t^iH_h\|_{L_p((0, T), H^{3-i}_q(\Omega))}
\leq C\|\pd_t^ih\|_{L_p((0, T), W^{3-i-1/q}_q(\Gamma))} \quad (i=0,1), \\
\|\pd_t^ih\|_{L_p((0, T), W^{3-1/q-i}_q(\Gamma))}
\leq C\|\pd_t^iH_h\|_{L_p((0, T), H^{3-i}_q(\Omega))} \quad (i=0,1)
\end{aligned}\end{equation}
for some constant $C > 0$. 

To transform Eq. \eqref{mhd.2}
to the equations on $\Omega$, we use  Hanazawa transformation
defined by 
\begin{equation}\label{hanzawa:1}
x = y + H_h(t, y)\tilde\bn(y) :=\Xi_h(t,y). 
\end{equation}
Let $\delta > 0$ be a small number such that 
\begin{equation}\label{hanzawa:2}
|\Xi_h(y_1, t) - \Xi_h(y_2, t)| \leq (1/2)|y_1-y_2|
\end{equation}
provided that
\begin{equation}\label{assump:0}
\sup_{0 < t < T}\|\bar\nabla H_h(\cdot, t)\|_{L_\infty(\Omega)} \leq \delta.
\end{equation}
Here and in the following, we write $\bar\nabla H_h = (\pd_x^\alpha H_h \mid |\alpha| \leq 1)
=(H_h, \nabla H_h)$.
From \eqref{hanzawa:1}, the map $x = \Xi(y, t)$ is 
injective.  And also, under suitable regularity condition on
$H_h$, for example, $H_h \in C^{1+\alpha}$ for each $t \in (0, T)$ with
some small $\alpha > 0$, the map $x = \Xi_h(y, t)$ becomes an open and 
closed map, so that $\{x = \Xi_h(y, t) \mid y \in \Omega\}
= \Omega$ because $x = \Xi_h(y, t)$ is an identity map on 
$\Omega\setminus U_\Gamma$.  We assume that the initial surface $\Gamma_0$ is 
given by 
$$\Gamma_0 = \{x = y + h_0(y)\bn \mid y \in \Gamma\}
$$
with a  given small function $h_0$.  Let $H_{h_0}$ be an extension of $h_0$
which is given by a unique solution of equation \eqref{eq:harmonic},
where $H_h$ and $h$ are replaced with $H_{h_0}$ and $h_0$, respectively. 

Let $\Xi_{h_0} = y + H_{h_0}\tilde\bn$, and set 
\begin{gather}
\bu(y, t) = \bv(\Xi^{-1}_h(y, t), t), \quad 
\fq(y, t) = \fp(\Xi_h^{-1}(y, t), t), \quad 
\bG(y, t) = \bH(\Xi_h^{-1}(y, t), t), \nonumber \\
\Gamma_t = \{x = \Xi_h(y, t) \mid y \in \Gamma\},
\quad
\Omega_{t\pm} = \{x = \Xi_h(y, t) \mid y \in \Omega_\pm\},
\nonumber \\
\bu_0(y) = \bv_0(\Xi^{-1}_{h_0}(y)), \quad 
\bG_0(y) = \bH_0(\Xi_{h_0}^{-1}(y)).
\label{change:1}
\end{gather}
%%%%%%%%%%%%%%%%%%
%%%%%%%%%%%%%%%%%%
Noting that $x = y$ near $S_\pm$, 
we have 
\begin{gather*}
\bu_\pm=0,\quad \bn_\pm\cdot\bG_\pm =0, \quad
(\curl\bG_\pm)\bn_\pm = 0
\quad\text{on $S_\pm\times(0, T)$}, \\
(\bu, \bG, h)|_{t=0}  = (\bv_0, \bH_0, h_0)
\quad\text{in $\dot\Omega\times \Gamma$}, \quad H_h|_{t=0}=h_0
\quad \text{on $\Gamma$}.
\end{gather*}

In what follows, we derive equations and interface conditions which
$\bu$, $\fq$ and $\bG$ satisfy
in $\Omega_{t\pm}$ and on $\Gamma_t$.

\subsection{Derivation of equations}\label{sec:new2.1}

%%%%%%%%%%%%%%%%%%%%%%%%%%
In this subsection, we derive  equations obtained by Hanzawa transformation:
$x = y +H_h(y, t)\tilde\bn(y)$ from 
 the first, second and third equations in
Eq. \eqref{mhd.2}. 
  We assume that 
$H_h$ satisifies \eqref{assump:0}  
with small positive number $\delta > 0$.
We have  
$$\frac{\pd x}{\pd y} = \bI + \frac{\pd (H_h\tilde\bn)}{\pd y}
$$
and then, choosing $\delta > 0$ in \eqref{assump:0} small enough, we see that
there exists an $N\times N$ matrix, $V_0(\bK)$, of bounded real nalytic functions defined
on $U_\delta  = \{\bK \in \BR^{N+1} \mid 
|\bK| \leq \delta \}$ with $V_0(0) = 0$ such that 

\begin{equation}\label{trans:1}
\frac{\pd y}{\pd x} = \Bigl(\frac{\pd x}{\pd y}\Bigr)^{-1}
= \bI + \bV_0(\bar\nabla H_h).
\end{equation}
Here and in the following, $\bK = (\kappa_0,
\kappa_1, \ldots, \kappa_N)$ and $\kappa_0$, $\kappa_1, \ldots, \kappa_N$
are  independent variables corresponding to
$H_h$, $\pd H_h/\pd y_1, \ldots, \pd H_h/\pd y_N$, respectively. 
Let $V_{0ij}(\bK)$ be the $(i, j)$ th component of $V_0(\bK)$, and 
then by the chain rule  
\begin{equation}\label{trans:2}
\frac{\pd }{\pd x_j} = \sum_{k=1}^N (\delta_{jk} + V_{0jk}(\bK))
\frac{\pd}{\pd y_k},
\quad\nabla_x = (\bI + \bV_0(\bK))\nabla_y.
\end{equation}
Since $V_{0jk}(0) = 0$, we may write $V_{0jk}(\bK) = \tilde V_{0jk}(\bK)\bK$ with 
$$\tilde V_{0jk}(\bK)\bK= \int^1_0\frac{d}{d\theta}(V_{0jk}(\theta \bK))\,d\theta.$$
In particular, 
\begin{equation}\label{dc:1}\begin{split}
\curl_{ij}(\bv) &= \frac{\pd v_i}{\pd x_j} - \frac{\pd v_j}{\pd x_i}
= \curl_{ij}(\bu) + V_{Cij}(\bK)\nabla\bu, \\
D_{ij}(\bv) &= \frac{\pd v_i}{\pd x_j} + \frac{\pd v_j}{\pd x_i}
=D_{ij}(\bu) + V_{Dij}(\bK)\nabla\bu, \\
\end{split}\end{equation}
with
\begin{align*}
V_{Cij}(\bK)\nabla\bu &= \sum_{k=1}^N(V_{0jk}(\bK)\frac{\pd u_i}{\pd y_j}
- V_{0ik}(\bK)\frac{\pd u_j}{\pd y_k}), \\
V_{Dij}(\bK)\nabla\bu &= \sum_{k=1}^N(V_{0jk}(\bK)\frac{\pd u_i}{\pd y_j}
+ V_{0ik}(\bK)\frac{\pd u_j}{\pd y_k}).
\end{align*}
Here and in the following, for an $N\times N$ matrix $A$, $A_{ij}$ denotes
its $(i, j)$ th component and $(A_{ij})$ denotes an $N\times N$ matrix
whose $(i, j)$ th component is $A_{ij}$. 
To obtain the first equation in \eqref{mhd.3} in Subsec \ref{sec:new2.4} below, we make the 
pressure term linear.  From 
$\nabla\fp = (\bI + \bV_0(\bK))\nabla \fq$ it follows that  
$$\frac{\pd \fq}{\pd y_j} = \sum_{k=1}^N (\delta_{jk} + \frac{\pd x_k}{\pd y_j})
\frac{\pd \fp}{\pd x_k}.$$
Let $\tilde\bn = (\tilde n_1, \ldots \tilde n_N)^\top$, and then
\begin{equation}\label{time:1}\frac{\pd v_i}{\pd t} = \frac{\pd}{\pd t}u_i(y + H_h\tilde\bn, t)
= \frac{\pd u_i}{\pd t} + \sum_{j=1}^N\frac{\pd u_i}{\pd y_j}\frac{\pd H_h}{\pd t}\tilde n_j.
\end{equation}
Thus, 
the first equation in \eqref{mhd.2} is transformed to 
\begin{align*}
\frac{\pd \fq}{\pd y_m} 
&=-\rho\sum_{i=1}^N(\delta_{mi} + \frac{\pd x_m}{\pd y_i})
\{\frac{\pd u_i}{\pd t} + \sum_{j=1}^N\frac{\pd u_i}{\pd y_j}\frac{\pd H_h}{\pd t}\tilde n_j
+ \sum_{j,k=1}^N u_j(\delta_{jk} + V_{0jk}(\bar\nabla H_h))\frac{\pd u_i}{\pd y_k}\} \\
&+ 
\sum_{i,j,k=1}^N(\delta_{mi} + \frac{\pd x_m}{\pd y_i})
(\delta_{jk} + V_{0jk}(\bar\nabla H_h))
\frac{\pd}{\pd y_k}\{\nu(D_{ij}(\bu) 
+ V_{Dij}(\bar\nabla H_h)\nabla\bu) + 
T_{Mij}(\bG)\}\\
&= -\rho\pd_t u_m + \sum_{k=1}^N \frac{\pd}{\pd y_k}(\nu D_{mk}(\bu))
+ f_{1m}(\bu, \bG, H_h)
\end{align*}
with
\begin{equation}\label{f:1}\begin{aligned}
&f_{1m}(\bu, \bG, H_h) = -\rho\Bigl(\sum_{j=1}^N\frac{\pd u_m}{\pd y_j}
\frac{\pd H_h}{\pd t}\tilde n_j + \sum_{j,k=1}^Nu_j(\delta_{jk} + V_{0jk}(\bar\nabla H_h))
\frac{\pd u_m}{\pd y_k}\Bigr) \\
&\quad - \rho\sum_{i=1}^N\frac{\pd(H_h\tilde n_m)}{\pd y_i}
\Bigl(\frac{\pd u_i}{\pd t}
+ \sum_{j=1}^N\frac{\pd u_i}{\pd y_j}
\frac{\pd H_h}{\pd t}\tilde n_j + \sum_{j,k=1}^Nu_j(\delta_{jk} + V_{0jk}(\bar\nabla H_h))
\frac{\pd u_i}{\pd y_k}\Bigr)\\
&\quad + \sum_{j,k=1}^N(\delta_{jk} + V_{0jk}(\bar\nabla H_h))
\frac{\pd}{\pd y_k}(\nu(D_{mj}(\bu)
+ V_{D_{mj}}(\bar\nabla H_h)\nabla\bu) + T_{Mmj}(\bG)) \\
&\quad +\sum_{i,j,k=1}^N\frac{\pd(H_h\tilde n_m)}{\pd y_i}
(\delta_{jk} + V_{0jk}(\bar\nabla H_h))
\frac{\pd}{\pd y_k}(\nu(D_{ij}(\bu)
+ V_{D_{ij}}(\bar\nabla H_h)\nabla\bu) + T_{Mij}(\bG))
\end{aligned}\end{equation} 
Thus, setting $\bff_1(\bu, \bG, H_h) = (f_{11}(\bu, \bG, H_h), 
\ldots, f_{1N}(\bu, \bG, H_h))^\top$, we have
\begin{equation}\label{maineq:1}
\rho\pd_t\bu - \DV\bT(\bu, \fq) = \bff_1(\bu, \bG, H_h)
\quad\text{in $\dot\Omega \times(0, T)$}.
\end{equation}
Since $V_{0jk}(0) = 0$ and $V_{\bD ij}(0) = 0$, we may write
\begin{equation}\label{maineq:1*}\begin{aligned}
\bff_1(\bu, \bG, H_h) & = f^1_0\bar\nabla H_h\otimes\pd_t\bu + \CF^1_0(\bar\nabla H_h)
\pd_tH_h\otimes\nabla\bu + \CF^1_1(\bar\nabla H_h)\bu\otimes\nabla\bu\\
&+ \CF^1_2(\bar\nabla H_h)\bar\nabla H_h \otimes\nabla^2\bu
+ \CF^1_3(\bar\nabla H_h)\bar\nabla^2 H_h\otimes\nabla \bu
 + \CF^1_4(\bar\nabla H_h)\bG\otimes\nabla\bG
\end{aligned}\end{equation}
where $f^1_0$ is a bounded function and $\CF^1_j(\bK)$ are some matrices
of bounded analytic
functions defined on $U_\delta$.  Here and in the following, we write
$\bar\nabla^k H_h = (\pd^\alpha_y H_h \mid |\alpha| \leq k)$ for $k \geq 2$
and $\bar\nabla H_h = (\pd^\alpha_y H_h \mid |\alpha| \leq 1)$. 

We next consider the divergence free condition: $\dv\bv=0$. 
By \eqref{trans:2}, 
\begin{equation}\label{div:1}
\dv\bv = \sum_{j=1}^N\frac{\pd v_j}{\pd x_j}
= \sum_{j,k=1}^N(\delta_{jk} + V_{0jk}(\bar\nabla H_h))\frac{\pd u_{ j}}{\pd y_k}.
\end{equation} 
Let $J = \det(\pd x/\pd y)$ and then, choosing $\delta > 0$ small enough
in \eqref{assump:0},  we can write
\begin{equation}\label{jacob:1}
J = 1 + J_0(\bar\nabla H_h), 
\end{equation}
where $J_0(\bK)$ is a real analytic functions
defined on $U_\delta$ such that $J_0(0) = 0$. Using this symbol, we have
\begin{align*}
(\dv_x\bv_\pm, \varphi)_{\Omega_{t\pm}}
&= -(\bv_\pm, \nabla_x\varphi)_{\Omega_{t\pm}}
= -\sum_{j=1}^N(Ju_{\pm,j}, \sum_{k=1}^N(\delta_{jk} 
+ V_{0jk}(\bar\nabla H_h))\frac{\pd\varphi}{\pd y_k})_{\Omega_\pm}
\\
&= (\sum_{j,k=1}^N\frac{\pd}{\pd y_k}\{J(\delta_{jk} + 
V_{0jk}(\bar\nabla H_h))u_{\pm j}\},\varphi)_{\Omega_\pm},
\end{align*}
so that
\begin{equation}\label{div:2}
\dv \bv_\pm = J^{-1}\sum_{j,k=1}^N\frac{\pd}{\pd y_k}\{J(\delta_{jk} + 
V_{0jk}(\bar\nabla H_h))u_{\pm j}\}.
\end{equation}
Combining \eqref{div:1}, \eqref{div:2}, and \eqref{jacob:1}  yields that
\begin{equation}\label{div:3}
\dv \bu  = g(\bu, H_h) = \dv \bg(\bu, H_h)
\quad\text{in $\dot\Omega\times(0, T)$}
\end{equation}
with 
\begin{equation}\label{g:1}\begin{split}
g(\bu, H_h) &= \sum_{j,k=1}^N V_{0jk}(\bar\nabla H_h)
\frac{\pd u_{\pm j}}{\pd y_k} + 
J_0(\bar\nabla H_h)\{\dv\bu + \sum_{j,k=1}^N V_{0jk}(\bar\nabla H_h)
\frac{\pd u_{j}}{\pd y_k}\}, \\
\bg(\bu, H_h)|_k &= \sum_{j=1}^NV_{0jk}(\bar\nabla H_h)u_{ j}
+ J_0(\bar\nabla H_h)\sum_{j=1}^N(\delta_{jk} + V_{0jk}(\bar\nabla H_h))u_{ j}.
\end{split}\end{equation}
Since $V_{0jk}(0) = J_0(0) = 0$, we may write
\begin{equation}\label{div:3*}
g(\bu, H_h) = \CG_1(\bar\nabla H_h) \bar\nabla H_h \otimes \nabla \bu, 
\quad \bg(\bu, H_h) = \CG_2(\bar\nabla H_h) \bar\nabla H_h \otimes \bu, 
\end{equation}
where $\CG_i(\bK)$ are some matrices of bounded analytic functions defined on $U_\delta$. 

We next consider the third  equation in Eq. \eqref{mhd.2}.
By \eqref{time:1}, 
$$\mu\pd_t\bH = \mu\pd_t\bG + \mu \sum_{j,k=1}^N
\tilde n_j\frac{\pd \bG}{\pd y_j}\frac{\pd H_h}{\pd t}.
$$
Moreover, 
\begin{align*}
\Delta &
= \sum_{j=1}^N(\sum_{k=1}^N(\delta_{jk}+ V_{0jk}(\bar\nabla H_h))\frac{\pd}{\pd y_k})
(\sum_{\ell=1}^N(\delta_{j\ell} + V_{0j\ell}(\bar\nabla H_h))\frac{\pd}{\pd y_\ell})\\
& = \sum_{j=1}^N\{\frac{\pd^2}{\pd y_j^2}
+\sum_{\ell=1}^N\frac{\pd}{\pd y_j}(V_{0j\ell}(\bar\nabla H_h)\frac{\pd}{\pd y_\ell})
+ \sum_{\ell, k=1}^NV_{0jk}(\bar\nabla H_h)\frac{\pd}{\pd y_k}
((\delta_{j\ell} + V_{0j\ell}(\bar\nabla H_h))\frac{\pd}{\pd y_\ell})\\
& = \Delta + V_{\Delta 2}(\bar\nabla H_h)\nabla^2 + V_{\Delta 1}(\bar\nabla H_h)\nabla
\end{align*}
with
\begin{align*}
V_{\Delta 2}(\bar\nabla H_h)\nabla^2 &= 2\sum_{j,k=1}^NV_{0jk}(\bar\nabla H_h)
\frac{\pd^2}{\pd y_j\pd y_k}
+ \sum_{j,k,\ell=1}^N V_{0jk}(\bar\nabla H_h)V_{0j\ell}(\bar\nabla H_h)
\frac{\pd^2}{\pd y_k\pd y_\ell}, \\
V_{\Delta 1}(\bar\nabla H_h)\nabla &= \sum_{j,k=1}^N \frac{\pd V_{0j\ell}(\bar\nabla H_h)}
{\pd y_j}\frac{\pd}{\pd y_k}
+ \sum_{j,k,\ell=1}^NV_{0jk}(\bar\nabla H_h)\frac{\pd V_{0j\ell}(\bar\nabla H_h)}{\pd y_k}
\frac{\pd}{\pd y_\ell}.
\end{align*}
Thus, setting 
\begin{equation}\label{g:1*}\begin{split}
\bff_2(\bu, \bG, H_h)
 &=- \mu\sum_{j,k=1}^N\tilde n_j
\frac{\pd\bG}{\pd y_j}\frac{\pd H_h}{\pd t} 
 + \alpha^{-1}V_{\Delta 2}(\bar\nabla H_h)\nabla^2\bG_m
 + \alpha^{-1}V_{\Delta 1}(\bar\nabla H_h)\nabla \bG\\
& + \mu\sum_{j,k=1}^N(\delta_{jk} + V_{0jk}(\bK))
\frac{\pd}{\pd y_k}(\bu\otimes\bG - \bG\otimes\bu), 
\end{split}\end{equation}
we have 
\begin{equation}\label{maineq:2}
\mu\pd_t\bG - \alpha^{-1}\Delta\bG = \bff_2(\bu, \bG, H_h)
\quad\text{in $\dot\Omega\times(0, T)$}.
\end{equation}
Since $V_{0jk}(0) = 0$,  we may write
\begin{equation}\label{maineq:2*}\begin{aligned}
\bff_2(\bu, \bG, H_h) & = f_2\nabla\bG\otimes\pd_tH_h + \CF^2_1(\bar\nabla H_h)\bar\nabla H_h\otimes\nabla^2\bG
+ \CF^2_2(\bar\nabla H_h)\bar\nabla^2 H_h \otimes \nabla \bG \\
&+ \CF^2_3(\bar\nabla H_h)\nabla \bu\otimes \bG + \CF^2_4(\bar\nabla H_h)\bu \otimes \nabla \bG.
\end{aligned}\end{equation}
where $f_2$ is a bounded function and $\CF^2_j(\bK)$ are some matrices of bounded
analytic functions
defined on $U_\delta$.

\subsection{The unit outer normal  and 
the Laplace Beltrami operator on $\Gamma_t$}
\label{sec:new2.2}

Since $\Gamma$ is a compact hypersurface of $C^3$ class, we have the following
proposition. 
\begin{prop}\label{domain} 
For any  constant  $M_1 \in (0, 1)$, there exist a finite number $n \in \BN$, 
constants $M_2>0$, $d, d'\in (0, 1)$,  $n$ $N$-vectors of functions 
$\Phi^\ell\in C^3(\BR^N)^N$, $n$ points 
$x^\ell \in \Gamma$ and two domains $\CO_\pm$
such that the following assertions hold:
\begin{itemize}
\item[\thetag{i}]~The maps: $\BR^N \ni x \mapsto \Phi^\ell(x) \in \BR^N$
are bijective for $j \in \BN$. 
\item[\thetag{ii}]~$\Omega =
(\bigcup_{\ell=1}^n \Phi^\ell(B_d))\cup \CO_+ \cup \CO_- $,  $B_{d'}(x^\ell)
\subset \Phi^\ell(B_d)  
\subset \Omega$,
$B_{d'}(x^\ell) \cap \Omega_\pm \subset \Phi^\ell(B_d \cap \BR^N_\pm) \subset \Omega_\pm$ and 
$\Gamma \cap B_{d'}(x^\ell) \subset \Phi^\ell(B_d \cap \BR_0^N) $, where 
$B_d=\{x \in \BR^N \mid |x| < d\}$, 
$B_{d'}(x^\ell) = \{x  \in \BR^N \mid |x-x^\ell| < d'\}$, 
$\BR^N_\pm = \{x = (x_1, \ldots, x_N) \mid \pm x_N > 0\}$, and 
$\BR^N_0 = \{x = (x_1, \ldots, x_N) \in \BR^N \mid x_N=0\}$. 
\item[\thetag{iii}]~There exist $n$ $C^\infty$ functions $\zeta^\ell $
such that ${\rm supp}\, \zeta^\ell \subset B_{d'}(x^\ell)$ 
and $\sum_{\ell=1}^n \zeta^\ell = 1$ on $\Gamma$. 
\item[\thetag{iv}]~$\nabla\Phi^\ell = \CA^\ell+ B^\ell$, 
$\nabla(\Phi^\ell)^{-1}=\CA^{\ell,-1}
+ B^{\ell,-1}$ 
where $\CA^\ell$ are $N\times N$ constant 
orthogonal matrices and $B^\ell$
are $N\times N$ matrices of 
$C^3(\BR^N)$ functions satisfying
the conditions:
$
\|B_\ell\|_{L_\infty(\BR^N)} \leq M_1$ and 
$\|\nabla B_\ell\|_{H^1_\infty(\BR^N)} \leq M_2$ 
for  $\ell = 1, \ldots, n$.
\end{itemize}
\end{prop}

In what follows, we write  $B_{d'}(x^\ell)$ simply by  $B^\ell$, and set 
$V_0 = B_d \cap \BR^N_0$. In what follows, the index $\ell$ runs from $1$ through $n$.  
Recall that $\Gamma \cap B^\ell \subset \Phi^\ell(V_0)$,
$\sum_{\ell=1}^n \zeta^\ell=1$ on $\Gamma$,  and
${\rm supp}\, \zeta^\ell \subset B^\ell \subset \Phi^\ell(B_d)  \subset \Omega$.  Let 
$$\tau_j(u) = \frac{\pd \Phi^\ell(u)}{\pd u_j} = A^\ell_j + \CB^\ell_j(u)$$
for $j = 1, \ldots, N$ and $u=(u_1, \ldots, u_N) \in \BR^N$.  
By Proposition \ref{domain}, $A^\ell_j$ are $N$-constan vectors and $\CB^\ell_j(u)$
are $N$ vector of functions such that 
\begin{equation}\label{norm:3} A^\ell_j\cdot A^\ell_k = \delta_{jk}, \quad
\|\CB^\ell_j\|_{L_\infty(\BR^N)} \leq M_1, \quad
\|\nabla \CB^\ell_j\|_{H^1_\infty(\BR^N)} \leq M_2
\end{equation}
where $\delta_{jk}$ are the Kronecker delta symbols defined by
$\delta_{jj} = 1$ and $\delta_{jk} = 0$ for $j\not=k$. 
Notice that $\{\tau_j(u', 0)\}_{j=1}^{N-1}$, $u'=(u_1, \ldots, u_{N-1},0)
\in V_0$,  forms a basis of the tangent space
of $\Gamma \cap B^\ell$.   Let $g^\ell_{ij}(u) = 
\tau^\ell_i(u)\cdot\tau^\ell_j(u)$, $G^\ell(u)$  an $N\times N$ matrix
whose $(i,j)$ th component is $g^\ell_{ij}(u)$,  $g^\ell(u) = \sqrt{\det G^\ell(u))}$, 
and $g_\ell^{ij}(u)$ the $(i,j)$ th component of $(G^\ell)^{-1}$, respectively. 
$G^\ell(u',0)$ is a first fundamental matrix of the tangent space of 
$\Gamma \cap B^\ell$. By \eqref{norm:3} there exist functions 
$\tg^\ell(u)$, $\tg^\ell_{ij}(u)$ and 
$\tg_\ell^{ij}(u)$ such that 
\begin{gather}
g^\ell_{ij}(u) = \delta_{ij} +  \tg^\ell_{ij}(u), \quad 
g^\ell(u) = 1 + \tg^\ell(u), \quad g_\ell^{ij}(u) = \delta_{ij} + \tg_\ell^{ij}(u), 
\nonumber \\
\|(\tg^\ell_{ij}, \tg^\ell, \tg_\ell^{ij})\|_{L_\infty(\BR^N)} \leq CM_1,
\quad
\|\nabla(\tg^\ell_{ij}, \tg^\ell, \tg_\ell^{ij})\|_{H^1_\infty(\BR^N)} \leq C_{M_2}.
\label{norm:5}
\end{gather}
Here and in the following the constant $C_{M_2}$ is a generic constant 
depending on $M_2$. Here and in the following,
we may assume that $0 < M_1 < 1  \leq M_2$. 

We now define an extension of $\bn$ to $\BR^N$ satisfying \eqref{norm:1}. 
Let $\varphi_{i, j}(u) = \pd \Phi^\ell_i(u)/\pd u_j$ with $\Phi^\ell = (\Phi^\ell_1,
\ldots, \Phi^\ell_N)^\top$, and let $\CN^\ell_i(u)$ be an $N\times(N-1)$ 
be functions defined by setting 
$$\CN^\ell_i(u) = (-1)^{i+N}
\det\left(\begin{matrix} \varphi_{1,1} & \cdots & \varphi_{1,N-1} \\
\vdots & \ddots &\vdots \\
\varphi_{i-1,1} & \cdots &\varphi_{i-1, N-1} \\
\varphi_{i+1, 1} & \cdots& \varphi_{i+1, N-1} \\
\vdots & \ddots & \vdots \\
\varphi_{N, 1} & \cdots & \varphi_{N, N-1}
\end{matrix}\right)
$$
for $i=1 = 1, \ldots, N-1$, and set $\CN^\ell = (\CN^\ell_1, \ldots,
\CN^\ell_N)^\top$. Then, we have
$$<\CN^\ell, \frac{\pd \Phi^\ell}{\pd u_k}> = 0
\quad\text{for $k=1, \ldots, N-1$},
$$
because
$$0 = \det\left(\begin{matrix} \varphi_{1,1} & \cdots & \varphi_{1, N-1} & \varphi_{1,k}\\
\vdots & \ddots & \vdots & \vdots \\
\varphi_{N1} & \cdots & \varphi_{N,N-1} & \varphi_{N, k}
\end{matrix}\right)
= \sum_{j=1}^N\CN^\ell_j\varphi_{j,k} = 
<\CN^\ell, \frac{\pd \Phi^\ell}{\pd u_k}>
$$
for $k=1, \ldots, N-1$.  Let $\tilde\bn^\ell = \CN^\ell/|\CN^\ell|$, and then
\begin{gather}
<\tilde\bn^\ell, \tau^\ell_j(u)>=0 \quad\text{ for $j=1\ldots N-1$ and $u \in \BR^N$}.
\nonumber \\
\tilde \bn^\ell\circ(\Phi^\ell)^{-1} = \bn
\quad\text{on $\Gamma \cap B^\ell$}.
\label{normal:1}\end{gather}
Moreover, by \eqref{norm:3} $\|\nabla\tilde \bn^\ell\|_{H^1_\infty(\BR^N)}
\leq C_{M_2}$.  for some constant $C_{M_2}$ depending on $M_2$.  Let 
$$\tilde \bn = \sum_{\ell=1}^n \zeta^\ell \tilde\bn^\ell\circ(\Phi^\ell)^{-1},
$$
and then $\tilde \bn$ satisfies the properties given in \eqref{norm:1}.

Next, we give a representation formula of $\bn_t$.  Since 
$\Gamma_t \cap B^\ell$ is represented by 
$x = \Phi^\ell(u', 0) + H_h(\Phi^\ell(u', 0), t)\bn(\Phi^\ell(u', 0))$
for $(u', 0) \in V_0$,  setting $\tilde H^\ell_h = H_h(\Phi^\ell(u), t)$, 
we define $\tau^\ell_t = 
(\tau^\ell_{t1}(u), \ldots, \tau^\ell_{t N-1})^\top$ by 
\begin{equation}\label{basis:1}
\tau^\ell_{tj}(u) = \frac{\pd}{\pd u_j}(\Phi^\ell(u) + \tilde H_h^\ell(u, t) \tilde\bn^\ell(u)).
\end{equation}
Notice that $\{\tau^\ell_{tj}(u',0)\}_{j=1}^{N-1}$ forms a basis of the tangent space
$\Gamma_t$ locally. 
To obtain a formula of $\tilde\bn_t$, we set 
$\tilde\bn^\ell_t = a(\tilde\bn^\ell + \sum_{j=1}^{N-1}\tau^\ell_jb_j)$ and we decide
$a$ and $b_j$ in such a way that  $|\tilde\bn^\ell_t| = 1$ and $<\tilde\bn^\ell_t,  \tau^\ell_t>=0$.
From  $|\tilde\bn^\ell_t|^2=1$ it follows that 
$$1 = a^2(\tilde\bn^\ell + \sum_{j=1}^{N-1}b_j\tau^\ell_j)
\cdot(\tilde\bn^\ell + \sum_{k=1}^{N-1}
b_k\tau^\ell_k)
= a^2(1 + \sum_{j,k=1}^{N-1}g^\ell_{jk}(u)b_jb_k),$$
so that
\begin{equation}\label{form:1}
a = (1 + \sum_{j,k=1}^{N-1}g^\ell_{jk}(u)b_jb_k)^{-1/2}.
\end{equation}
From $<\tilde\bn^\ell_t,  \tau^\ell_t>=0$ and \eqref{basis:1} it follows that 
$$0 = (\tilde\bn^\ell + \sum_{k=1}^{N-1} b_k\tau^\ell_k)
\cdot(\tau^\ell_j + \tilde H_h^\ell\frac{\pd\tilde \bn^\ell}{\pd u_j}+
\frac{\pd \tilde H_h^\ell}{\pd u_j}\tilde\bn^\ell)
= \sum_{j,k=1}^{N-1}g^\ell_{jk}b_k + \frac{\pd \tilde H_h^\ell}{\pd u_j}
+ \sum_{k=1}^{N-1}b_k<\tau^\ell_k, \frac{\pd \tilde\bn^\ell}{\pd u_j}>\tilde H_h^\ell,
$$
where we have used the first formula in \eqref{normal:1} and 
$<\tilde\bn^\ell, \pd \tilde\bn^\ell/\pd u_j> = 0$ which follows from
$|\tilde\bn^\ell|^2=1$. 
Setting $L^\ell = <\pd \tilde\bn^\ell/\pd u_j, \tau^\ell_k>$, 
we have
$$\nabla' \tilde H_h^\ell = -(G^\ell + L^\ell\tilde H_h^\ell)\bb
$$
where we have set $\bb=(b_1, \ldots, b_{N-1})^\top$ and $\nabla'\tilde H_h^\ell
= (\pd\tilde H_h^\ell/\pd  u_1, \ldots, \pd \tilde H_h^\ell/\pd u_{N-1})^\top.$
We now  introduce a symbol  $O^2_\ell$
which denotes a generic term of the form: 
$$O^2_\ell = a^\ell(u, \bar\nabla' \tilde H_h^\ell)
\bar\nabla' \tilde H_h^\ell\otimes \bar\nabla' \tilde H_h^\ell$$
with some matrix $a^\ell(u, \bK')$ defined on $\BR^N\times
U_\delta'$ satisfying the conditions:  
\begin{align*}
\|a^\ell\|_{L_\infty(\BR^N\times U'_\delta)} &\leq C_{M_2}, 
\\
|\nabla_u a^\ell(u, \bar\nabla'\tilde H_h^\ell)| &\leq C_{M_2}|\bar\nabla_u^2 H_h^\ell|
\\
|\nabla^2_ua^\ell(u, \bar\nabla'\tilde H_h^\ell)| &\leq C_{M_2}(|\bar\nabla_u^3 
\tilde H_h^\ell|
+ |\bar\nabla_u^2 \tilde H_h^\ell|^2), \\
|\pd_t a^\ell(u, \bar\nabla'\tilde H_h^\ell)|
& \leq C_{M_2}|\bar\nabla^1_u\pd_t \tilde H_h^\ell|, \\
|\nabla_u\pd_t a^\ell(u, \bar\nabla'\tilde H_h^\ell)| &\leq C_{M_2}
(|\bar\nabla_u^2\pd_t \tilde H_h^\ell|
+ |\bar\nabla_u^2 \tilde H_h^\ell||\bar\nabla^1_u\pd_t \tilde H_h^\ell|)
\end{align*}
provided that \eqref{assump:0} holds with some small number $\delta > 0$,
where we have set $\nabla_u = (\pd/\pd u_1, \ldots, \pd/\pd u_N)$,
 $\bar\nabla_u^k a = (\pd^\alpha a/\pd u^\alpha \mid |\alpha| \leq k)$,
 and  
$\bK' = (k_0, k_1, \ldots, k_{N-1}) \in U'_\delta = \{\bK' \in \BR^N \mid 
|\bK'| < \delta\}$. 
	Choosing $\delta > 0$ small enough in \eqref{assump:0} and using 
\eqref{norm:5} with small $M_1$, we see that 
$(G^\ell + L^\ell \tilde H_h^\ell)^{-1}=(\bI + (G^\ell)^{-1}L^\ell \tilde H_h^\ell)^{-1}
(G^\ell)^{-1}$ exists, and then
$$\bb = -(\bI + (G^\ell)^{-1}L^\ell \tilde H_h^\ell)^{-1}(G^\ell)^{-1}\nabla'\tilde H_h^\ell.
$$
Therefore,  we have
\begin{equation}\label{norm:6-1}\begin{aligned}
\tilde \bn^\ell_t &= (1 + <G^\ell(\bI + (G^\ell)^{-1}L^\ell \tilde H_h^\ell)^{-1}(G^\ell)^{-1}
\nabla'\tilde H_h^\ell, 
(\bI + (G^\ell)^{-1}L^\ell \tilde H_h^\ell)^{-1}(G^\ell)^{-1}\nabla'\tilde H_h^\ell>)^{-1/2}\\
&\quad \times
(\tilde\bn^\ell - <(\bI + (G^\ell)^{-1}L^\ell \tilde H_h^\ell)^{-1}(G^\ell)^{-1}\nabla'\tilde H_h^\ell,
\tilde\tau^\ell>)\\
& = \tilde\bn^\ell - <(G^\ell)^{-1}\nabla' H_h^\ell, \tilde\tau^\ell> 
+ O^2_\ell.
\end{aligned}\end{equation}
Since
$$\frac{\pd \tilde H_h^\ell}{\pd u_j} 
= \sum_{k=1}^N \frac{\pd \Phi^\ell_k}{\pd u_j}\frac{\pd H_h}{\pd y_k}\circ\Phi^{\ell}
$$
setting 
$$
<G^{-1}\nabla_\Gamma H_h, \tau> = \sum_{\ell=1}^n \zeta^\ell
<(G^\ell)^{-1}\nabla'\tilde H_h^\ell, \tilde\tau^\ell>\circ(\Phi^\ell)^{-1},
$$
by \eqref{norm:6-1} we see that there exists a matrix of functions,
$\bV_\bn(y, \bK)$, 
defined on $\BR^N\times U_\delta$ such that  
\begin{equation}\label{form:2}
\bn_t = \bn - <G^{-1}\nabla_\Gamma H_h, \tau> + \bV_\bn(\cdot, \bar\nabla H_h)
\bar\nabla H_h \otimes \bar\nabla H_h
\quad\text{on $\Gamma$}
\end{equation}
and $\bV_\bn(y, \bK)$ satisfies the following conditions:
${\rm supp}\, \bV_\bn(y, \bK) \subset U_\Gamma$ for any $\bK \in U_\delta$, and 
\begin{align*}
\|\bV_\bn\|_{L_\infty(\BR^N\times U_\delta)} &\leq C_{M_2}, 
\\
|\nabla \bV_\bn(y, \bar\nabla H_h)| &\leq C|\bar\nabla^2 H_h|
\\
|\nabla^2 \bV_\bn(y, \bar\nabla  H_h)| &\leq C(|\bar\nabla^3 H_h|
+ |\bar\nabla^2  H_h|^2), \\
|\pd_t\bV_\bn(y, \bar\nabla H_h)|
& \leq C|\bar\nabla\pd_t  H_h|, \\
|\nabla\pd_t \bV_\bn(y, \bar\nabla H_h)| &\leq C
(|\bar\nabla^2\pd_t H_h|
+ |\bar\nabla^2  H_h||\bar\nabla\pd_t H_h|)
\end{align*}
provided that \eqref{assump:0} holds with some small $\delta > 0$.

We next represent $\Delta_{\Gamma_t}$.  Let $G_t=(g_{ijt})$ be the first fundamental
form and set $g_t = \sqrt{\det G_t}$ and $G_t^{-1} = (g_t^{ij})$. Then, $\Delta_{\Gamma_t}$
is given by setting
\begin{equation}\label{laplace:1}\Delta_{\Gamma_t} f = \frac{1}{g_t}\sum_{i,j=1}^{N-1}(g_tg^{ij}_t
\frac{\pd f}{\pd u_j})
\quad\text{on $V_0$}.
\end{equation}
Since $<\tilde \bn^\ell, \pd\tilde\bn^\ell/\pd u_j> = 0$ and
$<\pd\Phi^\ell/\pd u_j, \tilde\bn> = <\tau^\ell_j, \tilde\bn> = 0$, 
in view of \eqref{basis:1},  setting
$$\alpha^\ell_{ij} = <\tau^\ell_i, \frac{\pd\tilde\bn^\ell}{\pd u_j}>
+ <\tau^\ell_j, \frac{\pd\tilde\bn^\ell}{\pd u_i}>, \quad
\beta^\ell_{ij} = <\frac{\pd \tilde\bn^\ell}{\pd u_i}, \frac{\pd \tilde\bn^\ell}{\pd u_j}>,
$$
we have 
$$g^\ell_{tij} = <\tau^\ell_{ti}, \tau^\ell_{tj}> = g^\ell_{ij} 
+ \alpha^\ell_{ij}\tilde H_h^\ell + \beta^\ell_{ij}(\tilde H_h^\ell)^2
+ <\frac{\pd \tilde H_h^\ell}{\pd u_i}, \frac{\pd\tilde H_h^\ell}{\pd u_j}>.
$$
Notice that $\alpha^\ell_{ij}$ and $\beta^\ell_{ij}$ are all bounded $C^2$ functions. 
Here,  what a function, $f$, is bounded $C^2$ means that $f$ is a 
$C^2$ function and $f$ and its derivatives up to order $2$ are all bounded. 
Let $g^\ell_t = \sqrt{\det(g^\ell_{tij})}$ and $(G_t^\ell)^{-1} = (g^{ij\ell}_t)$,
and then  
by  \eqref{assump:0} with small $\delta > 0$ and \eqref{norm:5}, 
we have the representation formulas:  
$$g^\ell_t = g^\ell + \gamma_0^\ell(u)\tilde H_h^\ell + O^2_\ell,
\quad \frac{1}{g^\ell_t} = \frac{1}{g^\ell} + \gamma^\ell_1(u)\tilde H_h^\ell + O^2_\ell,
\quad g^{ij\ell}_t = g^{ij}_\ell + \gamma_{ij}^\ell(u)\tilde H_h^\ell
+ O^2_\ell,
$$
where $\gamma^\ell_0(u)$, $\gamma^\ell_1(u)$ and
 $\gamma^\ell_{ij}(u)$ are some bounded $C^2$ functions
defined on $\BR^N$.  In view of \eqref{laplace:1}, setting 
\begin{equation}\label{laplace:2}\begin{aligned}
V^{1\ell}_{\Delta ij} &= \gamma_{ij}^\ell(u)\tilde H_h^\ell + O^2_\ell, \\
V^{2\ell}_{\Delta j} & = \sum_{i=1}^{N-1}\Bigl(\frac{\pd}{\pd u_i}(\gamma^\ell_{ij}(u)
\tilde H_h^\ell)
+ \frac{\pd}{\pd u_i}O^2_\ell  +\frac{1}{g_\ell}(\pd_i(\gamma^\ell_0(u)\tilde H_h^\ell)
+ \pd_i O^2_\ell)\\
&\quad + (\gamma^\ell_1(u)\tilde H_h^\ell + O^2_\ell)(\pd_ig^\ell
 + \pd_i(\gamma^\ell_0(u)\tilde H_h^\ell)
+ \pd_i O^2_\ell)\Bigr)
\end{aligned}\end{equation}
we have 
\begin{equation}\label{lap:2.4}
\Delta_{\Gamma_t} = \Delta_{\Gamma} + \dot\Delta_{\Gamma_t}
\quad\text{on $\Gamma \cap B^\ell$}, 
\end{equation}
where  $\dot\Delta_{\Gamma_t}$ is an operator 
 defined by setting 
\begin{equation}\label{lap:2.4*}\dot\Delta_{\Gamma_t}f 
= \sum_{\ell=1}^n\zeta^\ell\Bigl(
\sum_{i,j=1}^{N-1}
V^{1\ell}_{\Delta ij}\frac{\pd^2( f(\Phi^\ell(u',0))}{\pd u_i\pd u_j}
+\sum_{j=1}^{N-1} V^{2\ell}_{\Delta j}\frac{\pd (f(\Phi^\ell(u',0))}{\pd u_j}\Bigr)
\circ(\Phi^\ell)^{-1}.
\end{equation}

We finally derive a formula for the surface tension.  Recall that
$H(\Gamma_t)\bn_t = \Delta_{\Gamma_t}x$ for $x \in \Gamma_t$. 
For $x \in \Gamma_t$, $x$ is represented by $x = \Phi^\ell(u', 0)
+ \tilde H_h^\ell(u',0, t)\tilde\bn^\ell(u',0)$
locally.  By \eqref{lap:2.4} and \eqref{lap:2.4*}, 
we have
\begin{align*}
&<H(\Gamma_t)\bn_t, \bn> = <\Delta_\Gamma(y+ H_h\bn), \bn> \\
&+ \sum_{\ell=1}^n \zeta^\ell\Bigl(
\sum_{i,j=1}^{N-1}<V^{1\ell}_{\Delta ij}\frac{\pd^2}{\pd u_i\pd u_j}
(\Phi^\ell + \tilde H_h^\ell\bn^\ell), \bn^\ell>
+ \sum_{j=1}^{N-1}<V^{2\ell}_{\Delta j}\frac{\pd}{\pd u_j}
(\Phi^\ell + \tilde H_h^\ell\bn^\ell), \bn^\ell>
\Bigr)
\end{align*}
on $\Gamma$. 
Since $\Delta_\Gamma y = H_h(\Gamma)\bn$ for $y \in \Gamma$ and
\begin{align*}
<\Delta_\Gamma \bn, \bn> 
= \sum_{\ell=1}^n\zeta^\ell\sum_{i,j=1}^{N-1} g^{ij}_\ell
<\frac{\pd^2 \bn^\ell}{\pd u_i\pd u_j}, \bn^\ell>
= -\sum_{\ell=1}^\infty\zeta^\ell\sum_{i,j=1}^{N-1} g^{ij}_\ell
<\frac{\pd  \bn^\ell}{\pd u_i},\frac{\pd \bn^\ell}{\pd u_j}>
=-<G^{-1}\nabla_\Gamma\bn, \nabla_\Gamma\bn>
\end{align*}
as follows from  $<\pd \tilde\bn^\ell/\pd u_i, \tilde \bn^\ell> = 0$,
we have
$$ <\Delta_\Gamma(y+ H_h\bn), \bn>
= H_h(\Gamma) + \Delta_\Gamma H_h  - <G^{-1}\nabla_\Gamma \bn, \nabla_\Gamma\bn>H_h.$$
Moreover, by \eqref{lap:2.4*} and \eqref{laplace:2},  we have
\begin{align*}
<V^{1\ell}_{\Delta ij}\frac{\pd^2}{\pd u_i\pd u_j}\Phi^\ell, \tilde\bn^\ell>
&=  <\gamma^\ell_{ij}\frac{\pd^2}{\pd u_i\pd u_j}\Phi^\ell, \tilde\bn^\ell> \tilde H_h^\ell
+ O^2_\ell, \\
<V^{1\ell}_{\Delta ij}\frac{\pd^2}{\pd u_i\pd u_j}(\tilde H_h^\ell \tilde\bn^\ell), \tilde\bn^\ell>
& = V^{1\ell}_{\Delta ij}\frac{\pd ^2 \tilde H_h^\ell}{\pd u_i\pd u_j} 
+  <V^{1 \ell}_{\Delta ij}\frac{\pd ^2\tilde \bn^\ell}{\pd u_i\pd u_j}, \tilde\bn^\ell>
\tilde H_h^\ell\\
&=  V^{1\ell}_{\Delta ij}\frac{\pd ^2 \tilde H_h^\ell}{\pd u_i\pd u_j} 
-  <(\frac{\pd}{\pd u_i}V^{1 \ell}_{\Delta ij})\frac{\pd \tilde \bn^\ell}{\pd u_j}, \tilde\bn^\ell>
\tilde H_h^\ell
-<V^{1 \ell}_{\Delta ij}\frac{\pd \tilde \bn^\ell}{\pd u_j}, \frac{\pd\tilde\bn^\ell}{\pd u_i}>
\tilde H_h^\ell,\\
<V^{2\ell}_{\Delta j}\frac{\pd}{\pd u_j}(\tilde H_h^\ell \tilde\bn^\ell), \tilde\bn^\ell>
&= V^{2\ell}_{\Delta j}\frac{\pd \tilde H_h^\ell}{\pd u_j}. 
\end{align*}
Combining these formulas gives that 
\begin{equation}\label{surface}\begin{aligned}
< H_h(\Gamma_t)\bn_t, \bn> &= H_h(\Gamma) + \Delta_\Gamma H_h 
+ a(y)H_h + \bV_s(y, \bar\nabla H_h)\bar\nabla H_h\otimes \bar\nabla^2 H_h 
\end{aligned}\end{equation}
where $a(y)$ is a bounded $C^1$ function, and 
 $\bV_s = \bV_s(y, \bK)$ are some matrices of 
 functions defined on $\BR^N\times U_\delta$ such that
${\rm supp}\, \bV_s(y, \bK) \subset U_\Gamma$ for any $\bK \in U_\delta$,
$\sup_{t \in (0,  T)} \|\bV_s(\cdot, \bar\nabla H_h)\|_{L_\infty(\Omega)} \leq C_{M_2}$, 
\begin{equation} \label{surface*}\begin{aligned}
& |\nabla \bV_s(y, \bar\nabla H_h)| \leq C_{M_2}|\bar\nabla^2H_h(y, t)|, 
\quad 
 |\pd_t\bV_s(y, \bar\nabla H_h)| \leq C_{M_2}|\bar\nabla \pd_tH_h(y, t)|.
% \\
 % & |\nabla\pd_t \bV_s(y, \bar\nabla H_h)| \leq C_{M_2}(|\bar\nabla^2\pd_tH_h(y, t)|+
 % |\bar\nabla^2H_h(y, t)||\bar\nabla\pd_tH_h(y, t)|).
  \end{aligned}\end{equation}
  provided that \eqref{assump:0} holds with some small constant $\delta > 0$. 
%%%%%%%%%%%%%%%%%%%%%%
%%%%%%%%%%%%%%%%%%%%%
%%%%%%%%%%%%%%%%%%%%%%

\subsection{Derivation of transmission conditions and kinematic condition}\label{sec:new2.3}

%%%%%%%%%%%%%%%%%%%%%%%%%

%%%%%%%%%%%%%%%%%%%%%%%%
%%%%%%%%%%%%%%%%%%%%%%%%%%
%%%%%%%%%%%%%%%%%%%%%%%%
We first consider the kinematic condition: 
$V_{\Gamma_t} = \bv_+\cdot\bn_t$. Note that $\bv_+ = \bv_-$ on $\Gamma_t$.
 Since 
$$V_{\Gamma_t} = \frac{\pd x}{\pd t}\cdot\bn_t= \frac{\pd H_h}{\pd t}
\bn\cdot\bn_t, $$
it follows from  \eqref{form:2} that 
\begin{equation}\label{maineq:3}
\pd_th  + <\nabla_\Gamma h \perp \bu_+> - \bu_+\cdot\bn  = 
<\bu_+-\frac{\pd H_h}{\pd t}\bn, \bV_n(\cdot, \bar\nabla H_h)\bar\nabla H_h \otimes
\bar\nabla H_h>.
\end{equation}
Here and in the following, we write
 $$<\nabla_\Gamma h \perp \bu_+>
 = \sum_{\ell=1}^n\zeta^\ell\Bigl(
 \sum_{i,j=1}^{N-1}g^{ij}_\ell\frac{(\pd h\circ\Phi^\ell)}{\pd u_j}
 <\tilde \tau^\ell_i, \bu_+\circ\Phi^\ell>\Bigr).
$$
If we move $<\nabla_\Gamma h \perp \bu_+>$ to the right hand side in proving
the local wellposedness by using a standard fixed point argument, we have to 
assume the smallness of initial velocity field $\bu_0$ as well as the smallness of 
initial height $h_0$.  But, this is not  satisfactory.  We have to treat at least
the large initial velocity case for the local well-posedness.  To avoid the 
smallness assumption of initial velocity field, 
we use an idea due to Padula and Solonnikov 
\cite{PS}.   Let $\bu_0 \in B^{2(1-1/p)}_{q,p}(\dot\Omega)$ be an 
initial velocity field and set $\bu^+_0 = \bu_0|_{\Omega_+}$.  We know that 
$[[\bu_0]] = 0$ on $\Gamma$, which follows from the compatibility
conditions. Let $\tilde \bu^+_0$ be an extension of 
$\bu^+_0$ to $\BR^N$ such that $\tilde\bu^+_0 = \bu^+_0$ in 
$\Omega_+$ and 
\begin{equation}\label{4.4.1} \|\tilde\bu^+_0\|_{B^{2(1-1/p)}_{q,p}(\BR^N)}
\leq C\|\bu^+_0\|_{B^{2(1-1/p)}_{q,p}(\Omega_+)}.
\end{equation}
 Let 
 $$\bu_\kappa = \frac{1}{\kappa}\int^\kappa_0T_0(s)\tilde\bu^+_0\,ds$$
 where $\{T_0(s)\}_{s\geq 0}$ is a $C^0$ analytic semigroup generated by
 $-\Delta + \lambda_0$ with large $\lambda_0$ in $\BR^N$, that is 
 $$T_0(s)f = \CF^{-1}\bigl[e^{-s(|\xi|^2+\lambda_0)}\hat f(\xi)\bigr](x).
 $$
 Here, $\hat f$ denotes the Fourier transform of $f$ and $\CF^{-1}$ the 
 inverse Fourier transform.  We know that
 \begin{equation}\label{4.4.2}\begin{aligned}
 \|T_0(\cdot)\tilde\bu^+_0\|_{L_p((0, \infty), H^2_q(\BR^N))} 
 + \|\pd_tT_0(\cdot)\tilde\bu^+_0\|_{L_p((0, \infty), L_q(\BR^N))}
 &+ \|T_0(\cdot)\tilde\bu^+_0\|_{L_\infty((0, \infty), B^{2(1-1/p)}_{q,p}(\BR^N))}\\
 &\leq C\|\bu^+_0\|_{B^{2(1-1/p}_{q,p}(\BR^N)},
 \end{aligned}\end{equation}
 which yields that
 \begin{equation}\label{4.4.3}\begin{aligned}
 \|\bu_\kappa\|_{B^{2(1-1/p)}_{q,p}(\BR^N)} & 
 \leq C\|\bu^+_0\|_{B^{2(1-1/p)}_{q,p}(\Omega_+)}, \\
 \|\bu_\kappa\|_{H^2_q(\BR^N)}
 & \leq C\kappa^{-1/p}\|\bu^+_0\|_{B^{2(1-1/p)}_{q,p}(\Omega_+)}.
 \end{aligned}\end{equation}
 As a kinematic condition, we use the following equation:
 \begin{equation}\label{kinema:1} \pd_th + <\nabla_\Gamma h \perp \bu_\kappa> 
 -\bu\cdot\bn = d(\bu, H_h)
 \end{equation}
 with
 \begin{equation}\label{kinema:2}
 d(\bu, H_h) = <\nabla_\Gamma H_h \perp \bu- \bu_\kappa> 
 + <\bu - \frac{\pd H_h}{\pd t}\bn, \bV_\bn(\cdot, \bar\nabla H_h)
 \bar\nabla H_h \otimes\bar\nabla H_h>.
 \end{equation}
 
Let  $\CE_\mp$ be an  the extension map
acting on $\bu_\pm \in H^2_q(\Omega_\pm)$ satisfying the 
properties:
$\CE_\mp(\bu_\pm) \in H^2_q(\Omega)$,  $\CE_\mp(\bu_\pm) = \bu_\pm$ in $\Omega_\pm$, 
\begin{equation}\label{trace:1}
(\pd_x^\alpha \CE_\mp(\bu_\pm))(x_0)
= \lim_{x\to x_0 \atop x \in \Omega_\pm}\pd_x^\alpha \bu_\pm(x)
\end{equation}
for $x_0 \in \Gamma$ and $\alpha \in \BN_0^N$ with $|\alpha|\leq 1$, 
and 
\begin{equation}\label{trace:1*}
\|\CE_\mp(\bu_\pm)\|_{H^\ell_q(\Omega)}
\leq C_{\ell, q}\|\bu_\pm\|_{H^\ell_q(\Omega_\pm)}
\end{equation}
for $\ell=0,1,2$. Note that 
\begin{equation}\label{trace:2}
[[\pd_x^\alpha\bu]] = \pd_x^\alpha\CE_-(\bu_+)|_\Gamma
- \pd_x^\alpha\CE_+(\bu_-)|_\Gamma
\end{equation}
for $|\alpha|\leq 1$ on $\Gamma$.
For the notational simplicity, we write
\begin{equation}\label{trace:3}
tr[\bu] = \CE_-(\bu_+) - \CE_+(\bu_-)
\end{equation}
and then, we have
\begin{equation}\label{trace:4}\begin{aligned}
\pd_x^\alpha\,tr[\bu]|_{\Gamma}&=[[\pd_x^\alpha\bu]]\\
\|tr[\bu]\|_{H^i_q(\Omega)} &\leq C(\|\bu_+\|_{H^i_q(\Omega)} + 
\|\bu_-\|_{H^i_q(\Omega)}) = C\|\bu\|_{H^i_q(\dot\Omega)}
\end{aligned}\end{equation}
for $i=0,1,2$.

We next consider the interface conditions.  First, we consider
\begin{equation}\label{B}
[[(\bT(\bv, \fp) + \bT_M(\bH_h))\bn_t]]=\sigma H_h(\Gamma_t)\bn_t
\quad \text{on $\Gamma_t$}.
\end{equation}
Let 
\begin{equation}\label{Proj}
\BPi_t\bd = \bd - <\bd, \bn_t>\bn_t, 
\quad
\BPi_0\bd = \bd - <\bd, \bn>\bn
\end{equation}
The following lemma was given in Solonnikov \cite{Sol15}.
\begin{lem}\label{Lem:2.1}
If $\bn_t\cdot\bn\not=0$, then for arbitrary vector $\bd$, 
$\bd=0$ is equivalent to
\begin{equation}\label{deom}
\BPi_0\BPi_t\bd = 0 \quad\text{and}\quad
\bn\cdot\bd = 0.
\end{equation}
\end{lem}
In view of Lemma \ref{Lem:2.1}, the interfae condition \eqref{B} is
equivalent to that the following two conditions hold:
\begin{align}
\BPi_0\BPi_t
[[\nu(\bD(\bu) + \bV_D(\bK)\nabla\bu) ) + \bT_M(\bG)]]\bn_t &=0, 
\label{equiv:1} \\
\bn\cdot([[\nu(\bD(\bu) + \bV_D(\bK)\nabla \bu) - \fq\bI+\bT_M(\bG)]]\bn_t
-\sigma H_h(\Gamma_t)\bn_t) &= 0.
\label{equiv:2}
\end{align}
Here and hereafter, $\bV_D(\bK)\nabla \bu$ is the $N\times N$ matrix
with $(i, j)$ components $V_{Dij}(\bK)\nabla\bu$ (cf. \eqref{dc:1}).
Noting that $\BPi_0\BPi_0 = \BPi_0$, 
 we see that the condition \eqref{equiv:1} is written by 
\begin{equation}\label{maineq:4}
\BPi_0[[\nu\bD(\bu)]]\bn = \bh'_1(\bu, \bG, H_h)
\end{equation}
with
\begin{equation}\label{g:3}\begin{split}
\bh'_1(\bu, \bG, H_h) &= \BPi_0(\BPi_0 - \BPi_t)[[\nu\bD(\bu)]]\bn_t
+ \BPi_0[[\nu\bD(\bu)]](\bn-\bn_t) \\
&-\BPi_0\BPi_t[[\nu\bV_D(\bK)\nabla\bu + \bT_M(\bG)]]\bn_t.
\end{split}\end{equation}

On the other hand, by \eqref{surface} 
we see that Eq. \eqref{equiv:2} is written by 
\begin{equation}\label{mainq:5}
\bn\cdot[[\nu\bD(\bu)-\fq\bI]]\bn 
 -\sigma(\Delta_\Gamma h + ah)
 = h_{1N}(\bu, \bG, H_h) + \sigma \bV_s(\cdot, \bar\nabla H_h)\bar\nabla H_h
 \otimes \bar\nabla^2 H_h.
\end{equation}
with
\begin{equation}\label{g:4}\begin{split}
&h_{1N}(\bu, \bG, H_h)  = (\bn\cdot\bn_t)^{-1}\{\bn\cdot[[\nu\bD(\bu)]]
(\bn - \bn_t) -\bn\cdot[[\nu\bV_D(\bK)\nabla\bu + \bT_M(\bG)]]\bn_t\}.
\end{split}\end{equation}
In particular, setting 
$$\bh_1(\bu, \bG, H_h ) = (\bh'_1(\bu, \bG, H_h), 
 h_{1N}(\bu, \bG, H_h) + \sigma \bV_s(\cdot, \bar\nabla H_h)\bar\nabla H_h
 \otimes \bar\nabla^2 H_h), $$
 in view of \eqref{dc:1}, \eqref{form:2}, and \eqref{trace:4},  we may write
\begin{equation}\label{g.3*}\begin{aligned}
\bh_1(\bu, \bG, H_h ) & = \bV^1_\bh(\cdot, \bar\nabla H_h)\bar\nabla H_h \otimes\nabla
tr[\bu] +a(y)tr[\bG]\otimes tr[\bG] 
+ \bV^2_\bh(\cdot, \bar\nabla H_h)\bar\nabla H_h\otimes
tr[\bG]\otimes tr[\bG]\\
& + \bV_s(\cdot, \bar\nabla H_h)\bar\nabla H_h \otimes\bar\nabla^2H_h.
\end{aligned}\end{equation}
Here,  $a(y)$ is an $N$-vector of bounded $C^2$ function, and
$\bV^i_\bh(\cdot, \bK)$ ($i=1,2$) are some matrices of funtions defined
on $\BR^N\times U_\delta$ satisfying the conditions:
$\|\bV^i_\bh\|_{L_\infty(\BR^N\times U_\delta)} \leq C$, 
${\rm supp}\, \bV^i_\bh(y, \bK) \subset  U_\Gamma$, 
\begin{equation} \label{g.4*}\begin{aligned}
& |\nabla \bV_\bh^i(y, \bar\nabla H_h)| \leq C|\bar\nabla^2H_h(y, t)|, 
\quad 
 |\pd_t\bV_\bh^i(y, \bar\nabla H_h)| \leq C|\bar\nabla \pd_tH_h(y, t)|.
% & |\nabla\pd_t \bV_\bh^i(y, \bar\nabla H_h)| \leq C(|\bar\nabla^2\pd_tH_h(y, t)|+
 % |\bar\nabla^2H_h(y, t)||\bar\nabla\pd_tH_h(y, t)|) 
  \end{aligned}\end{equation}
provided that \eqref{assump:0} holds with some somall $\delta > 0$.

From \eqref{dc:1} we see that the interface condition:
$[[\alpha^{-1}{\rm curl}\,\bH_h + \mu(\bv\otimes\bH_h-\bH_h\otimes\bv)]]\bn_t=0$
is written by
\begin{equation}\label{maineq:6}
[[\alpha^{-1}{\rm curl}\,\bG]]\bn = \bh_2(\bu, \bG, H_h)
\end{equation}
with
$$\bh_2(\bu, \bG, H_h) = [[\alpha^{-1}\curl\bG]](\bn - \bn_t) 
-[[\alpha^{-1}\bV_C(\bK)\nabla\bu]]\bn_t 
-[[\mu(\bu\otimes\bG - \bG\otimes\bu)]]\bn_t.
$$
Here, $\bV_C(\bK)\nabla\bu$ is the $N\times N$ matrix with $(i, j)$ components 
quantities $V_{Cij}(\bK)\nabla\bu$ given in \eqref{dc:1}.
In particular, in view of \eqref{dc:1}, \eqref{form:2}, and \eqref{trace:4}, we may write
\begin{equation}\label{maineq:6*}\begin{aligned}
\bh_2(\bu, \bG, H_h)
&= \bV^3_\bh(\cdot, \bar\nabla H_h)\bar\nabla H_h \otimes\nabla tr[\bu](\CE_-(\bu_+)
+ b(y)tr[\bu]\otimes tr[\bG]\\
&+ \bV^4_\bh(\cdot, \bar\nabla H_h)
\bar\nabla H_h \otimes tr[\bu]\otimes tr[\bG]. 
\end{aligned}\end{equation}
Here, $b(y)$ is an $N$-vector of $C^2$ functions, 
and $\bV^i_\bh(\cdot, \bK)$ ($i=3,4$) are some matrices of funtions defined
on $\BR^N\times U_\delta$ satisfying the same conditions as these stated in \eqref{g.4*}
provided that \eqref{assump:0} holds with some somall $\delta > 0$.

From \eqref{div:1}, we see that the interface condition:
$[[\mu\dv\bH_h]]=0$ is written by 
\begin{equation}\label{maineq:7}
[[\mu\dv\bG]]= h_3(\bu, \bG, H_h)
\end{equation}
with
$$h_3(\bu, \bG, H_h) = -\mu\sum_{j,k=1}^N\tilde V_{0jk}(\bK)\bK
\frac{\pd}{\pd y_k}tr[\bu]_j
$$
where $V_{0jk}(\bK) = \tilde V_{0jk}(\bK)\bK$ are the symbols given in \eqref{trans:2} and
$tr[\bu] = (tr[\bu]_1, \ldots, tr[\bu]_N)$.

Finally, the interface conditions:
$[[\mu\bH\cdot\bn_t]]=0$ and $[[\bH-<\bH, \bn_t>\bn_t]]=0$
are written by 
\begin{equation}\label{maineq:8}
[[\mu\bG\cdot\bn]]= k_1(\bG, H_h), \quad
[[\bG-<\bG, \bn>\bn]] = \bk_2(\bG, H_h)
\end{equation}
with
$$k_1(\bG, H_h) = [[\mu\bG\cdot(\bn-\bn_t)]],\quad
\bk_2(\bG, H_h) = [[<\bG, \bn_t-\bn>\bn_t]]
+ [[<\bG, \bn>(\bn_t-\bn)]].
$$
In particular, in view of  \eqref{form:2} and \eqref{trace:4}, we may write
\begin{equation}\label{maineq:6*}\begin{aligned}
(k_1(\bG, H_h), \bk_2(\bG, H_h))
&=\bV^5_\bk(\cdot, \bar\nabla H_h)\bar\nabla H_h\otimes tr[\bG]. 
\end{aligned}\end{equation}
Here, $\bV^5_\bk(\cdot, \bK)$ is some matrices of funtions defined
on $\BR^N\times U_\delta$ satisfying the same conditions as these stated in \eqref{g.4*}
provided that \eqref{assump:0} holds with some somall $\delta > 0$.

\subsection{Statement of the local well-posedness theorem}\label{sec:new2.4}

Summing up the results obtained in subsections \ref{sec:new2.1},
\ref{sec:new2.2}, and \ref{sec:new2.3}, we have seen that equations \eqref{mhd.2}
are transformed to the following equations: 
\begin{equation}\label{mhd.3}\left\{
\begin{aligned}
\rho\pd_t\bu - \DV\bT(\bu, \fq) &= \bff_1(\bu, \bG, H_h),
&\quad &\text{in $\dot\Omega\times(0, T)$}, \\
 \dv\bu= g(\bu, H_h) &= \dv\bg(\bu, H_h)
&\quad &\text{in $\dot\Omega\times(0, T)$}, \\
\pd_t h +<\nabla_\Gamma h \perp \bu_\kappa> - \bn\cdot\bu &= d(\bu, H_h)
&\quad &\text{on $\Gamma\times(0, T)$},\\
[[\bu]]=0, \quad 
[[\bT(\bu, \fq)\bn]] -\sigma (\Delta_\Gamma h + ah)\bn &= \bh_1(\bu, \bG, H_h),
&\quad &\text{on $\Gamma\times(0, T)$}, \\
\mu\pd_t\bG - \alpha^{-1}\Delta\bG &=\bff_2(\bu, \bG, H_h)
&\quad &\text{in $\dot\Omega\times(0, T)$}, \\
[[\alpha^{-1}\curl\bG]]\bn= \bh_2(\bu, \bG, H_h), \quad [[\mu\dv\bG]] 
&=h_3(\bu, \bG, H_h)
&\quad &\text{on $\Gamma\times(0, T)$}, \\
[[\mu\bG\cdot\bn]]=k_1(\bG, H_h),
\quad [[\bG-<\bG, \bn>\bn]]&=\bk_2(\bG, H_h)
&\quad &\text{on $\Gamma\times(0, T)$}, \\
\bu_\pm=0,\quad \bn_\pm\cdot\bG_\pm =0, \quad
(\curl\bG_\pm)\bn_\pm &= 0
&\quad&\text{on $S_\pm\times(0, T)$}, \\
(\bu, \bG, h)|_{t=0}  &= (\bu_0, \bG_0, h_0)
&\quad&\text{in $\dot\Omega\times\dot\Omega\times \Gamma$},
\end{aligned}
\right.
\end{equation}
where, $H_h$ is a function satisfying  Eq. \eqref{eq:harmonic} for $h$.

 The purpose of this paper is to prove
the following local in time unique existence theorem. 
%To state the assumption on the domain $\Omega$, we introduce the weak Neumann
%problem:
%\begin{equation}\label{wdn}(\rho^{-1}\nabla u, \nabla\varphi)_{\dot\Omega}
%= (\bff, \nabla\varphi)_{\dot\Omega}
%\quad\text{for any $\varphi \in \hat H_h^1_{q'}(\Omega)$}
%\end{equation}
%where we have set
%$$\hat H^1_{q'}(\Omega) = \{\varphi \in L_{q', {\rm loc}}(\Omega) \mid
%\nabla\varphi \in L_{q'}(\Omega)\}, 
%\quad q'= q/(q-1).
%$$
%In this paper, we say that  the weak 
%Neumann problem is uniquely solvable, if for any
%$\bff \in L_q(\Omega)^N$, problem \eqref{wdn} admits a unique
%solution $u \in \hat H^1_q(\Omega)$ possessing the estimate:
%$\|\nabla u\|_{L_q(\Omega)} \leq C\|\bff\|_{L_q(\Omega)}$ with some
%constant $C > 0$ independent of $\bff$ and $u$.
%\begin{remark}
%\thetag{i} For $q = 2$, the weak Neumann problem is uniquely
%solvable on  $\Omega$.  But, in the case that 
%$q\not=2$, in general it is not obvious whether the weak Neumann problem
%is uniquely solvable. 
%In particular, in the case that $\Omega= \BR^N$ or that $\Omega$ is a bounded domain, 
%we can show that the weak Neumann problem is uniquely solvable
%for any $q \in (1, \infty)$. \\
%\thetag{ii}  In this paper, what $g = \dv\bg$
%in $\dot\Omega$ denotes that 
%\begin{equation}\label{div:0} 
%(g, \varphi)_\Omega = -(\bg, \nabla\varphi)_\Omega
%\quad\text{for all $\varphi \in \hat H^1_{q'}(\Omega)$}.
%\end{equation}
%\end{remark}
\begin{thm}\label{thm:main}
Let $2 < p < \infty$, $N < q < \infty$ and  $2/p + N/q < 1$  and $B > 0$. Assume that 
 the condition \eqref{initial:0} holds.
Then, there exist a small number  $\epsilon$ and a small time $T > 0$ depending on $B$ such that
if initial data $h_0 \in B^{3-1/p-1/q}_{q,p}(\Gamma)$ satisfies 
the smallness condition $\|h_0\|_{B^{3-1/p-1/q}_{q,p}}\leq \epsilon$, 
and  $(\bv_0, \bH_0) \in B^{2(1-1/p)}_{q,p}(\dot\Omega)^{2N}$
 satisfies $\|(\bv_0, \bH_0)\|_{B^{2(1-1/p)}_{q,p}(\dot\Omega)}
 \leq B$ and the compatibility condition:
\begin{equation}\label{compati}\begin{split}
\dv \bv_0 = 0 
 \quad&\text{in $\dot\Omega$}, \\
[[(\nu\bD(\bv_0)+\bT_M(\bH_0)\bn]]_\tau 
= 0, \quad [[\bv_0]]=0 \quad&\text{on $\Gamma$},  \\
[[\{\alpha^{-1}\curl\bH_0 + \mu(\bv_0\otimes\bH_0- \bH_0\otimes\bv_0)\}\bn]]
=0, \quad
[[\mu\dv\bH_0]] = 0 \quad&\text{on $\Gamma$}, \\
[[\mu\bH_0\cdot\bn]]=0, \quad
[[\bH_0-<\bH_0, \bn>\bn]] = 0 \quad&\text{on $\Gamma$}, \\
\bv_{0\pm} = 0, \quad \bn_0\cdot\bH_{0\pm} = 0, \quad
(\curl \bH_{0\pm})\bn_\pm = 0 
\quad&\text{on $S_\pm$},
\end{split}\end{equation}
then, Eq. \eqref{mhd.3} admits  unique solutions 
$\bu$, $\fq$,  $\bG$, and $h$ with
\begin{alignat*}2
&\bu \in H^1_p((0, T), L_q(\dot\Omega)^N) \cap 
L_p((0, T), H^2_q(\dot\Omega)^N),
&\quad &\fq \in L_p((0, T), H^1_q(\dot\Omega) + \hat H^1_q(\Omega), \\
&\bG \in H^1_p((0, T), L_q(\dot\Omega)^N) \cap 
L_p((0, T), H^2_q(\dot\Omega)^N), &\quad
&h \in H^1_p((0, T), W^{2-1/q}_q(\Gamma)^N) \cap L_p((0, T),
W^{3-1/q}_q(\Gamma)), \\
&\|H_h\|_{L_\infty((0, T), H^1_\infty(\Omega)} \leq \delta
\end{alignat*}
possessing the estimate:
\begin{align*}
&\|(\bu, \bG)\|_{L_p((0, T), H^2_q(\dot\Omega))} 
+ \|\pd_t(\bu, \bG)\|_{L_p((0, T), L_q(\dot\Omega))} \\
&\quad + \|h\|_{L_p((0, T), W^{3-1/q}_q(\Gamma))}
+ \|\pd_th\|_{L_p((0, T), W^{2-1/q}_q(\Gamma))}
+ \|\pd_th\|_{L_\infty((0, T), W^{1-1/q}_q(\Gamma))}
\leq f(B)
\end{align*}
Here, $\delta$ is a constant appearing in \eqref{assump:0}
and $f(B)$ is some polynomial of $B$. 
\end{thm}
%%%%%%%%%%%%%%%
%%%%%%%%%%%%%%%%%
\section{Linear Theory}

Since the coupling of the velocity field and the magnetic field in 
\eqref{mhd.2} is semilinear, the linearized equations are decouple.  Namely, 
we consider the two linearlized equations:  one is the Stokes equations
with transmission conditions on $\Gamma$ and non-slip conditions 
on $S_\pm$ and another is the system of the heat 
equations with transmission conditions on $\Gamma$ and the perfect
wall conditions on $S_\pm$.  In the following, we set 
$\dot\Omega = \Omega_+ \cup \Omega_-$ and $\Omega
= \dot\Omega\cup \Gamma$. And, we assume that $\Gamma$ is a compact
hypersurface of $C^3$ class and that $S_\pm$ are hypersurfaces of $C^2$ class. 
\subsection{Two phase problem for the Stokes equautions}
\label{subsec:4-1}
This subsection is devoted to presenting the $L_p$-$L_q$
maximal regularity for the two phase problem of the 
Stokes equations with transmission conditions given as follows:
\begin{equation}\label{4.1.1}\left\{
\begin{aligned}
\rho\pd_t\bu - \DV\bT(\bu, \fq) &= \bff_1
&\quad &\text{in $\dot\Omega\times(0, T)$}, \\
 \dv\bu= g &= \dv\bg
&\quad &\text{in $\dot\Omega\times(0, T)$}, \\
\pd_t h +<\nabla_\Gamma h \perp \bw_\kappa> - \bn\cdot\bu &= d
&\quad &\text{on $\Gamma\times(0, T)$},\\
[[\bu]]=0, \quad 
[[\bT(\bu, \fq)\bn]] -\sigma (a h + \Delta_\Gamma h )\bn &= \bh
&\quad &\text{on $\Gamma\times(0, T)$}, \\
\bu_\pm&=0,
&\quad&\text{on $S_\pm\times(0, T)$}, \\
(\bu, h)|_{t=0}  &= (\bu_0, h_0)
&\quad&\text{in $\dot\Omega\times\Gamma$}.
\end{aligned}
\right.
\end{equation}

An assumption for equations  \eqref{4.1.1} is the following: 
\begin{itemize}\item[\thetag{a.1}]~
 $a$ is a bounded $ C^1$ functions defined on $\Omega$.
 \item[\thetag{a.2}]~
$\bw_\kappa$ is a family of $N$-vector of functions defind on 
$\Gamma$ for $\kappa \in (0, 1)$ such that 
\begin{align*}
&|\bw_\kappa(x)| \leq m_1, 
\quad |\bw_\kappa(x)- \bw_\kappa(y)| \leq m_1|x-y|^b \enskip
\text{for any $x, y \in \Gamma$}, \quad 
\|\bw_\kappa\|_{W^{2-1/r}_r(\Gamma)} \leq m_2\kappa^{-c}.
 \end{align*} 
 Here, $m_1$, $m_2$, $b$ and $c$ are some positive constants
 and $r \in (N, \infty)$.
\end{itemize}
\begin{thm}\label{thm:4.1.1} 
Let $1 < p < \infty$, $1 < q \leq r$, $2/p+1/q \not=1, 2$, and $T > 0$.
 Assume that the assumptions
\thetag{a.1} and \thetag{a.2} are satisfied.  
Then, there exists a $\gamma_0 > 0$ such that
the following assertion holds:  Let 
$\bu_0 \in B^{2(1-1/p)}_{q, p}(\dot\Omega)$ and $h_0 \in 
B^{3-1/p-1/q}_{q,p}(\Gamma)$. Let $\bff$, $g$, $\bg$, $\bh=(\bh', h_N)$,
 and $d$ appearing in the right-hand side of Eq. \eqref{4.1.1} be 
given functions 
satisfying the following conditions:
\begin{gather*}
 \bff \in L_p((0, T), L_q(\dot\Omega)^{N}), 
\quad 
e^{-\gamma t} g \in L_p(\BR, H^1_q(\dot\Omega)) \cap 
H^{1/2}_p(\BR, L_q(\dot\Omega)), \quad 
e^{-\gamma t}\bg \in H^1_p(\BR, L_q(\dot\Omega)^N), \\
e^{-\gamma t}\bh \in L_p(\BR, H^1_q(\Omega)^N) \cap 
H^{1/2}_p(\BR, L_q(\Omega)^N), \quad 
d \in L_p((0, T), W^{2-1/q}_q(\Gamma))
\end{gather*} 
for any $\gamma \geq \gamma_0$ with some $\gamma_0$.
  Assume that $\bu_0$,  $g$, and
$\bh$ satisfy the following compatibility conditions:
\begin{align}
\dv\bu_0 = g|_{t=0} \quad \text{on $\dot\Omega$}, & \label{compati:1} \\
[[\nu\bD(\bu_0)\bn]]_\tau = \bh_\tau|_{t=0} 
\quad\text{on $\Gamma$} & \quad\text{provided
$2/p + 1/q < 1$}, \label{compati:2}\\
[[\bu_0]] =0\quad\text{on $\Gamma$}, \quad
\bu_{0, \pm} = 0 \quad\text{on $S_\pm$} &\quad\text{provided
$2/p + 1/q < 2$}, \label{compati:3}
\end{align}
where $\bd_\tau = \bd - <\bd, \bn>\bn$. Then, Eq. \eqref{4.1.1} admit unique
solutions $\bu$, $\fq$, and $h$ with
\begin{align*}
\bu & \in L_p((0, T), H^2_q(\dot\Omega)^N) \cap
H^1_p((0, T), L_q(\dot\Omega)^N), \quad 
\fq  \in L_p((0, T), H^1_q(\dot\Omega) + \hat H^1_q(\Omega)),\\
h & \in L_p((0, T, W^{3-1/q}_q(\Gamma))
 \cap H^1_p((0, T), W^{2-1/q}_q(\Gamma))
\end{align*}
possessing the estimates:
\begin{align*}
&\|\pd_t\bu\|_{L_p((0, T), L_q(\dot\Omega))}
+ \|\bu\|_{L_p((0, T), H^2_q(\dot\Omega))} 
 + \|\pd_th\|_{L_p((0, \infty), W^{2-1/q}_q(\Gamma))}
+ \|h\|_{L_p((0, T), W^{3-1/q}_q(\Gamma))}\\
&\quad \leq Ce^{\gamma\kappa^{-c}T}\{\|\bu_0\|_{B^{2(1-1/p)}_{q,p}(\dot\Omega)}
+ \kappa^{-c} \|h_0\|_{B^{3-1/p-/q}_{q,p}(\Gamma)} 
+ \|\bff\|_{L_p((0, T), L_q(\dot\Omega))}\\
&\quad + \|e^{-\gamma t}g\|_{L_p(\BR, H^1_q(\dot\Omega))}
+ \|e^{-\gamma t}\bg\|_{H^1_p(\BR, L_q(\dot\Omega))}
+ \|e^{-\gamma t}\bh\|_{L_p(\BR, H^1_q(\Omega))}\\
&\quad +(1+\gamma^{1/2})
(\|e^{-\gamma t}g\|_{H^{1/2}_p(\BR, L_q(\dot\Omega))}
+ \|e^{-\gamma t}\bh\|_{H^{1/2}_p(\BR, L_q(\Omega))})
+ \|d\|_{L_p((0, T), H^2_q(\Omega))}\}
\end{align*}
for any $\gamma \geq \gamma_0$ with some constant $C > 0$ independent of 
$\gamma$. 
\end{thm}
\begin{remark} \thetag1
Theorem \ref{thm:4.1.1} has been proved in Shibata and Saito
\cite{SS}.  And the reason why we assume that $\Gamma$ is a compact in this paper is that
the weak Neumann problem is uniquely solvable. Namely,  if we consider 
the weak Neumann problem:
\begin{equation}\label{wdn}(\rho^{-1}\nabla u, \nabla\varphi)_{\dot\Omega}
= (\bff, \nabla\varphi)_{\dot\Omega}
\quad\text{for any $\varphi \in \hat H^1_{q'}(\Omega)$}
\end{equation}
where we have set
$$\hat H^1_{q'}(\Omega) = \{\varphi \in L_{q', {\rm loc}}(\Omega) \mid
\nabla\varphi \in L_{q'}(\Omega)\}, 
\quad q'= q/(q-1),
$$
then for any $\bff \in L_q(\Omega)^N$, problem \eqref{wdn} admits a unique
solution $u \in \hat H^1_q(\Omega)$ satisfying the estimate:
$\|\nabla u\|_{L_q(\Omega)} \leq C\|\bff\|_{L_q(\Omega)}$ with some constant
$C > 0$.  If  $\Gamma$ is unbounded, then in general
 we have to 
assume that the weak Neumann problem is uniquely solvable, 
except for a few cases like $\Gamma$ is flat, that is $\Gamma = \{x = (x_1, \ldots, x_N) \in \BR^N \mid x_N=0\}$, 
or $\Gamma$ is asymptotically flat. 
\end{remark}
\subsection{Two phase problem for the linear electro-magnetic field equations}
\label{subsec:4-2} 
This subsection is devoted to presenting the 
$L_p$-$L_q$ maximal regularity for the linear electro-magnetic field equations.  
The problem is formulated by the following equations:
\begin{equation}\label{4.2.1}\left\{ \enskip \begin{split}
\mu\pd_t\bH - \alpha^{-1}\Delta\bH = \bff 
\quad&\text{in $\dot\Omega\times(0, \infty)$}, \\
[[\alpha^{-1}\curl\bH]]\bn = \bh', \quad
[[\mu\dv\bH]] = h_N 
\quad&\text{on $\Gamma\times(0, \infty)$}, \\
[[\bH-<\bH, \bn>\bn]]= \bk', \quad 
[[\mu \bH\cdot\bn]] = k_N
\quad&\text{on $\Gamma\times(0, \infty)$}, \\
\bn_\pm\cdot\bH_\pm = 0, \quad
(\curl\bH_\pm)\bn_\pm = 0 \quad
&\text{on $S_\pm\times(0, \infty)$}, \\
\bH|_{t=0} = \bH_0\quad
&\text{in $\dot\Omega$}.
\end{split}\right.\end{equation}
\begin{thm}\label{thm:4.2.1}
Let $1 < p, q < \infty$, $2/p+1/q \not=1,2$,  and $T > 0$. 
Then, there exists a $\gamma_0$ such that the following assertion holds:
Let $\bH_0 \in B^{2(1-1/p)}_{q,p}(\dot\Omega)$ and let 
$\bff$, $\bh = (\bh', h_N)$, and $\bk = (\bk', k_N)$
be given functions appearing in the right-hand side of Eq. \eqref{4.2.1}
and satisfying the following conditions:
\begin{align*}
&\bff \in L_p((0, T), L_q(\dot\Omega)^N), \quad
e^{-\gamma t}\bh \in L_p(\BR, H^1_q(\Omega)^N) \cap 
H^{1/2}_p(\BR, L_q(\Omega)^N), \\
&e^{-\gamma t}\bk \in L_p(\BR, H^2_q(\Omega)^N) \cap 
H^1_p(\BR, L_q(\Omega)^N)
\end{align*}
for any $\gamma \geq \gamma_0$.  Assume that $\bff$, $\bh$ and 
$\bk$ satisfy the following compatibility conditions:
\begin{align}
&[[\alpha^{-1}\curl\bH_0]]\bn=\bh'|_{t=0}, \quad
[[\mu\dv\bH_0]]=h_N|_{t=0} \quad\text{on $\Gamma$}, 
\quad[[\curl\bH_{0\pm})\bn_\pm = 0
\quad\text{on $S_\pm$} \label{compati:4}
\intertext{provided $2/p + 1/q < 1$;} 
&[[\bH_0-<\bH_0, \bn>\bn]]=\bk'|_{t=0}, \quad
[[\mu\bH_0\cdot\bn]]=k_N|_{t=0} \quad\text{on $\Gamma$}, \quad
\bn_\pm\cdot\bH_{0\pm} = 0 \quad\text{on $S_\pm$} \label{compati:5}
\end{align}
provided $2/p + 1/q < 2$. 
Then, problem \eqref{4.2.1} admits a unique solution $\bH$
with
$$\bH
\in L_p((0, T), H^2_q(\dot\Omega)^N) \cap 
H^1_p((0, T), L_q(\dot\Omega)^N)$$
possessing the estimate:
\begin{align*}
&\|\pd_t\bH\|_{L_p((0, T), L_q(\dot\Omega))}
+ \|\bH\|_{L_p((0, T), H^2_q(\dot\Omega))} 
\leq Ce^{\gamma T}\{\|\bH_0\|_{B^{2(1-1/p)}_{q,p}(\dot\Omega)} 
+ \|\bff\|_{L_p(\BR, L_q(\dot\Omega))} \\
&\quad+ \|e^{-\gamma t}\bh\|_{L_p(\BR, H^1_q(\Omega))}
 + \|e^{-\gamma t}\bh\|_{H^{1/2}_p(\BR, L_q(\Omega))}
+ \|e^{-\gamma t}\bk\|_{L_p(\BR, H^2_q(\Omega))}
+ \|e^{-\gamma t}\pd_t\bk\|_{L_p(\BR, L_q(\Omega))}\}
\end{align*}
for any $\gamma \geq \gamma_0$ with some constant $C > 0$ independent of 
$\gamma$. 
\end{thm}
\begin{remark} Theorem \ref{thm:4.2.1} was proved by Froloba and Shibata \cite{FS1} 
under the assumption that $\Omega$ is a uniformly $C^3$ domain.  Of course, if  
$\Gamma$ is a compact hypersurface of $C^3$ class, then $\Omega$ is a 
uniform $C^3$ domain. 
\end{remark}
%%%%%%%%%%%%%%%%%%%%%%%

%%%%%%%%%%%%%%%% nonlinear term 
\section{ Estimates of non-linear terms}\label{sec:4}

First of all, we give an iteration scheme to prove Theorem \ref{thm:main} by the 
Banach fixed point theorem. 
Given $h$ satisfying \eqref{height:1}, 
let $H_h$  be a unique solution of equation \eqref{eq:harmonic}
satisfying \eqref{height:2} and \eqref{1.3.2}.  Let  $\bU_T$ be 
 an underlying space  defined by setting 
\begin{equation}\label{under:0}\begin{aligned}
\bU_T &= \{(\bu, \bG, h) \mid (\bu, \bG) 
\in H^1_p((0, T), L_q(\dot\Omega)^{2N}) \cap 
L_p((0, T), H^2_q(\dot\Omega)^{2N}), \\
&h \in L_p((0, T), W^{3-1/q}_q(\Gamma)) \cap 
H^1_p((0, T), W^{2-1/q}_q(\Gamma)), \\
&(\bu, \bG,  h)|_{t=0} = (\bu_0, \bG_0, h_0) \quad\text{in $\dot\Omega\times\dot\Omega
\times\Gamma$}, \quad E_T(\bu, \bG, h) \leq L, \quad
\|H_h\|_{L_\infty((0, T), H^1_\infty(\Omega))} \leq \delta\},
\end{aligned}\end{equation}
where we have set  
\begin{equation}\label{norm:2}\begin{aligned}
E_T(\bu, \bG, h)& = E^1_T(\bu) + E^1_T(\bG)+  E^2_T(h) + \|\pd_th\|_{L_\infty((0, T), 
W^{1-1/q}_q(\Gamma))}, \\
E^1_T(\bw) & =\|\bw\|_{L_p((0, T), H^2_q(\dot\Omega))}
+ \|\pd_t\bw\|_{L_p((0, T), L_q(\dot\Omega))} \quad \bw \in \{\bu, \bG\}, \\
E^2_T(h) & = \|h\|_{L_p((0, T), W^{3-1/q}_q(\Gamma))} 
+ \|\pd_th\|_{L_p((0, T), W^{2-1/q}_q(\Gamma))}.
\end{aligned}\end{equation}
For initial data $\bu_0$, $\bG_0$, and $h_0$, we assume that
\begin{equation}\label{initial:4.1}\|\bu_0\|_{B^{2(1-1/p)}_{q,p}(\dot\Omega)} \leq B, 
\quad \|\bG_0\|_{B^{2(1-1/p)}_{q,p}(\dot\Omega)} \leq B,
\quad \|h_0\|_{B^{3-1/p-1/q}_{q,p}(\Gamma)} \leq \epsilon.
\end{equation}
Here, $B$ is a given positive number. Since we mainly consider the case where $\bu_0$ and 
$\bG_0$ are large, we may assume that $B > 1$ in the following. And  
we shall choose $L>0$ large enough and $\epsilon > 0$ small 
enough eventually, and so we may assume that $0 < \epsilon < 1 < L$.
Given $(\bu, \bG, h) \in \bU_T$, let $(\bv, \fq, \rho)$ be solutions of the equations:
\begin{equation}\label{eq:4.1}\left\{
\begin{aligned}
\rho\pd_t\bv - \DV\bT(\bv, \fq) &= \bff_1(\bu, \bG, H_h),
&\quad &\text{in $\dot\Omega\times(0, T)$}, \\
 \dv\bv= g(\bu, H_h) &= \dv\bg(\bu, H_h)
&\quad &\text{in $\dot\Omega\times(0, T)$}, \\
\pd_t \rho +<\nabla_\Gamma \rho \perp \bu_\kappa> - \bn\cdot\bv_+ &= d(\bu, H_h)
&\quad &\text{on $\Gamma\times(0, T)$},\\
[[\bv]]=0, \quad 
[[\bT(\bv, \fq)\bn]] -\sigma (\Delta_\Gamma \rho + a\rho)\bn &= \bh_1(\bu, \bG, H_h),
&\quad &\text{on $\Gamma\times(0, T)$}, \\
\bv_\pm&=0
&\quad&\text{on $S_\pm\times(0, T)$}, \\
(\bv, \rho)|_{t=0}  &= (\bu_0,  h_0)
&\quad&\text{in $\dot\Omega\times \Gamma$}.
\end{aligned}
\right.
\end{equation}
And,  let $\bH$ be a solution of the equations: 
\begin{equation}\label{eq:4.2}\left\{
\begin{aligned}
\mu\pd_t\bH - \alpha^{-1}\Delta\bH &=\bff_2(\bu, \bG, H_h)
&\quad &\text{in $\dot\Omega\times(0, T)$}, \\
[[\alpha^{-1}\curl\bH]]\bn= \bh_2(\bu, \bG, H_h), \quad [[\mu\dv\bH]] 
&=h_3(\bu, \bG, H_h)
&\quad &\text{on $\Gamma\times(0, T)$}, \\
[[\mu\bH\cdot\bn]]=k_1(\bG, H_\rho),
\quad [[\bH-<\bH, \bn>\bn]]&=\bk_2(\bG, H_\rho)
&\quad &\text{on $\Gamma\times(0, T)$}, \\
\quad \bn_\pm\cdot\bH_\pm =0, \quad
(\curl\bH_\pm)\bn_\pm &= 0
&\quad&\text{on $S_\pm\times(0, T)$}, \\
\bH|_{t=0}  &= \bG_0
&\quad&\text{in $\dot\Omega$}.
\end{aligned}
\right.
\end{equation}
Notice that to define $\bH$ we use not only $H_h$ but also $H_\rho$  unlike  
Padula and Solonnikov \cite{PS} to avoid their technial assumption that 
the velocity field is slightly regular than the magnetic field.

In this section, we shall show the estimates of the nonlinear terms 
appearing in the right sides of equations \eqref{eq:4.1}
and  \eqref{eq:4.2}.  Since $(\bu, \bG, h) \in \bU_T$,  we have 
\begin{align}\label{under:1}
E_T(\bu, \bG, h) &\leq L, \\
\|H_h\|_{L_\infty((0, T), H^1_\infty(\Omega))} &\leq \delta. \label{4.5.4}
\end{align}
Below, we assume that $2 < p < \infty$, $N < q < \infty$ and
$2/p + N/q < 1$.  
We use the following inequalities which follows from Sobolev's inequality.
\begin{equation}\label{4.5.1}\begin{aligned}
\|f\|_{L_\infty(\dot\Omega)} & \leq C\|f\|_{H^1_q(\dot\Omega)}, \\ 
\|fg\|_{H^1_q(\dot\Omega)} & \leq C\|f\|_{H^1_q(\dot\Omega)}
\|g\|_{H^1_q(\dot\Omega)},  \\
\|fg\|_{H^2_q(\dot\Omega)} & \leq C(\|f\|_{H^2_q(\dot\Omega)}
\|g\|_{H^1_q(\dot\Omega)}
+ \|f\|_{H^1_q(\dot\Omega)}
\|g\|_{H^2_q(\dot\Omega)}),  \\
\|fg\|_{W^{1-1/q}_q(\Gamma)} & \leq C\|f\|_{W^{1-1/q}_q(\Gamma)}
\|g\|_{W^{1-1/q}_q(\Gamma)},  \\
\|fg\|_{W^{2-1/q}_q(\Gamma)} & \leq C(\|f\|_{W^{2-1/q}_q(\Gamma)}
\|g\|_{W^{1-1/q}_q(\Gamma)}
+ \|f\|_{W^{1-1/q}_q(\Gamma)}
\|g\|_{W^{2-1/q}_q(\Gamma)}).
\end{aligned}\end{equation}
For any $C^k$ function, $f(u)$,  defined for $|u| \leq \sigma$, 
we consider a composite function $f(u(x))$, and then for $N < q < \infty$, 
we have
\begin{equation}\label{eq:4.5.1}\begin{aligned}
\|\nabla(f(u)vw)\|_{L_q(\dot\Omega)} & \leq C\|(f, f')\|_{L_\infty}
(1+\|\nabla u\|_{L_q(\dot\Omega)})\|v\|_{H^1_q(\dot\Omega)}
\|w\|_{H^1_q(\dot\Omega)} ; \\
\|\nabla^2(f(u)vw)\|_{L_q(\dot\Omega)} & \leq 
C\|(f, f', f'')\|_{L_\infty}(\|\nabla v\|_{H^1_q(\dot\Omega)}\|w\|_{H^1_q(\dot\Omega)}
+ \|v\|_{H^1_q(\dot\Omega)}\|\nabla w\|_{H^1_q(\dot\Omega)}\\
&\quad + \|\nabla u\|_{H^1_q(\dot\Omega)}(1 + \|\nabla u\|_{L_q(\dot\Omega))})
\|v\|_{H^1_q(\dot\Omega)}\|w\|_{H^1_q(\dot\Omega)}),
\end{aligned}\end{equation}
provided that $\|u\|_{L_\infty(\dot\Omega)} \leq \sigma$. 
We  use the following estimate of the time trace proved by 
a real interpolation theorem:
\begin{align}
&\|\bw\|_{L_\infty((0, T), B^{2(1-1/p)}_{q,p}(\dot\Omega))}\nonumber \\
&\quad \leq C\{\|\bw_0\|_{B^{2(1-1/p)}_{q,p}(\dot\Omega)}
+ \|\bw\|_{L_p((0, T), H^2_q(\dot\Omega))}
+ \|\pd_t\bw\|_{L_p((0, T), L_q(\dot\Omega))}\}
\leq C(B+L)\label{4.5.12} \\
\intertext{for $\bw \in \{\bu, \bG\}$, }
&\|h\|_{L_\infty((0, T), B^{3-1/p-1/q}_{q,p}(\Gamma))}\nonumber\\
&\quad
\leq C\{\|\rho_0\|_{B^{3-1/p-1/q}_{q,p}(\Gamma)} + 
\|h\|_{L_p((0, T), W^{3-1/q}_q(\Gamma))} + \|\pd_th\|_{L_p((0, T), W^{2-1/q}_q(\Gamma))}\}
\leq CL, \label{4.5.13}.
\end{align}
And we have 
\begin{equation}\label{smallh:1}\|h\|_{L_\infty((0, T), W^{2-1/q}_q(\Gamma))}
\leq \|\rho_0\|_{W^{2-1/q}_q(\Gamma)} + T^{1/p'}\|\pd_th\|_{L_p((0, T),
W^{2-1/q}_q(\Gamma))}
\leq \epsilon + T^{1/{p'}}L..
\end{equation}
In what follows, we assume that $0 < \epsilon = T = \kappa < 1$  and $1 \leq B, L$.  
In particular, $\epsilon + LT^{1/p'} \leq T + LT^{1/p'} \leq 2LT^{p1/p'}$.  In what follows,
we assume that $LT^{1/p'} \leq 1$, and so by \eqref{smallh:1},  
\begin{equation}\|h\|_{L_\infty((0, T), W^{2-1/q}_q(\Gamma))} \leq LT^{1/p'},
\quad \|h\|_{L_\infty((0, T), W^{2-1/q}_q(\Gamma))} \leq 1. 
\label{4.5.14}
\end{equation}

We first estimate $\bff_1(\bu, \bG, H_h)$. In view of \eqref{maineq:1*}, we may write
\begin{equation}\label{non:f1}
\bff_1(\bu, \bG, H_h) = \bV_{\bff_1}(\cdot, \bar\nabla H_h)(\bar\nabla H_h\otimes(\pd_t\bu, \nabla^2\bu)
+ \pd_tH_h\otimes \nabla\bu + \bu\otimes\nabla\bu
+ \bar\nabla^2H_h\otimes\nabla\bu + \bG\otimes\nabla\bG),
\end{equation}
where $\bV_{\bff_1}(y, \bK)$ is a matrix of bounded functions defined on $\Omega\times
\{\bK \in \BR^{N+1} \mid |\bK| \leq \delta\}$. 
Applying \eqref{1.3.2}, 
\eqref{4.5.4} and \eqref{4.5.1}, 
 we have
\begin{align*}
\|\bff_1(\bu, \bG, H_h)\|_{L_q(\dot\Omega)} &\leq 
C\{\|h\|_{W^{2-1/q}_q(\Gamma)}\|(\pd_t\bu, \nabla^2\bu)\|_{L_q(\dot\Omega)}
+ \|\pd_th\|_{W^{1-1/q}_q(\Gamma)}\|\nabla\bu\|_{L_q(\dot\Omega)}
+\|\bu\|_{H^1_q(\dot\Omega)}^2 \\
&\quad 
 + \|h\|_{W^{2-1/q}_q(\dot\Omega)}\|\nabla\bu\|_{H^1_q(\Omega)}
+ \|\bG\|_{H^1_q(\dot\Omega)}^2\}.
\end{align*}
Here and in the following, we write
$$\|f\|_{L_p((a, b))} = \Bigl(\int^b_a|f(t)|^p\,dt\Bigr)^{1/p}.
$$
For a  maximal regularity term, $f$, and a lower order term, g,  we estimate
$$\|fg\|_{L_p((0, T))} \leq \|f\|_{L_p((0, T))}\|g\|_{L_\infty((0, T))}.$$
And only in the lower order term, $g$, case, we estimate
$$\|g\|_{L_p((0, T))} \leq T^{1/p}\|g\|_{L_\infty((0, T))}.
$$
Thus, using \eqref{1.3.2}, we have 
\begin{align}
&\|\bff_1(\bu, \bG, H_h)\|_{L_p((0, T), L_q(\dot\Omega))}
= \|\|\bff_1(\bu, \bG, H_h)\|_{L_q(\dot\Omega)}\|_{L_p((0, T))} \nonumber \\
&\quad  \leq C\{
\|h\|_{L_\infty((0, T), W^{2-1/q}_q(\Gamma))}
(\|\pd_t\bu\|_{L_p((0, T), L_q(\dot\Omega))}
+ \|\bu\|_{L_p((0, T), H^2_q(\dot\Omega))})\label{est:f} \\
&\quad 
+ T^{1/p}(\|\pd_th\|_{L_\infty((0, T), W^{1-1/q}_q(\Gamma))}\|\bu\|_{L_\infty((0, T), H^1_q(\dot\Omega))}
+\|\bu\|_{L_\infty((0, T), H^1_q(\dot\Omega))}^2+ \|\bG\|_{L_\infty((0, T), H^1_q(\dot\Omega))}^2)\}.
\nonumber\end{align}
By \eqref{1.3.2}, \eqref{under:1}, \eqref{initial:4.1}, \eqref{4.5.12}, and \eqref{4.5.13}, we have 
\begin{equation}\label{4.5.10}\begin{aligned}
&\|\bw\|_{L_\infty((0, T), H^1_q(\dot\Omega)} \leq C(B+L)
\quad\text{for $\bw \in \{\bu, \bG\}$}, \\
&\|\pd_t\bw\|_{L_p((0, T), L_q(\dot\Omega))}
+ \|\bw\|_{L_p((0, T), H^2_q(\dot\Omega))} \leq  L\quad\text{for $\bw \in \{\bu, \bG\}$},
\\
&\|\pd_th\|_{L_\infty((0, T), W^{1-1/q}_q(\Gamma))}  \leq CE_T(\bu, \bG, h)
\leq CL, \\
&\|\pd_th\|_{L_p((0, T), W^{2-1/q}_q(\Gamma))}  + \|h\|_{L_p((0, T), W^{3-1/q}_q(\Gamma))} \leq L. 
\end{aligned}\end{equation}
Thus, by \eqref{4.5.14} and \eqref{4.5.10}
$$
\|\bff_1(\bu, \bG, H_h)\|_{L_p((0, T), L_q(\dot\Omega))}
\leq C\{T^{1/p'}L^2+ T^{1/p}(L+B)^2\}.
$$
Since $1/p' > 1/p$ as follows from $1 < p' =p/(p-1) < 2 < p < \infty$, we have 
\begin{equation}\label{4.5.8}
\|\bff_1(\bu, \bG, H_h)\|_{L_p((0, T), L_q(\dot\Omega))} \leq CT^{1/p}(B+L)^2.
\end{equation}

We next estimate $d(\bu, H_h)$ given in \eqref{kinema:2}. 
We shall prove that
\begin{equation}\label{4.5.24}\begin{aligned}
\|d(\bu, H_h)\|_{L_\infty((0, T), W^{1-1/q}_q(\Gamma))}
& \leq C}L(B+L)T^{1/{p'}; \\
\|d(\bu, H_h)\|_{L_p((0, T), W^{2-1/q}_q(\Gamma))}
& \leq C_sL^2(B+L)T^{\frac{s}{p'(1+s)}},
\end{aligned}\end{equation}
where  $s$ is a constant for which $s \in (0, 1-2/p)$.  Here and in the following, $C_s$ is 
a generic constant depending on $s$, whose value may change from line to line.

In fact, 
by \eqref{1.3.2},  \eqref{4.5.4}, \eqref{4.5.1}, and \eqref{eq:4.5.1}, 
\begin{equation}\label{non:d1}\begin{aligned}
&\|d(\bu, H_h)\|_{W^{1-1/q}_q(\Gamma)}
\leq C\{\|h\|_{W^{2-1/q}_q(\Gamma)}\|\bu_+-\bu_\kappa\|_{H^1_q(\Omega_+)}
\\
&\quad +(1 + \|h\|_{W^{2-1/q}_q(\Gamma)})
\|h\|_{W^{2-1/q}_q(\Gamma)}^2(\|\bu\|_{H^1_q(\Omega_+)}
+ \|\pd_th\|_{W^{1-1/q}_q(\Gamma)})\}; \\
&\|d(\bu, H_h)\|_{W^{2-1/q}_q(\Gamma)} \leq C\{\|h\|_{W^{3-1/q}_q(\Gamma)}
\|\bu-\bu_\kappa\|_{H^1_q(\Omega_+)}
+\|h\|_{W^{2-1/q}_q(\Gamma)}\|\bu-\bu_\kappa\|_{H^2_q(\Omega_+)} 
\\
& \quad +\|h\|_{W^{3-1/q}_q(\Gamma)}(1 + \|h\|_{W^{2-1/q}_q(\Gamma)})
\|h\|_{W^{2-1/q}_q(\Gamma)}^2(\|\bu\|_{H^1_q(\Omega_+)}
+ \|\pd_th\|_{W^{1-1/q}_q(\Gamma)}) \\
&\quad 
+(\|\bu\|_{H^2_q(\Omega_+)} + \|\pd_th\|_{W^{2-1/q}_q(\Gamma)})\|h\|_{W^{2-1/q}_q(\Gamma)}^2
\\
&\quad +(\|\bu\|_{H^1_q(\Omega_+)} + \|\pd_th\|_{W^{1-1/q}_q(\Gamma)})
\|h\|_{W^{3-1/q}_q(\Gamma)}\|h\|_{W^{2-1/q}_q(\Gamma)}\}.
\end{aligned}
\end{equation}
By \eqref{4.4.3}, we have
\begin{equation}\label{Sol:0}\begin{aligned}
\|\bu_+-\bu_\kappa\|_{L_\infty((0, T), H^1_q(\Omega_+)}
&\leq C(\|\bu\|_{L_\infty((0, T), H^1_q(\dot\Omega))} + \|\bu_0\|_{L_\infty((0, T), B^{2(1-1/p)}_{q,p}(\dot\Omega)})
\\
&\leq C(L+B), 
\end{aligned}\end{equation}
and so by \eqref{4.5.14} and \eqref{4.5.10} 
\begin{align*}
\|d(\bu, H_h)\|_{L_\infty((0, T), W^{1-1/q}_q(\Gamma))}
&\leq C\{ T^{1/p'}L(B+L) + 
(1 +  T^{1/p'}L)( T^{1/p'}L)^2(B+L)\}\\
& \leq CT^{1/p'}L(B+L),
\end{align*}
which shows the first inequality in \eqref{4.5.24}.

To prove the second inequality in \eqref{4.5.24}, we use the estimates:
\begin{equation}\label{eq:d.3}\begin{aligned}
\|\bu-\bu_\kappa\|_{L_p((0, T), H^2_q(\Omega_+))} &\leq C(L+B), \\
\|\bu-\bu_\kappa\|_{L_\infty((0, T), H^1_q(\Omega_+))}
&\leq C_sT^{\frac{s}{p'(1+s)}}(L+B).
\end{aligned}\end{equation}
Here, $s$ is a fixed constant for which $0 < s < 1-2/p$. 
In fact, 
by \eqref{4.4.3} and \eqref{initial:4.1} 
$$\|\bu - \bu_\kappa\|_{H^2_q(\Omega_+)}
\leq C(\|\bu(\cdot, t)\|_{H^2_q(\dot\Omega)} + \kappa^{-1/p}B),
$$
and so by \eqref{under:1} and \eqref{4.5.10} we have
$$
\|\bu-\bu_\kappa\|_{L_p((0, T), H^2_q(\Omega_+)}
\leq C(L + T^{1/p}\kappa^{-1/p}B) \leq C(L+B), 
$$
because we have taken $\kappa=T$. This shows the first inequality in \eqref{eq:d.3}. 
For $t \in (0, T)$
by \eqref{4.4.2}, \eqref{initial:4.1} and \eqref{under:1} 
\begin{align*}
\|\bu-\bu_\kappa\|_{L_q(\Omega_+)}
& \leq \|\bu - \bu_0\|_{L_q(\Omega+)} + \|\bu_0 - \bu_\kappa\|_{L_q(\Omega_+)} \\
& \leq \int^T_0\|\bu_t(\cdot, t)\|_{L_q(\Omega_+)}\,dt
+ \frac{1}{\kappa}\int^\kappa_0\|T_0(s)\tilde\bu_0^+-\bu_0\|_{L_q(\Omega_+)}\,ds\\
& \leq T^{1/{p'}}L + \frac{1}{\kappa}\int^\kappa_0\Bigl(\int^s_0
\|\pd_rT_0(r)\tilde\bu_0^+\|_{L_q(\Omega_+)}\,dr\Bigr)\,ds 
\\
& \leq T^{1/{p'}}L + C\kappa^{1/{p'}}B \leq C(L+B)T^{1/p'}.
\end{align*}
By real interpolation, 
\begin{align*}
\|\bu-\bu_\kappa\|_{H^1_q(\Omega_+)} 
&\leq C_s\|\bu-\bu_\kappa\|^{s/(1+s)}_{L_q(\Omega_+)}
\|\bu - \bu_\kappa\|_{W^{1+s}_q(\Omega_+)}^{1/(1+s)}
\leq C_s\|\bu-\bu_\kappa\|^{s/(1+s)}_{L_q(\Omega_+)}
\|\bu-\bu_\kappa\|_{B^{2(1-1/p)}_{q,p}(\dot\Omega)}^{1/(1+s)}
\end{align*}
for any $s \in (0, 1-2/p))$, which, combined with \eqref{4.4.3}
and \eqref{4.5.12},  yields the second inequality in \eqref{eq:d.3}. 

Applying \eqref{4.5.14}, \eqref{4.5.10}, and \eqref{eq:d.3}
to the second inequality in \eqref{non:d1} yields that 
\begin{align*}
\|d(\bu, &H_h)\|_{L_p((0, T), W^{2-1/q}_q(\Gamma))}\\
 &\leq C\{L(B+L)T^{1/p'} + C_sT^{\frac{s}{p'(1+s)}}L(L+B) \\
&\qquad +LLT^{1/p'}(L+B) + LLT^{1/p'}
+ (L+B)LLt^{1/p'}\} \\
&  \leq CT^{\frac{s}{p'(1+s)}}L^2(L+B), 
\end{align*}
which proves the second inequlity in \eqref{4.5.24}.

We now estimate $g(\bu, H_h)$, $\bg(\bu, H_h)$ and 
$\bh_1(\bu, \bG,  H_h)$ given in \eqref{g:1} and \eqref{g.3*}, respectively.
We have to extend them to the whole time line $\BR$. For this purpose,
we first define operators which have nice behaviour at infinity in time 
and whose initial values are $\bw_{0\pm}
\in B^{2(1-1/p)}_{q,p}(\Omega_\pm)$ 
for $\bw \in \{\bu_0, \bG_0\}$. 
Let $\tilde\bw_{0\pm} \in B^{2(1-1/p)}_{q,p}(\BR^N)$ be extensions of 
$\CE_\mp(\bw_{0\pm})$ to $\BR^N$ such that 
$$\tilde\bw_{0\pm} =\CE_\mp( \bw_{0\pm})\quad\text{in $\Omega$}, \quad
\|\tilde\bw_{0\pm}\|_{B^{2(1-1/p)}_{q,p}(\BR^N)} \leq 
C\|\bw_{0\pm}\|_{B^{2(1-1/p)}_{q,p}(\Omega_\pm)} \leq CB.
$$
Let $\gamma_0$ be a large positive number appearing in Theorem \ref{thm:4.1.1}
and Theorem \ref{thm:4.2.1} and we fix $\gamma_1$ in such a way that $\gamma_1 > \gamma_0$. 
Let $T_{v}(t)\bw_{0\pm}$ be defined by setting
$$T_{v}(t)\bw_{0\pm} = e^{-(2\gamma_1-\Delta)t}\tilde\bw_{0\pm}
= \CF^{-1}[e^{-(|\xi|^2+2\gamma_1)t}\CF[\tilde\bw_{0\pm}](\xi)].$$
In particular, $T_{v}(0)\bw_{0\pm} = \bw_{0\pm}$ in $\Omega_\pm$,
$T_{v}(0)\bw_{0\pm} = \CE_\mp(\bw_{0\pm})$ in $\Omega, $ and 
\begin{equation}\label{4.5.26} 
\|e^{\gamma_1 t}T_{v}(t)\bw_{0\pm}\|_{H^1_p((0, \infty), L_q(\Omega))} 
+ \|e^{\gamma_1 t}T_{v}(t)\bw_{0\pm}\|_{L_p((0, \infty), H^2_q(\Omega))} 
\leq CB.
\end{equation}
We also construct a similar operator for $H_h$. Let $\bW$, $P$, and $\Xi$
be solutions of the equations:
\begin{alignat*}2
\rho\pd_t\bW + \lambda_0\bW - \DV\bT(\bW, P)=0, \quad
\dv\bW = 0&&\quad&\text{in $\dot\Omega\times(0, \infty)$}, \\
\pd_t\Xi + \lambda_0\Xi - \bW\cdot\bn = 0&&\quad&\text{on $\Gamma\times(0, \infty)$}, \\
[[\bT(\bW, P)\bn]] - \sigma(\Delta_\Gamma\Xi )\bn=0,
\quad[[\bW]] = 0&&\quad&\text{on $\Gamma\times(0, \infty)$}, \\
\bW_\pm = 0& &\quad&\text{on $S_\pm\times(0, \infty)$}, \\
(\bW, \Xi)|_{t=0} = (0, h_0)&&\quad&\text{in $\dot\Omega\times\Gamma$}.
\end{alignat*}
For large $\lambda_0 > 0$ we know the unique existence of $\bW$, $P$, and $\Xi$ with
\begin{align*}
\bW &\in H^1_p((0, \infty), L_q(\dot\Omega)^N) \cap L_p((0, \infty), H^2_q(\dot\Omega)^N),
\\
\Xi &\in H^1_p((0, \infty), W^{2-1/q}_q(\Gamma)) 
\cap L_p((0, \infty), W^{3-1/q}_q(\Gamma)).
\end{align*}
possessing the estimate:
\begin{align*}
&\|e^{\gamma_1 t}\bW\|_{H^1_p((0, \infty), L_q(\dot\Omega))} 
+ \|e^{\gamma_1 t}\bW\|_{L_p((0, \infty), H^2_q(\dot\Omega))} 
+ \|e^{\gamma_1 t}\bW\|_{L_\infty((0, \infty), W^{2(1-1/p)}_{q,p}(\dot\Omega))}\\
&\quad + \|e^{\gamma_1 t}\Xi\|_{H^1_p((0, \infty), W^{2-1/q}_q(\Gamma))} 
+ \|e^{\gamma_1 t}\Xi\|_{L_p((0, \infty), W^{3-1/q}_q(\Gamma))} 
+\|e^{\gamma_1 t}\Xi\|_{H^1_\infty(0, \infty), W^{1-1/q}_q(\Gamma))} \\
&\qquad \leq C\|h_0\|_{B^{3-1/p-1/q}_{q,p}(\Gamma)} \leq C\epsilon.
\end{align*}
Let $T_h(t)h_0 = H_\Xi(x, t)$, where $H_\Xi$ is a unique solution of
\eqref{eq:harmonic} with $h = \Xi$, and then by \eqref{1.3.2} we have
\begin{equation}\label{4.5.27}\begin{aligned}
\|e^{\gamma_1 t}T_h(\cdot)h_0&\|_{H^1_p((0, \infty), H^2_q(\Omega))}
+ \|e^{\gamma_1 t}T(\cdot)h_0\|_{L_p((0, \infty), H^3_q(\Omega))} \\
&+\|e^{\gamma_1 t}T_h(\cdot)h_0\|_{L_\infty((0, \infty), B^{3-1/p}_{q,p}(\Omega))}
+\|e^{\gamma_1 t}T_h(\cdot)h_0\|_{H^1_\infty((0, \infty), H^1_q(\Omega))}
\leq C\epsilon.
\end{aligned}\end{equation}
In what follows, a generic constant $C$ depends on $\gamma_1$ when we use \eqref{4.5.26}
and \eqref{4.5.27}, but $\gamma_1$ is eventually fixed in such a way that the estimates given in 
Theorem \ref{thm:4.1.1} and Theorem \ref{thm:4.2.1} hold, and so 
we do not mention the dependence on $\gamma_1$. 

Given function, $f(t)$, defined on $(0, T)$, an extension, $e_T[f]$, of $f$ is defined
by setting
$$e_T[f] = \begin{cases} 0  \quad&\text{for $t < 0$}, \\
f(t) \quad&\text{for $0 < t < T$}, \\
f(2T-t) \quad&\text{for $T < t < 2T$}, \\
0 \quad&\text{for $t > 2T$}.
\end{cases}$$
Obviously, $e_T[f] = f$ for $t \in (0, T)$ and $e_T[f]$ vanishes for $t \not\in (0, 2T)$.
Moreover, if $f|_{t=0}=0$, then 
\begin{equation}\label{4.5.29}
\pd_te_T[f] = \begin{cases} 0  \quad&\text{for $t < 0$}, \\
\pd_tf(t) \quad&\text{for $0 < t < T$}, \\
-(\pd_tf)(2T-t) \quad&\text{for $T < t < 2T$}, \\
0 \quad&\text{for $t > 2T$}.
\end{cases}\end{equation}
If $f \in L_p((0, T), X)$ with some Banach space $X$ and $f|_{t=0} = 0$, then 
\begin{align*}
\|e_T[f]\|_{L_p(\BR, X)} & \leq 2\|f\|_{L_p((0, T), X)} \quad(1\leq p \leq \infty), \\
\|e_T[f]\|_{L_p(\BR, X)} & \leq 2T^{1/p}\|f\|_{L_\infty((0, T), X)} \quad(1\leq p < \infty).
\end{align*}
Moreover, if $f|_{t=0}=0$, then $e_T[f](t) = \int^t_0\pd_te_T[f]\,ds$, and so
$$\|e_T[f]\|_{L_\infty(\BR, X)} \leq 2(2T)^{1/p'}\|f\|_{L_p((0, T), X)}
\quad(1 < p < \infty, \enskip p' = p/(p-1)), 
$$
because $e_T[f]$ vanishes for $t \not\in (0, 2T)$.

Let $\psi \in C^\infty(\BR)$ which equals one for $t > -1$ and zero for
$t < -2$. Under these preparations, for $\bw \in \{\bu, \bG\}$ and 
$H_h$, we  define the extensions
$\CE_1[\bw_\pm]$, $\CE_1[tr[\bw]]$.   and $\CE_2[H_h]$ by letting 
\begin{equation}\label{4.5.30} \begin{aligned}
\CE_1[\bw_\pm] &= e_T[\bu_\pm-T_v(t)\bw_{0\pm}] + \psi(t)T_v(|t|)\bw_{0\pm},
\\
\CE_1[tr[\bw]] & = e_T[tr[\bw] - T_v(t)tr[\bw_0]]
+ \psi(t)T_v(|t|)tr[\bw_0], \\
\CE_2[H_h] & = e_T[H_h - T_h(t)h_0] + \psi(t)T_h(|t|)h_0.
\end{aligned}\end{equation}
Here, we have set $tr[\bw_0] = \tilde\bw_{0+} - \tilde \bw_{0-}$. 
Notice that $\bw_\pm-T_v(t)\bw_{0\pm} = 0$ for $t=0$,
$tr[\bw] - T_v(t)tr[\bw_0]$ for $t=0$,  and 
$H_h - T_h(t)h_0 = 0$ for $t=0$.  Obviously, 
\begin{equation}\label{4.5.31}
\CE_1[\bu_\pm] =\bu_\pm, \quad 
\CE_1[tr[\bw]] = tr[\bw], \quad \CE_2[H_h] = H_h
\quad\text{for $0 < t < T$}.
\end{equation}
By \eqref{4.5.10},  \eqref{4.5.26} and \eqref{4.5.27}, we have 
\begin{equation}\label{est:1}\begin{aligned}
&\|e^{-\gamma t}\CE_1[\bw]\|_{L_p(\BR, H^2_q(\dot\Omega))}
+ \|e^{-\gamma t}\pd_t\CE_1[\bw]\|_{L_p(\BR, L_q(\dot\Omega))}
\leq C(e^{2(\gamma-\gamma_1)}B + L), \\
&\|e^{-\gamma t}\CE_1[\bw]\|_{L_\infty(\BR, H^1_q(\dot\Omega))}
\leq Ce^{2(\gamma-\gamma_1)}B + C_s(B+L)T^{\frac{s}{p'(1+s)}}, \\
&\|\CE_2[H_h]\|_{L_p(\BR, H^3_q(\dot\Omega))} 
+ \|\CE_2[H_h]\|_{H^1_p(\BR, H^2_q(\dot\Omega))} 
+ \|\pd_t\CE_2[H_h]\|_{L_\infty(\BR, H^1_q(\dot\Omega))}
\leq C(\epsilon +L) \leq 2CL,\\
&\|\CE_2[H_h]\|_{L_\infty(\BR, H^2_q(\dot\Omega))}
 \leq C(\epsilon  + LT^{1/p'})
 \leq 2CLT^{1/p'}.
\end{aligned}\end{equation}
where $\bw \in \{\bu, \bG, tr[\bu], tr[\bG]\}$. 
In fact, the first and  third  inequalites in \eqref{est:1} follow from \eqref{4.5.26}, \eqref{4.5.27} and 
\eqref{under:1}. To prove the second inequality in \eqref{4.5.27}, we observe that
\begin{align*}
\|e_T[\bw-T_v(t)\bw_0]\|_{L_q(\dot\Omega)}
& \leq \int^t_0\|\pd_se_T[\bw-T_v(t)\bw_0]\|_{L_q(\dot\Omega)}\,ds \\
&\leq T^{1/{p'}}(\|\pd_t\bw\|_{L_p((0, T), L_q(\dot\Omega))}
+ \|\pd_tT_v(t)\bw_0\|_{L_p(0, 2T), L_q(\dot\Omega))} 
\leq T^{1/p'}(L+B);\\
\|e_T[\bw-T_v(t)\bw_0]\|_{W^{1+s}_q(\dot\Omega)}
&
\leq C_s\||e_T[\bw-T_v(t)\bw_0]\|_{B^{2(1-1/p)}_{q,p}(\dot\Omega)}
\leq C_s(B+L),
\end{align*}
for any $s \in (0, 1-2/p)$. Thus, using the interpolation
inequality: $\|v\|_{H^1_q(\dot\Omega)}
\leq C_s\|v\|_{L_q(\dot\Omega)}^{s/(1+s)}\|v\|_{W^{1+s}_q(\dot\Omega)}^{1/(1+s)}$, 
we have the second inequality in \eqref{est:1}. 
By \eqref{4.5.27} and \eqref{4.5.14}, 
\begin{align*}
\|\psi(t)T_h(|t|)h_0\|_{L_\infty(\BR, H^1_\infty(\Omega))}
&\leq C\epsilon, \\
\|e_T[H_h - T_h(t)h_0]\|_{L_\infty(\BR, H^1_\infty(\Omega))}
&\leq C(\|H_h\|_{L_\infty((0, T), H^2_q(\Omega))} 
+ \|T(\cdot)h_0\|_{L_\infty((0, \infty), H^2_q(\Omega))})\\
& \leq C(\epsilon + T^{1/{p'}}L), 
\end{align*}
and so we have the last inequality in \eqref{est:1}. 

Choosing $\epsilon>0$ and $T > 0$ small enough in the last inequality in
\eqref{est:1}, we may 
assume that 
\begin{equation}\label{4.5.34}
\sup_{t\in \BR} \|\CE_2[H_h]\|_{H^1_\infty(\Omega)} \leq \delta.
\end{equation}
And also, 
\begin{equation}\label{4.5.34*}
\|\CE_2[H_h]\|_{L_\infty(\BR, H^2_q(\dot\Omega))} \leq CLT^{1/p'}, \quad 
\|\CE_2[H_h]\|_{L_\infty(\BR, H^2_q(\dot\Omega))} \leq 1.
\end{equation}
%%%%%%%%%%%%%%

To estimate $H^{1/2}_p(\BR, L_q(\dot\Omega))$ norm,   we use
the following  lemmata.
%\begin{lem}\label{lem:4.5.1} Let $1 < p < \infty$, $N < q < \infty$
%and $0 < T \leq 1$.  Let $f \in H^1_p(\BR, H^1_q(\dot\Omega))$
%and $g \in H^{1/2}(\BR, L_q(\dot\Omega)) \cap L_p(\BR, H^1_q(\dot\Omega))$.
%If $f$ vanishes for $t \not\in (0, 2T)$, then we have
%\begin{align*}
%&\|fg\|_{H^{1/2}_p(\BR, L_q(\dot\Omega))}
%+ \|fg\|_{L_p(\BR, H^1_q(\dot\Omega))} \\
%& \leq C\{T^{1/{p'}}\|\pd_tf\|_{L_p(\BR, H^1_q(\dot\Omega))}
%+ T^{(q-N)/(pq)}
%\|\pd_tf\|_{L_\infty(\BR, L_q(\dot\Omega))}^{1-N/(2q))}
%\|\pd_tf\|_{L_p(\BR, H^1_q(\dot\Omega))}^{N/(2q)}\}\\
%&\qquad\times(\|g\|_{H^{1/2}_p(\BR, L_q(\dot\Omega))}
%+ \|g\|_{L_p(\BR, H^1_q(\dot\Omega))}). 
%\end{align*}
%\end{lem}
%\begin{proof} For a proof, see Shibata-Shimizu \cite[Lemma 2.6]{SS10}.
%\end{proof}
\begin{lem}\label{lem:4.5.3} Let $1 < p < \infty$ and $N < q < \infty$.
Let 
\begin{align*}
f  \in L_\infty(\BR, H^1_q(\dot\Omega)) \cap H^1_\infty(\BR, L_q(\dot\Omega)), \quad 
g \in H^{1/2}_p(\BR, H^1_q(\dot\Omega)) \cap L_p(\BR, H^1_q(\dot\Omega)).
\end{align*}
Then, we have
\begin{align*}
&\|fg\|_{H^{1/2}_p(\BR, L_q(\dot\Omega))} + \|fg\|_{L_p(\BR, H^1_q(\dot\Omega))}\\
&
\leq C(\|\pd_tf\|_{L_\infty(\BR, L_q(\dot\Omega))}+ \|f\|_{L_\infty(\BR, H^1_q(\dot\Omega))})^{1/2}
\|f\|_{L_\infty(\BR, H^1_q(\dot\Omega))}^{1/2}
(\|g\|_{H^{1/2}_p(\BR, L_q(\dot\Omega))}
+ \|g\|_{L_p(\BR, H^1_q(\dot\Omega))}).
\end{align*}
\end{lem}
\begin{proof} To prove Lemma \ref{lem:4.5.3}, we use the fact that
$$H^{1/2}_p(\BR, L_q(\dot\Omega)) \cap L_p(\BR, H^{1/2}_q(\dot\Omega))
=(L_p(\BR, L_q(\dot\Omega)), 
H^1_p(\BR, L_q(\dot\Omega)) \cap L_p(\BR, H^1_q(\dot\Omega)))_{[1/2]}
$$
where $(\cdot, \cdot)_{[1/2]}$ denotes a  complex interpolation functor of order
$1/2$. We have
\begin{align*}
&\|fg\|_{H^1_p(\BR, L_q(\dot\Omega))} + 
\|fg\|_{L_p(\BR, H^1_q(\dot\Omega))}\\
&\quad \leq C(\|\pd_tf\|_{L_\infty(\BR, L_q(\dot\Omega))}
\|g\|_{L_p((\BR, H^1_q(\dot\Omega))}
+ \|f\|_{L_\infty(\BR, H^1_q(\dot\Omega))}
\|g\|_{H^1_p(\BR, L_q(\dot\Omega))})\\
&\quad \leq C(\|\pd_tf\|_{L_\infty(\BR, L_q(\dot\Omega))}
+ \|f\|_{L_\infty(\BR, H^1_q(\dot\Omega))})
(\|g\|_{L_p((\BR, H^1_q(\dot\Omega))}
+\|g\|_{H^1_p(\BR, L_q(\dot\Omega))}).
\end{align*}
Moreover, 
$$\|fg\|_{L_p(\BR, L_q(\dot\Omega))}
\leq C\|f\|_{L_p(\BR, H^1_q(\dot\Omega))}\|g\|_{L_p(\BR, L_q(\dot\Omega))}.
$$
Thus, by complex interpolation, we have
\begin{align*}
&\|fg\|_{H^{1/2}_p(\BR, L_q(\dot\Omega))} + \|fg\|_{L_p(\BR, H^{1/2}_q(\dot\Omega))}\\
&
\leq C(\|\pd_tf\|_{L_\infty(\BR, L_q(\dot\Omega))}+ \|f\|_{L_\infty(\BR, H^1_q(\dot\Omega))})^{1/2}
\|f\|_{L_\infty(\BR, H^1_q(\dot\Omega))}^{1/2}
(\|g\|_{H^{1/2}_p(\BR, L_q(\dot\Omega))}
+ \|g\|_{L_p(\BR, H^{1/2}_q(\dot\Omega))}).
\end{align*}
Moreover, we have
$$\|fg\|_{L_p(\BR, H^1_q(\dot\Omega))} \leq C\|f\|_{L_\infty(\BR, H^1_q(\dot\Omega))}
\|g\|_{L_p(\BR, H^1_q(\dot\Omega))}.
$$
Thus, combining these two inequalities give the required estimate, which completes
 the proof of Lemma \ref{lem:4.5.3}. 
\end{proof}

\begin{lem}\label{lem:4.5.4}
Let $1 < p, q < \infty$.  Then, 
$$H^1_p(\BR, L_q(\dot\Omega)) \cap L_p(\BR, H^2_q(\dot\Omega))
\subset H^{1/2}_p(\BR, H^1_q(\dot\Omega))$$
and 
$$\|u\|_{H^{1/2}_p(\BR, H^1_q(\dot\Omega))} \leq C(\|u\|_{L_p(\BR, H^2_q(\dot\Omega))}
+ \|\pd_tu\|_{L_p(\BR, L_q(\dot\Omega))}).
$$
\end{lem}
\begin{proof} For a proof, see Shibata \cite[Proposition 1]{S2}. 
\end{proof}

We now estimate $\bh_1(\bu, \bG, H_h)$.  In view of \eqref{g.3*},
we define an extension of $\bh_1(\bu, \bG, H_h)$ to the whole time interval
$\BR$ by setting $\tilde\bh_1(\bu, \bG, H_h) = A^1 + A^2 + A^3$ with
\begin{equation}\label{eq:4.1*}\begin{aligned}
A^1 & = \bV^1_\bh(\cdot, \bar\nabla \CE_2[H_h])\bar\nabla \CE_2[H_h]
\otimes\nabla\CE_1[tr[\bu]], \\
A^2 & = a(\cdot)\CE_1[tr[\bG]]\otimes \CE_1[tr[\bG]]+ \bV^2_\bh(\cdot, \bar\nabla \CE_2[H_h])
\bar\nabla\CE_2[H_h])\otimes \CE_1[tr[\bG]]\otimes \CE_1[tr[\bG]], \\
A^3&= \bV_s(\cdot, \bar\nabla \CE_2[H_h])\bar\nabla \CE_2[H_h]\otimes
\bar\nabla^2\CE_2[H_h].
\end{aligned}\end{equation}
Obviously, $\tilde\bh_1(\bu, \bG, H_h) = \bh_1(\bu, \bG, H_h)$ for $t \in (0, T)$.
To estimate $A^1$, for notational simplicity we set 
$\CV^1 = \bV^1_\bh(\cdot, \bar\nabla \CE_2[H_h])\bar\nabla \CE_2[H_h]$.  
By \eqref{4.5.34} and \eqref{est:1}, 
\begin{align*}
\|\pd_t\CV^1\|_{L_\infty(\BR, L_q(\dot\Omega))}& \leq C\|\pd_t\CE_2[H_h]\|_{L_\infty(\BR, H^1_q(\dot\Omega))}
\leq CL, \\
\|\CV^1\|_{L_\infty(\BR, H^1_q(\dot\Omega))}&
\leq C\|\CE_2[H_h]\|_{L_\infty(\BR, H^2_q(\dot\Omega))} \leq CLT^{1/p'},
\end{align*}
and so, we have
\begin{equation}\label{est:2}
(\|\pd_t\CV^1\|_{L_\infty(\BR,  L_q(\dot\Omega))}
+ \|\CV^1\|_{L_\infty(\BR, H^1_q(\dot\Omega))})^{1/2}
 \|\CV^1\|_{L_\infty(\BR, H^1_q(\dot\Omega))}^{1/2} \leq CLT^{1/(2p')}.
 \end{equation}
 Thus, by \eqref{est:1}, \eqref{est:2}, Lemma \ref{lem:4.5.3} and Lemma \ref{lem:4.5.4}, we have
 \begin{equation}\label{4.5.44}\begin{aligned}
& \|e^{-\gamma t}A^1\|_{H^{1/2}_p(\BR, L_q(\dot\Omega)}
 + \|e^{-\gamma t}A^1\|_{L_p(\BR, H^1_q(\dot\Omega))} \\
 & \quad \leq CLT^{1/(2p')}(\|e^{-\gamma t}\nabla\CE_1[tr[\bu]]\|_{H^{1/2}_p(\BR, L_q(\dot\Omega))}
 + \|e^{-\gamma t} \nabla\CE_1[tr[\bu]]\|_{L_p(\BR, H^1_q(\dot\Omega))}\\
 &\quad  \leq CT^{1/(2p')}L(e^{2(\gamma-\gamma_1)}B + L).
 \end{aligned}\end{equation}
 %%%%%%%%%%%%%%%
 %%%%%%%%%%%%%%%%

Since 
\begin{align*}
\|e^{-\gamma t}\bar\nabla^2\CE_2[H_h]\|_{H^{1/2}_p(\BR, L_q(\dot\Omega))}
&\leq \|e^{-\gamma t}\CE_2[H_h]\|_{H^1_p(\BR, H^2_q(\dot\Omega))}
\leq C(e^{2(\gamma-\gamma_1)}\epsilon  + L) ; \\
\|e^{-\gamma t}\bar\nabla^2\CE_2[H_h]\|_{L_p(\BR, H^1_q(\dot\Omega))}
&\leq C\|e^{-\gamma t}\CE_2[H_h]\|_{L_p(\BR, H^3_q(\dot\Omega))}
\leq C(e^{2(\gamma-\gamma_1)}\epsilon + L) 
\end{align*}
as follows from \eqref{under:1}, the third formula of \eqref{4.5.30} and 
\eqref{4.5.27}, employing the same argument as in proving \eqref{4.5.44},
we have 
\begin{equation}\label{4.5.45}
\|e^{-\gamma t} A^3 \|_{H^{1/2}_p(\BR, L_q(\dot\Omega))}
+ \|e^{-\gamma t} A^3\|_{L_p(\BR, H^1_q(\dot\Omega))}
\leq  CT^{1/(2p')}L(e^{2(\gamma-\gamma_1)}\epsilon + L).
\end{equation}

We now estimate $A^2$. For this purpose we use the following
esitmate which follows from  complex interpolation theory: 
\begin{equation}\label{complex:1}
\|f\|_{H^{1/2}_p(\BR, L_q(\dot\Omega))}
\leq C\|f\|_{H^1_p(\BR, L_q(\dot\Omega))}^{1/2}
\|f\|_{L_p(\BR, L_q(\dot\Omega))}^{1/2}.
\end{equation}
Let 
\begin{align*}
A^2_1 & = \CE_1[tr[\bG]]\otimes  \CE_1[tr[\bG]], \quad 
A^2_2  = \bV^2_\bh(\cdot, \bar\nabla \CE_2[H_h])
\bar\nabla\CE_2[H_h]\otimes A^2_1. 
\end{align*}
We further divide $A^2_1$ into $A^2_1 = \sum_{j=1}^4 A^2_{1j}$ with 
$A^2_1 = A^2_{11} + A^2_{12} + A^2_{21} + A^2_{22}$, where 
\allowdisplaybreaks
\begin{align*}
A^2_{11} &= \CA_1\otimes \CA_1, \quad A^2_{12} = \CA_1\otimes \CA_2, 
\quad A^2_{21} = \CA_2\otimes \CA_1, \quad A^2_{22} = \CA_2\otimes \CA_2, \\
\CA_1 & = 
\psi(t)T_v(|t|)tr[\bG_{0}], \quad 
\CA_2  =  e_T[tr[\bG] - T_v(t)tr[\bG_{0}]]. 
\end{align*}
Using \eqref{4.5.1}, we have
\begin{equation}\label{semi:4.-1}\begin{aligned}
\|e^{-\gamma t}A^2_{11}\|_{H^i_p(\BR, L_q(\dot\Omega))} & \leq C\|e^{-\gamma t}\CA_1\|_{H^i_p(\BR, L_q(\dot\Omega))}
\|\CA_1\|_{L_\infty(\BR, H^1_q(\dot\Omega))} \quad(i=0,1), \\
\|e^{-\gamma t}A^2_{12}\|_{H^1_p(\BR, L_q(\dot\Omega))} & 
\leq C(\|\CA_1\|_{H^1_p(\BR, L_q(\dot\Omega))}
\|\CA_2\|_{L_\infty(\BR, H^1_q(\dot\Omega))} \\
&\qquad +
\|\CA_1\|_{L_\infty(\BR, H^1_q(\dot\Omega))}
\|\CA_2\|_{H^1_p(\BR, L_q(\dot\Omega))}),\\
\|e^{-\gamma t}A^2_{12}\|_{L_p(\BR, L_q(\dot\Omega))} & \leq C\|\CA_1\|_{L_\infty(\BR, H^1_q(\dot\Omega))}
\|\CA_2\|_{L_p(\BR, L_q(\dot\Omega))}, \\
\|e^{-\gamma t}A^2_{22}\|_{H^i_p(\BR, L_q(\dot\Omega))}
& \leq C\|\CA_2\|_{H^i_p(\BR, L_q(\dot\Omega))}
\|\CA_2\|_{L_\infty(\BR, H^1_q(\dot\Omega))} \quad(i=0,1)
\end{aligned}\end{equation}
where we have set $H^0_p = L_p$ and used the fact that $|e^{-\gamma t}\CA_2|
\leq |\CA_2|$, which follows from $\CA_2=0$ for $t \not\in (0, 2T)$.

By \eqref{trace:4}, \eqref{4.5.26}, \eqref{4.5.29}, \eqref{under:1}, and \eqref{4.5.10}, 
we have
\allowdisplaybreaks
\begin{align}
&\|e^{-\gamma t}\CA_1\|_{H^i_p(\BR, L_q(\dot\Omega))} \leq Ce^{2(\gamma-\gamma_1)}B; \quad 
\|\CA_1\|_{H^i_p(\BR, L_q(\dot\Omega))} \leq CB;  
\quad \|\CA_1\|_{L_\infty(\BR, H^1_q(\dot\Omega))} \leq CB; 
\nonumber\\
&\|\CA_2\|_{H^1_p(\BR, L_q(\dot\Omega))}
 \leq C(L+B);  \quad \|\CA_2\|_{L_\infty(\BR, H^1_q(\dot\Omega))} \leq C(L+B); 
 \label{semi:4.0} \\
&\|\CA_2\|_{L_p(\BR, L_q(\dot\Omega))}
\leq CT^{1/p}(\|tr[\bG]\|_{L_\infty((0, T), L_q(\dot\Omega))}
+ \|T_v(\cdot)tr[\bG_0]\|_{L_\infty((0, T), L_q(\dot\Omega))}) 
\leq C(L+B)T^{1/p}. \nonumber
\end{align}
%where $s$ is a positive constant smaller than $1-2/p$. 
Notice that $A^2_{12}$ and $A^2_{21}$ have the same estimate.  In view of \eqref{complex:1}, 
combining estimates in \eqref{semi:4.-1} and \eqref{semi:4.0} gives that 
\begin{equation}\label{semi:4.1}\begin{aligned}
\|e^{-\gamma t}A^2_1\|_{H^{1/2}_p(\BR, L_q(\dot\Omega))}
&\leq 
C(e^{2(\gamma-\gamma_1)}B^2 +(B(L+B))^{1/2}(B(L+B)T^{1/p})^{1/2} 
+(B+L)^2T^{1/(2p)})
\\
&\leq C(e^{2(\gamma-\gamma_1)}B^2 + (L+B)^2T^{1/(2p)}). 
\end{aligned}\end{equation}

And also, by \eqref{4.5.1}
\begin{align*}
\|e^{-\gamma t}A^2_1\|_{L_p(\BR, H^1_q(\dot\Omega))}
& \leq C\{\|e^{-\gamma t}\CA_1\|_{L_p(\BR, H^1_q(\dot\Omega))}
\|\CA_1\|_{L_\infty(\BR, H^1_q(\dot\Omega))} 
+ 2\|\CA_1\|_{L_\infty(\BR, H^1_q(\dot\Omega))}\|\CA_2\|_{L_p(\BR, H^1_q(\dot\Omega))}
\\
&\quad+ \|\CA_2\|_{L_p(\BR, H^1_q(\dot\Omega))}\|\CA_2\|_{L_\infty(\BR, H^1_q(\dot\Omega))}\}.
\end{align*}
By \eqref{4.5.26}, \eqref{4.5.29}, \eqref{under:1}, and \eqref{4.5.10} we have
\begin{equation}\label{semi:4.-1}\begin{aligned}
\|e^{-\gamma t} \CA_1\|_{L_p(\BR, H^1_q(\dot\Omega))} &\leq Ce^{2(\gamma-\gamma_1)}B; \\
\|\CA_2\|_{L_p(\BR, H^1_q(\dot\Omega))}
& \leq T^{1/p}\|\CA_2\|_{L_\infty((0, 2T), H^1_q(\dot\Omega))} \\
&\leq T^{1/p}(\|tr\bG\|_{L_\infty((0, T), H^1_q(\dot\Omega))}
+ \|T_v(\cdot)tr[\bG_0]\|_{L_\infty((0, T), H^1_q(\dot\Omega))}\\
&\leq CT^{1/p}(L+B).
\end{aligned}\end{equation}
Using  \eqref{semi:4.0} and \eqref{semi:4.-1}, we have  
 \begin{equation}\label{semi:4.2}
\|e^{-\gamma t}A^2_1\|_{L_p(\BR, H^1_q(\dot\Omega))}
\leq C(e^{2(\gamma-\gamma_1)}B^2 + (L+B)^2T^{1/p}).
\end{equation}
%%%%%%%%% ここまで  12/5%%%%%%%%%%%%%%%%%%%%

Moeover, by Lemma \ref{lem:4.5.3} and \eqref{est:2}, we have
\begin{align*}
\|e^{-\gamma t}A^2_2\|_{H^{1/2}_p(\BR, L_q(\dot\Omega))} + 
\|A^2_2\|_{L_p(\BR, H^1_q(\dot\Omega))} \leq
CLT^{1/(2p')}(\|A^2_1\|_{H^{1/2}_p(\BR, L_q(\dot\Omega))}
+ \|A^2_1\|_{L_p(\BR, H^1_q(\dot\Omega))}),
\end{align*}
which, combined with \eqref{semi:4.1} and \eqref{semi:4.2}, 
gives that
\begin{equation}\label{semi:4.6} \begin{aligned}
&\|e^{-\gamma t}A^2\|_{H^{1/2}_p(\BR, L_q(\dot\Omega))} + 
\|e^{-\gamma t}A^2\|_{L_p(\BR, H^1_q(\dot\Omega))} \\
&\quad \leq C(e^{2(\gamma-\gamma_1)}B^2 + (L+B)^2T^{1/(2p)} + 
LT^{1/(2p')}(e^{2(\gamma-\gamma_1)}B^2 + (L+B)^2T^{1/(2p)})) \\
& \quad \leq C(e^{2(\gamma-\gamma_1)}B^2 + L(B^2+L^2)T^{1/(2p)}),
\end{aligned}\end{equation}
where we assume that $LT^{1/2p'} \leq 1$. 

%%%%%%%%%%%%%%%%%%
Combining \eqref{4.5.44}, \eqref{4.5.45}, and \eqref{semi:4.6} yields that
\begin{equation}\label{semi:4.7}\begin{aligned}
&\|\tilde\bh_1(\bu, \bG, H_h)\|_{H^{1/2}_p(\BR, L_q(\dot\Omega))}
+ \|\tilde\bh_1(\bu, \bG, H_h)\|_{Lp(\BR, H^1_q(\dot\Omega))}
\\
&\quad 
\leq C(e^{2(\gamma-\gamma_1)}B^2 +L(L^2+B^2+e^{2(\gamma-\gamma_1)}B)T^{1/(2p)}),
\end{aligned}\end{equation}
where we have used the facts: $1/(2p) < 1/(2p')$,  $\epsilon e^{2(\gamma-\gamma_1)}
< Be^{2(\gamma-\gamma_1)}$, and $L \leq L^2$.

We finally consider $g(\bu, H_h)$ and $\bg(\bu, H_h)$. 
In view of \eqref{div:3*}, we set
\begin{equation}\label{4.5.32}
\tilde g(\bu, H_h)  = \CG_1(\bar\nabla \CE_2[H_h])\bar\nabla \CE_2[H_h]\otimes\nabla
\CE_1[\bu], 
\quad
\tilde\bg(\bu, H_h) = 
\CG_2(\bar\nabla \CE_2[H_h])\bar\nabla \CE_2[H_h]\otimes
\CE_1[\bu].
\end{equation}
Here, we have set $\CE_1[\bu] = \CE_1[\bu_\pm]$ for $x \in \Omega_\pm$. 
Obviously, 
$$\tilde g(\bu, H_h) = g(\bu, H_h), \quad \tilde\bg(\bu, H_h) = \bg(\bu, H_h)
\quad\text{for $t \in (0, T)$}
$$
and $\dv \tilde\bg(\bu, H_h) = \tilde g(\bu, H_h)$ as follows from \eqref{div:2},
\eqref{div:3} and \eqref{g:1}.   Employing the same argument as
that in proving \eqref{4.5.44}, we have
\begin{equation}\label{semi:4.7*}
\|e^{-\gamma t}\tilde g(\bu, H_h)\|_{H^{1/2}_p(\BR, L_q(\dot\Omega))}
+ \|e^{-\gamma t}\tilde g (\bu, H_h)\|_{L_p(\BR, H^1_q(\dot\Omega))}
\leq CT^{1/(2p')}L(L+e^{2(\gamma-\gamma_1)}B).
\end{equation}

By \eqref{4.5.34} we have
\begin{align*}
\|\pd_t\tilde\bg(\bu, H_h)\|_{L_q(\dot\Omega)}
& \leq C\{\|\bar\nabla\CE_2[H_h]\|_{H^1_q(\Omega)}\|\pd_t\CE_1[\bu]\|_{L_q(\dot\Omega)}
+ \|\pd_t\bar\nabla\CE_2[H_h]\|_{L_q(\Omega)}\|\CE_1[\bv]\|_{H^1_q(\dot\Omega)}\}
\end{align*}
Thus, 
using \eqref{4.5.29}, \eqref{4.5.1}, \eqref{4.5.12}, 
\eqref{4.5.13}, \eqref{4.5.14}, \eqref{4.5.26}, \eqref{4.5.27}, and
\eqref{4.5.34},  we have
\begin{equation}\label{4.5.36}\begin{aligned}
&\|e^{-\gamma t}\pd_t\tilde\bg(\bu, H_h)\|_{L_p(\BR, L_q(\Omega))} 
\\
&\quad\leq C(\|H_h\|_{L_\infty((0, T), H^2_q(\Omega))} + 
\|T_h(\cdot)h_0\|_{L_\infty((0, \infty), H^2_q(\Omega))}) \\
&\qquad\times(\|\pd_t\bu\|_{L_p((0, T), L_q(\Omega))}
+ \|e^{-\gamma t}T_v(|\cdot|)\bu_0\|_{H^1_p((-2, \infty), L_q(\dot\Omega))}) \\
&\quad+ (T^{1/p}\|\pd_t H_h\|_{L_\infty((0, T), H^1_q(\Omega))}
+ \|\pd_t T_h(\cdot)h_0\|_{L_p((0, \infty), H^1_q(\Omega))}) \\
&\qquad\times(\|\bu\|_{L_\infty((0, T), H^1_q(\dot\Omega))}
+ \|e^{-\gamma t}T_v(|\cdot|)\bu_0\|_{L_\infty((-2, \infty), H^1_q(\dot\Omega))})
\\
&\quad\leq C\{ \epsilon +LT^{1/p'})(L+e^{2(\gamma-\gamma_1)}B)
+(T^{1/p}L+\epsilon)(L+e^{2(\gamma-\gamma_1)}B)\} \\
&\qquad \leq C
L(L+e^{2(\gamma-\gamma_1)}B)T^{1/p}.
\end{aligned}\end{equation}
%%%%%%%%%%%%
We now apply Theorem \ref{thm:4.1.1} to equations \eqref{eq:4.1}
and use the estimate in Theorem \ref{thm:4.1.1} 
with $\gamma=\gamma_1$.
 And then, assuming that $1 \leq B \leq L$, 
 noting that  $s/(p'(1+s)) < 1/(2p) < 1/(2p')$ and using 
\eqref{4.5.8}, \eqref{4.5.24},  
\eqref{semi:4.7}, we have
\begin{equation}\label{semi:4.8}\begin{aligned}
&E^1_T(\bv) + \|\nabla\fq\|_{L_p((0, T), L_q(\dot\Omega))}
+ E^2_T(\rho) \\
&\quad \leq C(1+\gamma_1^{1/2})e^{\gamma_1 T^{1-1/p}}
\{B^2 + T^{-1/p}\epsilon +L^3T^{\frac{s}{p'(1+s)}}\}.
\end{aligned}\end{equation}
Here and in the following  $s \in (0, 1-2/p)$ and $\gamma_1$ are fixed, and so 
we do not take care of  the dependance of constants on $s$ and $\gamma_1$.  

By the third equation of \eqref{eq:4.1}, \eqref{4.4.3}, 
and \eqref{4.5.24},
we have
\begin{align*}
\|\pd_t\rho\|_{L_p((0, T), W^{1-1/q}_q(\Gamma))}
& \leq C\{\|\rho\|_{L_\infty((0, T), W^{2-1/q}_q(\Gamma))}B
+ \|\bv\|_{L_\infty((0, T), H^1_q(\dot\Omega))}
+ L(L+B)T^{1/p'}\} \\
& \leq C\{(\epsilon + T^{1/p'}E^2_T(\rho))B + B + E^1_T(\bv) + 
 L(L+B)T^{1/p'})\}
 \\
& \leq C\{B + E^1_T(\bv) + E^2_T(\rho) + L(L+B)T^{1/p'}\}, 
\end{align*}
where we used the facts that $\epsilon B \leq B$ and $T^{1/p'}B \leq T^{1/p'}L \leq 1$. 
which, combined with \eqref{semi:4.8}, gives that
\begin{align*}
&E^1_T(\bv) + E^2_T(\rho) + \|\pd_t\rho\|_{L_\infty((0, T), W^{1-1/q}_q(\Gamma))}
\\
&\quad\leq C[(1+\gamma_1^{1/2})e^{\gamma_1 T^{1-1/p}}\{B^2 + T^{-1/p}\epsilon
+ L^3T^{\frac{s}{p'(1+s)}}\} + B+L(L+B)T^{1/p'}].
\end{align*}
Noting that $0 < T=\epsilon < 1$ and $T^{-1/p}\epsilon = T^{1-1/p}
< 1 < B^2$, we have
\begin{equation}\label{semi:4.10}
E^1_T(\bv) + E^2_T(\rho) + \|\pd_t\rho\|_{L_\infty((0, T), W^{1-1/q}_q(\Gamma))}
\leq M_1(B^2 
 + L^3T^{\frac{s}{p'(1+s)}})
\end{equation}
for some positive constant $M_1$ depending on $s$ and $\gamma_1$
provided that $0 < T < 1$, $\epsilon=\kappa=T$, $T^{1/p'}L \leq 1$,  and 
$L > B \geq 1$.  

We now estimate $\bH$ by using Theorem \ref{thm:4.2.1} with the $\gamma_1$ given 
above.  
Let $\bff_2(\bu, \bG, H_h)$ be a nonlinear term given in 
\eqref{g:1*}.  Recalling the formula in \eqref{maineq:2*} and employing the same 
argument as that in proving  \eqref{4.5.8}, we have
\begin{equation}\label{4.5.8*}
\|\bff_2(\bu, \bG, H_h)\|_{L_p((0, T), L_q(\dot\Omega))}
\leq CT^{1/p}(L+B)^2.
\end{equation}
We next consider $\bh_2(\bu, \bG, H_h)$ and $h_3(\bu, \bG, H_h)$
 given in \eqref{maineq:6} and in \eqref{maineq:7}, respectively.  Let 
 $\tilde\bh_2(\bu, \bG, H_h)$ and $\tilde h_3(\bu, \bG, H_h)$ be their 
 extension to $\BR$ with respect to $t$ defined by setting 
\begin{equation}\label{eq:4.5.8} \begin{aligned}
 \tilde \bh_2(\bu, \bG, H_h) & = \bV^3_\bh(\cdot\bar\nabla\CE_2[H_h])
 \bar\nabla\CE_2[H_h]
\otimes \bar\nabla \CE_1[tr[\bu]] + b(y)\CE_1[tr[\bu]]\otimes \CE_1[tr[\bG]]
 \\
 & + \bV^4_\bh(\cdot, \bar\nabla\CE_2[H_h])\bar\nabla\CE_2[H_h]\otimes
\CE_1[tr[\bu]]\otimes \CE_1[tr[\bG]]; \\
 \tilde h_3(\bu, \bG, H_h) & =\mu\sum_{j,k=1}^N
 \tilde V_{0jk}(\nabla \CE_2[H_h])\nabla \CE_2[H_h]
 \frac{\pd}{\pd y_k}\CE_1[tr[\bu]]_j, 
 \end{aligned}\end{equation}
  Employing the same argument as in proving \eqref{4.5.44}, 
  we have 
  \begin{align}
  &\|e^{-\gamma t}(\tilde\bh_2(\bu, \bG, H_h), \tilde h_3(\bu, \bG, H_h))
  \|_{H^{1/2}_p(\BR, L_q(\dot\Omega))} +
  \|e^{-\gamma t}(\tilde\bh_2(\bu, \bG, H_h), \tilde h_3(\bu, \bG, H_h))
  \|_{L_p(\BR, H^1_q(\dot\Omega))} \nonumber  \\
  &\quad \leq 
  CT^{1/(2p')}L(e^{2(\gamma-\gamma_1)}B+L).
\label{4.5.37}
\end{align}

  %%%%%%%%%%%%%%%%%% ここまで 12.7 %%%%%%%%%%%%%%%

  We finally consider $k_1(\bG, H_\rho)$ and $\bk_2(\bG, H_\rho)$
  given in \eqref{maineq:8}. In view of \eqref{semi:4.10}, choosing
  $L$ so large that $M_1B^2 < L/2$ and 
  $T$ so small that $M_1L^3 T^{s/(p'(1+s))} \leq L/2$, we have
  \begin{equation}E^2_T(\rho) + \|\pd_t\rho\|_{L_\infty((0, T), 
  W^{1-1/q}_q(\Gamma))} \leq L.
  \label{small:4.1}
  \end{equation}
   In particular, we have 
  \begin{equation}\label{semi:11} \|\CE_2[H_\rho]\|_{L_\infty(\BR, H^2_q(\Omega)}
   \leq C(\epsilon + LT^{1/p'}).
   \end{equation}
  Thus, choosing $\epsilon= T$ so small,  we may also assume that
  \begin{equation}\label{4.10**}
  \sup_{t \in \BR}\|\CE_2[ H_\rho](\cdot, t)\|_{H^1_\infty(\Omega)} \leq 
  \delta.
   \end{equation}
   And also, we may assume that
   \begin{equation}\label{4.10***}
   \|\CE_2[H_\rho]\|_{L_\infty(\BR, H^2_q(\Omega)} \leq LT^{1/p'}, \quad 
   \|\CE_2[H_\rho]\|_{L_\infty(\BR, H^2_q(\Omega)} \leq 1.
   \end{equation}
   In view of \eqref{maineq:6*}, we define the
   extensions of $k_1(\bG, H_\rho)$ and $\bk_2(\bG, H_\rho)$ 
  by setting 
  \begin{equation}\label{eq:4.5.9}
  (\tilde k_1(\bG, H_\rho), \tilde \bk_2(\bG, H_\rho))
  = \bV^5_\bk(\cdot, \bar\nabla\CE_2[H_\rho])\bar\nabla \CE_2[H_\rho]
  \otimes\CE_1[tr[\bG]]. 
  \end{equation}
  Obviously, $(\tilde k_1(\bG, H_\rho), \tilde \bk_2(\bG, H_\rho))
=(k_1(\bG, H_\rho),  \bk_2(\bG, H_\rho))$ for $t \in (0, T)$. 
By \eqref{eq:4.5.1} and \eqref{4.10**} we have
\begin{align*}
&\|(\tilde k_1(\bG, H_\rho), \tilde\bk_2(\bG, H_\rho)\|_{H^2_q(\dot\Omega)} \\
&\quad \leq C\{\|\bar\nabla\CE_2[H_\rho]\|_{H^2_q(\dot\Omega)}
\|\CE_1[tr[\bG]]\|_{H^1_q(\dot\Omega)} 
+\|\bar\nabla\CE_2[H_\rho]\|_{H^1_q(\dot\Omega)}\|\nabla\CE_1[tr[\bG]]\|_{H^1_q(\dot\Omega)}\\
&\quad +\|\bar\nabla\CE_2[H_\rho]\|_{H^2_q(\dot\Omega)}(1+\|\bar\nabla\CE_2[H_\rho]\|_{H^1_q(\dot\Omega)})
\|\bar\nabla\CE_2[H_\rho]\|_{H^1_q(\dot\Omega)}\|\CE_1[tr[\bG]]\|_{H^1_q(\dot\Omega)}\}
\end{align*}
%%%%%%%%%%%%%%%
%By \eqref{4.5.27},  and \eqref{semi:11}, we have
%\begin{equation}\label{int:1}\|\bar\nabla^2\CE_2[H_\rho]\|_{L_\infty(\BR, L_q(\Omega))}
%\leq C(\epsilon + T^{1/p'}L).
%\end{equation}
%Recalling \eqref{4.5.30} and \eqref{4.5.26},  we see that 
%\begin{equation}\label{int:1*}\|e^{-\gamma t}\CE_1[tr[\bG]]\|_{L_\infty(\BR, H^1_q(\dot\Omega))}
%\leq C\{e^{2(\gamma-\gamma_1)}B + T^{\frac{s}{p'(1+s)}}(L+B)\}.
%\end{equation}
%n fact, 
%$$\|e^{-\gamma t}\CE_1[tr[\bG]]\|_{L_\infty(\BR, H^1_q(\dot\Omega))}
%\leq C(\|\bG-T_v(\cdot)\bG_0\|_{L_\infty((0, T), H^1_q(\dot\Omega))}+B\}.
%$$
%Moreover, 
%\begin{align*}
%\|\bG-T_v(\cdot)\bG_0\|_{L_q(\dot\Omega))} & 
%\leq \int^T_0\|\pd_t(\bG-T_v(\cdot)\bG_0)\|_{L_q(\dot\Omega)}\,dt
%\leq (L+B)T^{1/p'};\\
%\|\bG-T_v(\cdot)\bG_0\|_{W^{1+s}_q(\dot\Omega))} & \leq 
%\|\bG-T_v(\cdot)\bG_0\|_{B^{2(1-1/p)}_{q,p}(\dot\Omega)}
%\leq C(L+B).
%\end{align*}
%Since $\|u\|_{H^1_q(\dot\Omega)} \leq C_s\|u\|_{L_q(\dot\Omega)}^{s/(1+s)}
%\|u\|_{W^{1+s}_q(\dot\Omega))}^{1/(1+s)}$ for some $s \in (0, 1-2/p))$,  we have
%$$
%\|\bG-T_v(\cdot)\bG_0\|_{L_\infty((0, T), H^1_q(\dot\Omega))}
%\leq CT^{\frac{s}{p'(1+s)}}(L+B).
%$$
%This $s$ is the same as in \eqref{semi:4.8}. 
%Putting thse estimates together gives \eqref{int:1*}. 
%%%%%%%%%%%%%%%%%%%%%
By \eqref{4.10***}, \eqref{est:1}, \eqref{4.5.27}, and \eqref{small:4.1}, we have 
\begin{equation} \label{4.5.38}\begin{aligned}
&\|e^{-\gamma t}(\tilde k_1(\bG, H_\rho), \tilde \bk_2(\bG, H_\rho))
\|_{L_p(\BR, H^2_q(\dot\Omega)}  \\
&\quad 
\leq C\{(\|\rho\|_{L_p((0, T), W^{3-1/q}_q(\Gamma))} + \epsilon)
(e^{2(\gamma-\gamma_1)}B + (L+B)T^{\frac{s}{p'(1+s)}}) \\
&\hskip7cm
+L^2T^{1/p'}(L+e^{2(\gamma-\gamma_1)}B)\}  \\
&\quad
\leq C\{e^{2(\gamma-\gamma_1)}B\|\rho\|_{L_p((0, T), W^{3-1/q}_q(\Gamma))}
+L^2(L+e^{2(\gamma-\gamma_1)}B)T^{\frac{s}{p'(1+s)}}\}.
\end{aligned}\end{equation}
By \eqref{4.5.26}, \eqref{4.5.27}, \eqref{est:1}, \eqref{small:4.1},  and \eqref{semi:11}, we have 
\begin{align}
&\|e^{-\gamma t}\pd_t(\tilde k_1(\bG, H_\rho), \tilde \bk_2(\bG, H_\rho))
\|_{L_p(\BR, L_q(\dot\Omega)}  \nonumber \\
&\quad
\leq C\{\|\pd_t\CE_2[H_\rho]\|_{L_p(\BR, H^1_q(\Omega))}
\|e^{-\gamma t}\CE_1[tr[\bG]]\|_{L_\infty(\BR, H^1_q(\dot\Omega))} \nonumber  \\
&\qquad + \|\CE_2[H_\rho]\|_{L_\infty(\BR, H^2_q(\Omega))}
\|e^{-\gamma t}\pd_t\CE_1[tr[\bG]]\|_{L_p(\BR, L_q(\dot\Omega))}
\nonumber \\
&\quad\leq C\{(T^{1/p}L + \epsilon)(e^{2(\gamma-\gamma_1)}B +(L+B)T^{\frac{s}{p'(1+s)}})
+ (\epsilon +T^{1/p'}L )(L+e^{2(\gamma-\gamma_1)}B)\} \nonumber \\
&\quad \leq C(T^{1/p} + T^{1/p'})L (L+ e^{2(\gamma-\gamma_1)}B)
\leq CL(L+e^{2(\gamma-\gamma_1)}B)T^{1/p}.
\label{4.5.39}
\end{align}
Applying the estimate in Theorem \ref{thm:4.2.1} with $\gamma=\gamma_1$ to equations \eqref{eq:4.2}
and using \eqref{4.5.8*}, \eqref{4.5.37}, 
\eqref{4.5.38}, and \eqref{4.5.39}, we have
$$E^1_T(\bH) \leq Ce^{\gamma_1 T}\{B(1 + \|\rho\|_{L_p((0, T), W^{3-1/q}_q(\Gamma))})
+L^2(L+B)T^{\frac{s}{p'(1+s)}}\},
$$
which, combined with \eqref{semi:4.8}, yields that
\begin{equation}\label{4.5.40}\begin{aligned}
& E^1_T(\bv) + E^2_T(\rho) +
\|\pd_t\rho\|_{L_p((0, T), W^{1-1/q}_q(\Gamma))} + E^1_T(\bH) \\
&\quad
\leq M_1(B^2+L^3T^{\frac{s}{p'(1+s)}}) + Ce^{\gamma_1 T}
(B(1 + M_1B^2 + M_1L^3T^{\frac{s}{p'(1+s)}}) + L^2(L+B)T^{\frac{s}{p'(1+s)}})
\\
& \quad \leq M_1B^2+ Ce^{\gamma_1} B(1+M_1B^2)
+ \{M_1L^3+ Ce^{\gamma_1} M_1L^3 + L^2(L+B)\}T^{\frac{s}{p'(1+s)}}\end{aligned}\end{equation}
provided that   $0 < \epsilon = T = \kappa < 1$,  $L>1$, and $B> 1$. 
Choosing $L > 0$ so large that $L /2\geq  M_1B^2+ Ce^{\gamma_1} B(1+M_1B^2)$ and 
$T> 0$ so small that $L /2 \geq  \{M_1L^3+ Ce^{\gamma_1} M_1L^3 + L^2(L+B)\}T^{\frac{s}{p'(1+s)}}$, 
and setting $L = f(B) = 2(M_1B^2+ Ce^{\gamma_1} B(1+M_1B^2))$, 
we see that 
$E_T(\bv, \bH, \rho) \leq L$. 
If we define a map $\Phi$ by $\Phi(\bu, \bG, h) = (\bv, \bH, \rho)$,
then, $\Phi$ maps $U_T$ into itself.  

\section{Estimates of the difference of nonlinar terms
and completion of the proof of Theorem \ref{thm:main}}\label{sec:5}

Let $(\bu_i, \bG_i, h_i) \in \bU_T$ ($i=1,2$). In this section mainly we shall estimate 
$E_T(\bv_1 - \bv_2, \bH_1-\bH_2, \rho_1-\rho_2)$ with  
$(\bv_i, \bH_i, \rho_i) = \Phi(\bu_i, \bG_i, h_i)$ ($i = 1,2$)
and then we shall prove that  $\Phi$ is a contraction map on $U_T$ 
with a suitable choice of  $\epsilon > 0$. 
For  notational simplicity, we set
\begin{align*}
\tilde\bv&=\bv_1-\bv_2, \quad \tilde\bH=\bH_1-\bH_2, 
\quad \tilde\rho = \rho_1-\rho_2, \quad 
\CF_1  = \bff_1(\bu_1, \bG_1, H_{h_1}) - \bff_1(\bu_2, \bG_2, H_{h_2}), 
\\
 \fg &= g(\bu_1, H_{h_1}) - g(\bu_2, H_{h_2}), \quad 
\CG  = \bg(\bu_1, H_{h_1}) - \bg(\bu_2, H_{h_2}), \quad
\CD = d(\bu_1, H_{H_1}) - d(\bu_2, H_{h_2}), \\
\CH_1 & = \bh_1(\bu_1, \bG_1, H_{h_1}) - \bh_1(\bu_2, \bG_2, H_{h_2}), 
\quad 
\CF_2 = \bff_2(\bu_1, \bG_1, H_{h_1}) - \bff_2(\bu_2, \bG_2, H_{h_2}), \\
\CH_2 & = \bh_2(\bu_1, \bG_1, H_{h_1}) - \bh_2(\bu_2, \bG_2, H_{h_2}), 
\quad 
\CH_3  = \bh_3(\bu_1, \bG_1, H_{h_1}) - \bh_3(\bu_2, \bG_2, H_{h_2}),\\
\CK_1 & = k_1(\bG_1, H_{\rho_1}) - k_1(\bG_2, H_{\rho_2}), \quad 
\CK_2 = \bk_2(\bG_1, H_{\rho_1}) -\bk_2(\bG_2, H_{\rho_2}).
\end{align*}
And then, $\bar\bv$ and $\bar\rho$ satisfy the following equations with 
some pressure term $Q$ : 
\begin{equation}\label{eq:5.1}\left\{
\begin{aligned}
\pd_t\tilde\bv - \DV\bT(\tilde \bv, Q) &= \CF_1, 
&\quad &\text{in $\dot\Omega\times(0, T)$}, \\
 \dv\tilde\bv= \fg &= \dv\CG
&\quad &\text{in $\dot\Omega\times(0, T)$}, \\
\pd_t \tilde\rho +<\nabla_\Gamma \tilde\rho \perp \bu_\kappa> - \bn\cdot
\tilde\bv_+ &= \CD
&\quad &\text{on $\Gamma\times(0, T)$},\\
[[\tilde\bv]]=0, \quad 
[[\bT(\tilde\bv, Q)\bn]] -\sigma (\Delta_\Gamma \tilde\rho + a\tilde\rho)\bn &= \CH_1, 
&\quad &\text{on $\Gamma\times(0, T)$}, \\
\tilde\bv_\pm&=0
&\quad&\text{on $S_\pm\times(0, T)$}, \\
(\tilde\bv, \tilde\rho)|_{t=0}  &= (0,  0)
&\quad&\text{in $\dot\Omega\times \Gamma$}.
\end{aligned}
\right.
\end{equation}
And $\bar\bH$ satisfies the following equations:
\begin{equation}\label{eq:5.2}\left\{
\begin{aligned}
\mu\pd_t\tilde\bH - \alpha^{-1}\Delta\tilde\bH &=\CF_2
&\quad &\text{in $\dot\Omega\times(0, T)$}, \\
[[\alpha^{-1}\curl\tilde\bH]]\bn= \CH_2 \quad [[\mu\dv\tilde\bH]] 
&=\CH_3
&\quad &\text{on $\Gamma\times(0, T)$}, \\
[[\mu\tilde\bH\cdot\bn]]=\CK_1, 
\quad [[\tilde\bH-<\tilde\bH, \bn>\bn]]&=\CK_2
&\quad &\text{on $\Gamma\times(0, T)$}, \\
\quad \bn_\pm\cdot\tilde\bH_\pm =0, \quad
(\curl\tilde\bH_\pm)\bn_\pm &= 0
&\quad&\text{on $S_\pm\times(0, T)$}, \\
\tilde
\bH|_{t=0}  &= 0
&\quad&\text{in $\dot\Omega$}.
\end{aligned}
\right.
\end{equation}
We have to estimate the nonlinear terms appearing in the right side of equations
\eqref{eq:5.1} and \eqref{eq:5.2}.  We start with estimating $\CF_1$.  
As was written in \eqref{non:f1}, we write 
$$\bff_1(\bu, \bG, H_h) = \bV_{\bff_1}(\bar\nabla H_h)\bff_3(\bu, \bG, H_h)$$
with
$$\bff_3(\bu, \bG, H_h) = \bar\nabla H_h
\otimes(\pd_t\bu, \nabla^2\bu)
+ \pd_tH_h\otimes\nabla\bu 
+\bu\otimes\nabla\bu+ \bar\nabla^2 H_h\otimes\nabla \bu
+ \bG\otimes\nabla\bG.
$$
And then, we can write $\CF_1$ as follows:
\begin{align*}
&\CF_1  = (\bV_{\bff_1}(\bar\nabla H_{h_1}) - \bV_{\bff_1}(\bar\nabla H_{h_2}))
\bff_3(\bu_1, \bG_1, H_{h_1}) +
\bV_{\bff_1}(\bar\nabla H_{h_2})(\bff_3(\bu_1, \bG_1, H_{h_1}) - 
\bff_3(\bu_2, \bG_2, H_{h_2})); \\
&\bff_3(\bu_1, \bG_1, H_{h_1}) - 
\bff_3(\bu_2, \bG_2, H_{h_2}) \\
&\quad 
= \bar\nabla(H_{h_1}-H_{h_2})
\otimes(\pd_t\bu_1, \nabla^2\bu_1)
+ \bar\nabla H_{h_2}\otimes(\pd_t(\bu_1-\bu_2), \nabla^2(\bu_1-\bu_2))\\
&\quad +\pd_t(H_{h_1}-H_{h_2})\otimes \nabla\bu_1
+ \pd_tH_{h_2}\otimes\nabla (\bu_1-\bu_2)
+(\bu_1-\bu_2)\otimes\nabla\bu_1 + \bu_2\otimes\nabla(\bu_1-\bu_2)\\
&\quad + \bar\nabla^2(H_{h_1}-H_{h_2})\otimes\nabla \bu_1
+ \bar\nabla^2H_{h_2}\otimes\nabla(\bu_1-\bu_2)
 +(\bG_1-\bG_2)\otimes\nabla \bG_1 + \bG_2\otimes\nabla(\bG_1-\bG_2).
\end{align*}
Since we may write
\begin{equation}\label{diff:1}
\bV_{\bff_1}(\bar\nabla H_{h_1}) - \bV_{\bff_1}(\bar\nabla H_{h_2})
= \int^1_0(d_{\bK}\bV_{\bff_1})(\bar\nabla H_{h_2}
+ \theta\bar\nabla(H_{h_1}-H_{h_2}))\,d\theta\bar\nabla(H_{h_1}-H_{h_2}),
\end{equation}
where $d_\bK\bV_{\bff_1}f$ is the derivative of $\bV_{\bff_1}(\bK)$ with respect
to $\bK$, noting that $H_{h_1}-H_{h_2} = 0$ for $t=0$ and using
 \eqref{4.5.14} and \eqref{4.5.4}, we have
\begin{equation}\label{eq:5.1} 
\|\bV^1_{\bff_1}(\bar\nabla H_{h_1}) - \bV^1_{\bff_1}(\bar\nabla H_{h_2})
\|_{L_\infty((0, T), H^1_q(\Omega))}
\leq T^{1/p'}\|\pd_t h_1- \pd_th_2\|_{L_p((0, T), H^2_q(\Omega))}.
\end{equation}
Since $\bff_3(\bu, \bG, H_h)$ satisfies the 
estimate \eqref{4.5.8}, replacing $h$, $\bu$, and $\bG$ with $h_1$, $\bu_1$ and $\bG_1$, we have 
\begin{equation}\label{eq:5.2}
 \|\bff_3(\bu_1, \bG_1, H_{h_1})\|_{L_p((0, T), L_q(\dot\Omega))}
 \leq C\{T^{1/p}(L+B)^2+ (\epsilon + T^{1/p'}L)L\}
 \leq CT^{1/p}(L+B)^2.
\end{equation}
By \eqref{4.5.1} and  \eqref{1.3.2}, we have
\begin{align*}
\|\bff_3&(\bu_1, \bG_1, H_{h_1}) - \bff_3(\bu_2, \bG_2, H_{h_2})\|_{L_q(\dot\Omega)}\\
&\leq C\{\|h_1-h_2\|_{W^{2-1/q}_q(\Gamma)}\|(\pd_t\bu_1, \nabla^2\bu_1)\|_{L_q(\dot\Omega)}
+ \|h_2\|_{W^{2-1/q}_q(\Gamma)}(\|\pd_t(\bu_1-\bu_2), \nabla^2(\bu_1-\bu_2)\|_{L_q(\dot\Omega)})\\
&
 +\|\pd_t(h_1-h_2)\|_{W^{1-1/q}_q(\Gamma)}\|\nabla\bu_1\|_{L_q(\dot\Omega)}
 + \|\pd_th_2\|_{W^{1-1/q}_q(\Gamma)}\|\nabla(\bu_1-\bu_2)\|_{L_q(\dot\Omega)}\\
& +\|(\bu_1, \bu_2)\|_{H^1_q(\dot\Omega)}\|\bu_1-\bu_2\|_{H^1_q(\dot\Omega)}
+ \|h_1-h_2\|_{W^{2-1/q}_q(\Gamma)}\|\bu_1\|_{H^2_Q(\dot\Omega)}\\
&+ \|h_2\|_{W^{2-1/q}_q(\Gamma)}\|\bu_1-\bu_2\|_{H^2_q(\dot\Omega)}
+ \|(\bG_1, \bG_2)\|_{H^1_q(\dot\Omega)}\|\bG_1-\bG_2)\|_{H^1_q(\dot\Omega)}\}.
\end{align*}
%%%%%%%%%%%%%%
Since 
\begin{equation}\label{diffh.1}
\|h_1-h_2\|_{L_\infty((0, T), W^{2-1/q}_q(\Gamma)} \leq T^{1/p'}\|\pd_t(h_1-h_2)
\|_{L_p((0, T), W^{2-1/q}_q(\Gamma))},
\end{equation}
noting that $\bu_1-\bu_2=0$  and $\bG_1-\bG_2=0$ at $t=0$, by \eqref{4.5.14},
\eqref{4.5.10}, and \eqref{4.5.12}, we have
\begin{align*}
&\|\bff_3(\bu_1, \bG_1, H_{h_1}) - \bff_3(\bu_2, \bG_2, H_{h_2}))
\|_{L_p((0, T), L_q(\dot\Omega))}\\
&\quad
\leq C\{T^{1/{p'}}\|\pd_t(h_1-h_2)\|_{L_p((0, T), W^{2-1/q}_q(\Gamma))}L
+ LT^{1/p'}E^1_T(\bu_1-\bu_2) \\
&\quad + T^{1/p}(\|\pd_t(h_1-h_2)\|_{L_\infty((0, T), W^{1-1/q}_q(\Gamma))}(B+L)
+LE^1_T(\bu_1-\bu_2) +(B+L)E^1_T(\bu_1-\bu_2) )\\
&\quad 
+T^{1/p'}(\|\pd_t(h_1-h_2)\|_{L_p((0, T), W^{2-1/q}_q(\Gamma))}L 
+ L\|\bu_1-\bu_2\|_{L_p((0, T), H^2_q(\dot\Omega))})
+ T^{1/p}(B+L)E^2_T(\bG_1-\bG_2), 
\end{align*}
which, combined with \eqref{eq:5.1} and \eqref{eq:5.2}, leads to 
\begin{equation}\label{eq:5.3}
\|\CF_1\|_{L_p((0, T), L_q(\dot\Omega))}
\leq CT^{1/p}(L+B)E_T(\bu_1-\bu_2, \bG_1-\bG_2, h_1-h_2).
\end{equation}
where we have used the estimate: $T^{1/p'}T^{1/p}(L+B)^2 \leq 2(T^{1/p'}L)T^{1/p}(L+B)
\leq 2T^{1/p}(L+B)$, which follows from $B \leq L$ and $T^{1/p'}L\leq 1$. 

We next consider the difference $\CD$.  In view of \eqref{kinema:2}, 
we write
\begin{align*}
\CD& = <\nabla_\Gamma(H_{h_1}-H_{h_2})\perp \bu_1-\bu_\kappa> 
+ <\nabla_\Gamma H_{h_2} \perp \bu_1- \bu_2> \\
& + <\bu_1-\bu_2 - \frac{\pd}{\pd t}(H_{h_1}- H_{h_2})\bn, 
\bV_\bn(\cdot, \bar\nabla H_1)\bar\nabla H_1\otimes\bar\nabla H_1>
\\
& + <\bu_2 - \frac{\pd}{\pd t}H_{h_2}\bn,
(\tilde \bV_\bn(\cdot, \bar\nabla H_{h_1}) - \tilde \bV_\bn(\cdot, 
\bar\nabla H_{h_2}))\otimes \bar\nabla H_{h_1}> 
\\
& +  <\bu_2 - \frac{\pd}{\pd t}H_{h_2}\bn,
\bV_\bn(\cdot, \bar\nabla H_{h_2})\bar\nabla H_{h_2} 
\otimes \bar\nabla(H_{h_1} - H_{h_2})>,
\end{align*}
where we have set $\bV_\bn(\cdot, \bK)\bK = \tilde\bV_\bn(\cdot, \bK)$. 
We have
\begin{align}
\|\CD\|_{L_\infty((0, T), W^{1-1/q}_q(\Gamma))} 
&\leq  C(L+B)T^{1/p'}E_T(\bu_1-\bu_2, \bG_1-\bG_2, h_1-h_2);
\label{eq:5.5}\\
\|\CD\|_{L_p((0, T), W^{2-1/q}_q(\Gamma))} & \leq CL(L+B)T^{\frac{s}{p'(1+s))}}
E_T(\bu_1-\bu_2, \bG_1-\bG_2, h_1-h_2).  \label{eq:5.6}
\end{align} 
In fact, noting that  the difference: 
$\tilde \bV_\bn(\cdot, \bar\nabla H_{h_1}) - \tilde \bV_\bn(\cdot, 
\bar\nabla H_{h_2})$ has the similar formula to that in \eqref{diff:1},
 by \eqref{1.3.2}, \eqref{4.5.4}, \eqref{4.5.1}, and \eqref{eq:4.5.1},  we have
\begin{align*}
\|\CD&\|_{W^{1-1/q}_q(\Gamma)} 
 \leq C\{\|h_1-h_2\|_{W^{2-1/q}_q(\Gamma))}
\|\bu_1-\bu_\kappa\|_{ H^1_q(\Omega_+)}  + \|h_2\|_{W^{2-1/q}_q(\Gamma)} 
\|\bu_1-\bu_2\|_{ H^1_q(\Omega_+)} \\
 & +(\|\bu_1-\bu_2\|_{H^1_q(\Omega_+)} 
 + \|\pd_t(h_1-h_2)\|_{W^{1-1/q}_q(\Gamma))})
 (1+\|h_1\|_{W^{2-1/q}_q(\Gamma))})\|h_1\|_{W^{2-1/q}_q(\Gamma))}^2
 \\ 
& +(\|\bu_2\|_{H^1_q(\Omega_+)}+\|\pd_th_2\|_{W^{1-1/q}_q(\Gamma)})
(1 + \|h_1\|_{W^{2-1/q}_q(\Gamma)} + \|h_2\|_{W^{2-1/q}_q(\Gamma)}) \\
 &\hskip7cm \times \|h_1-h_2\|_{W^{2-1/q}_q(\Gamma)}\|h_1\|_{W^{2-1/q}_q(\Gamma)}
 \\
 &  +(\|\bu_2\|_{H^1_q(\Omega_+)}+\|\pd_th_2\|_{W^{1-1/q}_q(\Gamma)})
 (1 + \|h_2\|_{W^{2-1/q}_q(\Gamma)})\|h_2\|_{W^{2-1/q}_q(\Gamma)}
 \|h_1-h_2\|_{W^{2-1/q}_q(\Gamma)}\}. 
\end{align*}
Thus, by \eqref{4.5.12}, \eqref{4.5.14}, \eqref{eq:d.3}, and \eqref{diffh.1} we have 
\begin{align*}
&\|\CD\|_{L_\infty((0, T), W^{1-1/q}_q(\Gamma))} \\
&\leq C\{T^{1/p'}E^2_T(h_1-h_2)T^{\frac{s}{p'(1+s)}}(L+B) 
+  LT^{1/p'} E^1_T(\bu_1-\bu_2) \\
&+LT^{1/p'}(E^1_T(\bu_1-\bu_2) + 
\|\pd_t(h_1-h_2)\|_{L_\infty((0, T), W^{1-1/q}_q(\Gamma))})
 +(L+B)T^{1/p'}(\epsilon + LT^{1/p'})E^2_T(h_1-h_2)\},
\end{align*}
which leads to the first inequality in \eqref{eq:5.5}, because 
$T^{\frac{s}{p'(1+s)}}<1$ as follows from $0 < T< 1$. 

By \eqref{1.3.2}, \eqref{4.5.4}, \eqref{4.5.1}, and \eqref{eq:4.5.1}, we have
\begin{align*}
&\|\CD\|_{W^{2-1/q}_q(\Gamma)}  \leq C[\|h_1-h_2\|_{W^{3-1/q}_q(\Gamma)}
\|\bu_1-\bu_\kappa\|_{H^1_q(\Omega_+)}
+ \|h_1-h_2\|_{W^{2-1/q}_q(\Gamma)}\|\bu_1-\bu_\kappa\|_{H^2_q(\Omega_+)}\\
&\quad 
+ \|h_2\|_{W^{3-1/q}_q(\Gamma)}\|\bu_1-\bu_2\|_{H^1_q(\Omega_+)}
+ \|h_2\|_{W^{2-1/q}_q(\Gamma)}\|\bu_1-\bu_2\|_{H^2_q(\Omega_+)}\\
&\quad +(\|\bu_1-\bu_2\|_{H^2_q(\Omega_+)}
+ \|\pd_t(h_1-h_2)\|_{W^{2-1/q}_q(\Gamma)})(1+\|h_1\|_{W^{2-1/q}_q(\Gamma)})
\|h_1\|_{W^{2-1/q}_q(\Gamma)}^2 \\
&\quad + (\|\bu_1-\bu_2\|_{H^1_q(\Omega_+)}
+ \|\pd_t(h_1-h_2)\|_{W^{1-1/q}_q(\Gamma)})\\
&\qquad \times (\|h_1\|_{W^{3-1/q}_q(\Gamma)}\|h_1\|_{W^{2-1/q}_q(\Gamma)}
+ \|h_1\|_{W^{3-1/q}_q(\Gamma)}(1 + \|h_1\|_{W^{2-1/q}_q(\Gamma)})
\|h_1\|_{W^{2-1/q}_q(\Gamma)}^2)\\
&\quad +(\|\bu_2\|_{H^2_q(\Omega_+)} + \|\pd_th_2\|_{W^{2-1/q}_q(\Gamma)})
(1 + \|h_1\|_{W^{2-1/q}_q(\Gamma)} + \|h_2\|_{W^{2-1/q}_q(\Gamma)})\\
&\hskip6cm \times
\|h_1-h_2\|_{W^{2-1/q}_q(\Gamma)}\|h_1\|_{W^{2-1/q}_q(\Gamma)}\\
&\quad + (\|\bu_2\|_{H^1_q(\Omega_+)} + \|\pd_th_2\|_{W^{1-1/q}_q(\Gamma)})
\{\|h_1-h_2\|_{W^{3-1/q}_q(\Gamma)}\|h_1\|_{W^{2-1/q}_q(\Gamma)} \\
&\quad  + \|h_1-h_2\|_{W^{2-1/q}_q(\Gamma)}\|h_1\|_{W^{3-1/q}_q(\Gamma)}
+ (\|h_1\|_{W^{3-1/q}_q(\Gamma)} + \|h_2\|_{W^{3-1/q}_q(\Gamma)})\\
&\qquad\times(1+\|h_1\|_{W^{2-1/q}_q(\Gamma)} 
+ \|h_2\|_{W^{2-1/q}_q(\Gamma)}) \|h_1-h_2\|_{W^{2-1/q}_q(\Gamma)}
\|h_1\|_{W^{2-1/q}_q(\Gamma)}\}\\
%%%%%%%%%%%%%%%%%%%%%%%
&\quad +(\|\bu_2\|_{H^2_q(\Omega_+)} + \|\pd_th_2\|_{W^{2-1/q}_q(\Gamma)})
(1 + \|h_2\|_{W^{2-1/q}_q(\Gamma)})\|h_2\|_{W^{2-1/q}_q(\Gamma)}\|h_1-h_2\|_{W^{2-1/q}_q(\Gamma)}\\
&\quad + (\|\bu_2\|_{H^1_q(\Omega_+)} + \|\pd_th_2\|_{W^{1-1/q}_q(\Gamma)})\\
&\qquad\times\{\|h_2\|_{W^{3-1/q}_q(\Gamma)}\|h_1-h_2\|_{W^{2-1/q}_q(\Gamma)} 
 + \|h_2\|_{W^{2-1/q}_q(\Gamma)}\|h_1-h_2\|_{W^{3-1/q}_q(\Gamma)} \\
&\qquad \quad+ \|h_2\|_{W^{3-1/q}_q(\Gamma)}(1 +  \|h_2\|_{W^{2-1/q}_q(\Gamma)})
\|h_2\|_{W^{2-1/q}_q(\Gamma)}\|h_1-h_2\|_{W^{2-1/q}_q(\Gamma)}\}].
\end{align*}
Since 
\begin{equation}\label{est:6}\begin{aligned}
&\|\bu_1-\bu_2\|_{H^1_q(\Omega_+)}  \\
&\quad \leq C_sT^{\frac{s}{p'(1+s)}}
\|\pd_t(\bu_1-\bu_2)\|_{L_p((0, T), L_q(\Omega_+))}^{s/(1+s)}
\|\bu_1-\bu_2\|_{B^{2(1-1/p)}_{q,p}(\Omega_+)}
\leq C_sT^{\frac{s}{p'(1+s)}}E^1_T(\bu_1-\bu_2)
\end{aligned}\end{equation}
by \eqref{4.5.14}, \eqref{4.5.10}, and \eqref{eq:d.3}
we have 
\begin{align*}
&\|\CD\|_{L_p((0, T), W^{2-1/q}_q(\Gamma))}\\
&\quad \leq C[T^{\frac{s}{p'(1+s)}}(L+B)\|h_1-h_2\|_{L_p((0, T), W^{3-1/q}_q(\Gamma))}
+ T^{1/p'}\|\pd_t(h_1-h_2)\|_{L_p((0, T), W^{2-1/q}_q(\Gamma))}(B+L)\\
&\quad + LT^{\frac{s}{p'(1+s)}}E^1_T(\bu_1-\bu_2)
+LT^{1/p'}\|\bu_1-\bu_2\|_{L_p((0, T), H^2_q(\dot\Omega_+))} \\
&\quad+(\|\bu_1-\bu_2\|_{L_p((0, T), H^2_q(\dot\Omega_+))} 
+ \|\pd_t(h_1-h_2)\|_{L_p((0, T), W^{2-1/q}_q(\Gamma))})LT^{1/p'}\\
&\quad +(E^1_T(\bu_1-\bu_2) + \|\pd_t(h_1-h_2)\|_{L_\infty((0, T), W^{1-1/q}_q(\Gamma))})L^2T^{1/p'}\\
&\quad + LT^{1/p'}\|\pd_t(h_1-h_2)\|_{L_p((0, T), W^{2-1/q}_q(\Gamma))}
+(L+B)\{\|h_1-h_2\|_{L_p((0, T), W^{3-1/q}_q(\Gamma))}LT^{1/p'} \\
&\quad + T^{1/p'}\|\pd_t(h_1-h_2)\|_{L_p((0, T), W^{2-1/q}_q(\Gamma))}L\}
+ LT^{1/p'}\|\pd_t(h_1-h_2)\|_{L_p((0, T), W^{2-1/q}_q(\Gamma))}\\
&\quad + (L+B)(LT^{1/p'}\|\pd_t(h_1-h_2)\|_{L_p((0, T), W^{2-1/q}_q(\Gamma))}
+ LT^{1/p'}\|h_1-h_2\|_{L_p((0, T), W^{3-1/q}_q(\Gamma))}\\
&\hskip8cm+ LT^{1/p'}\|\pd_t(h_1-h_2)\|_{L_p((0, T), W^{2-1/q}_q(\Gamma))})],
\end{align*}
which yields \eqref{eq:5.6}.

We next conider $\CH_1$.  In view of \eqref{eq:4.1*}, we set
$$\tilde\CH_1 = \CH^1_1 + a(y)\CH^2_1 + \CH^3_1+ \CH^4_1$$
with
\begin{align*}
\CH^1_1 & = (\tilde\bV^1_\bh(\cdot, \bar\nabla \CE_2[H_{h_1}])
- \tilde\bV^1_\bh(\cdot, \bar\nabla \CE_2[H_{h_2}]))\otimes\nabla \CE_1[tr[\bu_1]]\\
&\quad + \tilde\bV_\bh(\cdot, \bar\nabla\CE_2[H_{h_2}])\bar\nabla \CE_2[H_{h_2}]
\otimes\nabla(\CE_1[tr[\bu_1]]-\CE_1[tr[\bu_2]]), \\
\CH^2_1 & = (\CE_1[tr[\bG_1]]-\CE_1[tr[\bG_2]])\otimes \CE_1[tr[\bG_1]]\\
&\quad+ \CE_1[tr[\bG_2]]\otimes (\CE_1[tr[\bG_1]]-\CE_1[tr[\bG_2]]), \\
\CH^3_1 & = (\tilde\bV^2_\bh(\cdot, \bar\nabla \CE_2[H_{h_1}])
- \tilde\bV^2_\bh(\cdot, \bar\nabla \CE_2[H_{h_2}]))\otimes \CE_1[tr[\bG_1]]
\otimes \CE_1[tr[\bG_1]]\\
&\quad + \bV^2_\bh(\cdot, \bar\nabla\CE_2[H_{h_2}])\bar\nabla\CE_2[H_{h_2}]
\otimes \CH^2_1, \\
\CH^4_1 & = (\tilde\bV_s(\cdot, \bar\nabla \CE_2[H_{h_1}])
- \tilde\bV_s(\cdot, \bar\nabla \CE_2[H_{h_2}]))\otimes 
\bar\nabla^2\CE_2[H_{h_1}] \\
&\quad + \bV_s(\cdot, \bar\nabla \CE_2[H_{h_2}])\bar\nabla\CE_2[H_{h_2}]
\otimes\bar\nabla^2(\CE_2[H_{h_1}] - \CE_2[H_{h_2}]),
\end{align*}
where we have set
$$\tilde\bV^1_\bh(\cdot, \bK) = \bV^1_\bh(\cdot, \bK)\bK, 
\quad
\tilde\bV^2_\bh(\cdot, \bK) = \bV^2_\bh(\cdot, \bK)\bK, 
\quad
\tilde\bV_s(\cdot, \bK) = \bV_s(\cdot, \bK)\bK. 
$$
We see that $\tilde\CH_1$ is defined for $t \in \BR$ and 
$\tilde \CH_1 = \CH_1$ for $t \in (0, T)$.  Writing 
\begin{align*}
\tilde\bV^1_\bh(\cdot, \bar\nabla\CE_2[H_{h_1}]) - 
\tilde\bV^1_\bh(\cdot, \bar\nabla\CE_2[H_{h_2}]) 
&= \int^1_0(d_K\tilde\bV^1_\bh)(\cdot, \bar\nabla \CE_2[H_{h_2}]
+ \theta\bar\nabla(\CE_2[H_{h_1}] - \CE_2[H_{h_2}]))\,d\theta\\
&\qquad \times
\bar\nabla(\CE_2[H_{h_1}] - \CE_2[H_{h_2}]),
\end{align*}
and since $\CE_2[H_{h_1}] - \CE_2[H_{h_2}] = e_T[H_{h_1}- H_{h_2}]$, 
by \eqref{1.3.2}, \eqref{4.5.1}, \eqref{est:1}, and \eqref{4.5.34*}, and \eqref{diffh.1}, 
we have 
\begin{equation}\label{eq:5.12}\begin{aligned}
&\|e^{-\gamma t}\pd_t(\tilde\bV^1_\bh(\cdot, \bar\nabla\CE_2[H_{h_1}]) - 
\tilde\bV^1_\bh(\cdot, \bar\nabla\CE_2[H_{h_2}]) )
\|_{L_\infty(\BR, L_q(\dot\Omega))} \\
&\quad \leq C\{(\|\pd_t\bar\nabla \CE_2[H_{h_1}]\|_{L_\infty(\BR, L_q(\Omega))}
+ \|\pd_t\bar\nabla\CE_2[H_{h_2}]\|_{L_\infty(\BR, L_q(\Omega))})
\|h_1-h_2\|_{L_\infty((0, T), W^{2-1/q}_q(\Gamma))}
 \\
&\quad +(\|\bar\nabla\CE_2[H_{h_1}]\|_{L_\infty(\BR, H^1_q(\Omega))}
+ \|\bar\nabla\CE_2[H_{h_2}]\|_{L_\infty(\BR, H^1_q(\Omega))})
\|\pd_t(h_1-h_2)\|_{L_\infty((0, T), W^{1-1/q}_q(\Gamma))}
 \\
& \quad \leq CLT^{1/p'}(\|\pd_t(h_1-h_2)\|_{L_p((0, T), W^{2-1/q}_q(\Gamma))}
+ \|\pd_t(h_1-h_2)\|_{L_\infty((0, T), W^{1-1/q}_q(\Gamma))}\}.
\end{aligned}\end{equation}
  And also, by \eqref{4.5.34*} and \eqref{diffh.1}, we have 
\begin{align}
&\|e^{-\gamma t}(\tilde\bV^1_\bh(\cdot, \bar\nabla\CE_2[H_{h_1}])-
\tilde\bV^1_\bh(\cdot, \bar\nabla \CE_2[H_{h_2}]))\|_{L_\infty(\BR, H^1_q(\Omega))}
\nonumber \\
&\quad \leq C(1 + \|\bar\nabla\CE_2[H_{h_1}]\|_{L_\infty(\BR, H^1_q(\Omega))}
+ \|\bar\nabla\CE_2[H_{h_2}]\|_{L_\infty(\BR, H^1_q(\Omega))}]
\|\bar\nabla e[H_{h_1}-H_{h_2}]\|_{L_\infty(\BR, H^1_q(\Omega))} 
\nonumber \\
&\quad \leq CT^{1/p'}\|\pd_t(h_1-h_2)\|_{L_p((0, T), W^{2-1/q}_q(\Gamma))}.
\label{eq:5.13}
\end{align}
By Lemma \ref{lem:4.5.4} and \eqref{est:1}, we have
\begin{align*}
&\|\nabla \CE_1[tr[\bu_1]]\|_{H^{1/2}_p(\BR, L_q(\dot\Omega))}
+ \|\nabla\CE_1[tr[\bu_1]]\|_{L_p(\BR, H^1_q(\dot\Omega))}
\\
&\quad \leq C(\|\CE_1[tr[\bu_1]]\|_{L_p(\BR, H^2_q(\dot\Omega))}
+ \|\pd_t\CE_1[tr[\bu_1]]\|_{L_p(\BR, L_q(\dot\Omega))} \\
&\quad  \leq C(B+L).
\end{align*}
Thus, setting 
$$\CH^{11}_1 = (\tilde\bV^1_\bh(\cdot, \bar\nabla \CE_2[H_{h_1}]) - 
\tilde\bV^1_\bh(\cdot, \bar\nabla \CE_2[H_{h_2}]))
\otimes \nabla\CE_1[tr[\bu_1]],$$
by Lemma \ref{lem:4.5.3} we have
\begin{equation}\label{eq:5.7}\begin{aligned}
&\|e^{-\gamma t}\CH^{11}_1\|_{H^{1/2}_p(\BR, L_q(\dot\Omega))}
+ \|e^{-\gamma t}\CH^{11}_1\|_{L_p((\BR, H^1_q(\dot\Omega))} \\
&\quad \leq CL(B+L)T^{1/p'}
(\|\pd_t(h_1-h_2)\|_{L_p((0, T), W^{2-1/q}_q(\Gamma))}
+ \|\pd_t(h_1-h_2)\|_{L_\infty((0, T), W^{1-1/q}_q(\Gamma))}).
\end{aligned}\end{equation}
Noticing that 
$\CE_1[tr[\bu_1]]-\CE_1[tr[\bu_2]] = e_T[tr[\bu_1]-tr[\bu_2]]$, 
by \eqref{trace:4} and Lemma \ref{lem:4.5.4}, we have
\begin{align*}
&\|e^{-\gamma t}
\nabla(\CE_1[tr[\bu_1]]-\CE_2[tr[\bu_2]])\|_{H^{1/2}_p(\BR, L_q(\dot\Omega))}
+\|e^{-\gamma t}\nabla(\CE_1[tr[\bu_1]]-\CE_2[tr[\bu_2]])
\|_{L_p(\BR, H^1_q(\dot\Omega))} \\
&\quad \leq C(\|\bu_1- \bu_2\|_{L_p((0, T), H^2_q(\dot\Omega))}
+ \|\pd_t(\bu_1-\bu_2)\|_{L_p((0, T), L_q(\dot\Omega))}).
\end{align*}
Thus, setting 
$\CH^{12}_1 =  \tilde\bV_\bh(\cdot, \bar\nabla\CE_2[H_{h_2}])\bar\nabla \CE_2[H_{h_2}]
\otimes\nabla(\CE_1[tr[\bu_1]]-\CE_1[tr[\bu_2]])$,
by \eqref{est:2} and Lemma \ref{lem:4.5.3},  we have
\begin{align*}
&\|e^{-\gamma t}\CH^{12}_1\|_{H^{1/2}_p(\BR, L_q(\dot\Omega))}
+ \|e^{-\gamma t}\CH^{12}_1\|_{L_p(\BR, H^1_q(\dot\Omega))} \\
&\quad \leq CLT^{1/(2p')}
(\|\bu_1- \bu_2\|_{L_p((0, T), H^2_q(\dot\Omega))}
+ \|\pd_t(\bu_1-\bu_2)\|_{L_p((0, T), L_q(\dot\Omega))}),
\end{align*}
which, combined with \eqref{eq:5.7}, yields that 
\begin{equation}\label{eq:5.8}\begin{aligned}
&\|e^{-\gamma t}\CH_1^1\|_{H^{1/2}_p(\BR, L_q(\dot\Omega))}
+
\|e^{-\gamma t}\CH_1^1\|_{L_p(\BR, H^1_q(\dot\Omega))}
\\
&\quad \leq CL(B+L)T^{1/(2p')}
 E_T(\bu_1-\bu_2, \bG_1-\bG_2, h_1-h_2).
\end{aligned}\end{equation}

We next consider $\CH^2_1$.  Since 
$\CE_1[tr[\bG_1]]- \CE_2[tr[\bG_2]] = e_T[tr[\bG_1] - tr[\bG_2]]$, 
we have
\begin{equation}\label{eq:5.8*}
\begin{aligned}
&\|e^{-\gamma t}(\CE_1[tr[\bG_1]]-\CE_1[tr[\bG_2]])
\|_{H^1_p(\BR, L_q(\dot\Omega))}
\leq C\|\bG_1- \bG_2\|_{H^1_p((0, T), L_q(\dot\Omega))}; \\
&\|e^{-\gamma t}(\CE_1[tr[\bG_1]]-\CE_1[tr[\bG_2]])
\|_{L_p(\BR, L_q(\dot\Omega))} \\
&\quad\leq CT^{1/p}(\|\bG_1- \bG_2\|_{L_p((0, T), H^2_q(\dot\Omega))}
+ \|\pd_t(\bG_1- \bG_2)\|_{L_p((0, T), L_q(\dot\Omega))});
\\
&\|e^{-\gamma t}(\CE_1[tr[\bG_1]]-\CE_1[tr[\bG_2]])
\|_{L_\infty(\BR, H^1_q(\dot\Omega))} \\
&\quad\leq \|\bG_1-\bG_2\|_{L_\infty((0, T), L_q(\dot\Omega))}^{s/(1+s)}
\|\bG_1-\bG_2\|_{L_\infty((0, T), W^s_q(\dot\Omega))}^{1/(1+s)}\\
&\quad \leq CT^{\frac{s}{p'(1+s)}}
(\|\bG_1- \bG_2\|_{L_p((0, T), H^2_q(\dot\Omega))}
+ \|\pd_t(\bG_1- \bG_2)\|_{L_p((0, T), L_q(\dot\Omega))}).
\end{aligned}\end{equation}
On the other hand, we have
\begin{equation}\label{eq:5.10}\begin{aligned}
\|\CE_1[tr[\bG_i]]\|_{H^1_p(\BR, L_q(\dot\Omega))} & 
\leq C(L+B); \\
\|\CE_1[tr[\bG_i]]\|_{L_p(\BR, L_q(\dot\Omega))} & 
\leq C(L+B); \\
\|\CE_1[tr[\bG_i]]\|_{L_\infty(\BR, H^1_q(\dot\Omega))} & 
\leq C(L+B)
\end{aligned}\end{equation}
for $i=1,2$, and therefore by \eqref{complex:1} we have
\begin{equation}\label{eq:5.9}\begin{aligned}
&\|e^{-\gamma t}\CH^2_1\|_{H^{1/2}_p(\BR, L_q(\dot\Omega))}
\leq C(L+B)T^{1/(2p)}E^1_T(\bG_1-\bG_2).
\end{aligned}\end{equation}
 And also, by \eqref{4.5.1}, \eqref{eq:5.10}
and \eqref{4.5.12}, we have
\begin{align*}
&\|e^{-\gamma t}\CH^2_1\|_{L_p(\BR, H^1_q(\dot\Omega))}
\leq C(L+B)\|\bG_1-\bG_2\|_{L_p((0, T), H^1_q(\dot\Omega))} 
\\
&\quad \leq C(L+B)T^{1/p}
(\|\pd_t(\bG_1-\bG_2)\|_{L_p((0, T), L_q(\dot\Omega))}
+ \|\bG_1- \bG_2\|_{L_p((0, T), H^2_q(\dot\Omega))}),
\end{align*}
which, combined with \eqref{eq:5.9}, yields that
\begin{equation}\label{eq:5.11}
\|e^{-\gamma t}\CH^2_1\|_{H^{1/2}_p(\BR, L_q(\dot\Omega))}
+ \|e^{-\gamma t}\CH^2_1\|_{L_p(\BR, H^1_q(\dot\Omega))}
\leq C(L+B)T^{1/(2p)}E^1_T(\bG_1-\bG_2).
\end{equation}
Since 
$$\|\CE_1[tr[\bG_1]]\otimes \CE_1[tr[\bG_1]]\|_{H^{1/2}_p(\BR, L_q(\dot\Omega))}
+ \|\CE_1[tr[\bG_1]]\otimes \CE_1[tr[\bG_1]]\|_{L_p(\BR, H^1_q(\dot\Omega))}
\leq C(L+B)^2,
$$
setting 
$$\CH^{31}_1=(\tilde\bV^2_\bh(\cdot,\bar\nabla\CE_2[H_{h_1}]) - 
\tilde\bV^2_\bh(\cdot,\bar\nabla\CE_2[H_{h_2}])
\CE_1[tr[\bG_1]]\otimes \CE_1[tr[\bG_1]],$$
by Lemma \ref{lem:4.5.3}, \eqref{eq:5.12}, and \eqref{eq:5.13}, we have
\begin{equation}\label{eq:5.14}\begin{aligned}
&\|e^{-\gamma t}\CH^{31}_1\|_{H^{1/2}_p(\BR, L_q(\dot\Omega))}\\
&\quad\leq CL(L+B)^2T^{1/p'}(\|\pd_t(h_1-h_2)\|_{L_p((0, T), W^{2-1/q}_q(\Gamma))}
+ \|\pd_t(h_1-h_2)\|_{L_\infty((0, T), W^{1-1/q}_q(\Gamma))}).
\end{aligned}\end{equation}
By \eqref{1.3.2}, \eqref{4.5.34*}, \eqref{est:1}, and \eqref{diffh.1},  we have
\begin{align*}
\|e^{-\gamma t}\CH^{31}_1\|_{L_p(\BR, H^1_q(\dot\Omega))}
& \leq  C(1 + \|\bar\nabla \CE_2[H_{h_1}]\|_{L_\infty(\BR, H^1_q(\Omega))}
+ \|\bar\nabla \CE_2[H_{h_1}]\|_{L_\infty(\BR, H^1_q(\Omega))})\\
&\qquad\times 
\|\bar\nabla e_T[H_{h_1}-H_{h_2}]\|_{L_p((0, T), H^1_q(\dot\Omega))}
\|\CE[tr[\bG_1]]\|_{L_\infty(\BR, H^1_q(\dot\Omega))}^2\\
&\leq C(L+B)^2T^{1/p'}\|\pd_t(h_1-h_2)\|_{L_p((0, T), W^{2-1/q}_q(\Gamma))},
\end{align*}
which, combined with \eqref{eq:5.14}, yields that
\begin{equation}\label{eq:5.15*}\begin{aligned}
&\|e^{-\gamma t}\CH^{31}_1\|_{H^{1/2}_p(\BR, L_q(\dot\Omega))}
+ \|e^{-\gamma t}\CH^{31}_2\|_{L_p(\BR, H^1_q(\dot\Omega))}\\
&\quad \leq CL(L+B)^2T^{1/p'}
(\|\pd_t(h_1-h_2)\|_{L_p((0, T), W^{2-1/q}_q(\Gamma))}
+ \|\pd_t(h_1-h_2)\|_{L_\infty((0, T), W^{1-1/q}_q(\Gamma))}).
\end{aligned}\end{equation}
Setting
$\CH^{32}_1=\bV^2_\bh(\cdot, \bar\nabla \CE_2[H_{h_2}])\bar\nabla\CE_2[H_{h_2}]\otimes \CH^2_1$, 
by Lemma \ref{lem:4.5.3}, \eqref{est:2}, and \eqref{eq:5.11}, we have
\begin{align*}
&\|e^{-\gamma t}\CH^{32}_1\|_{H^{1/2}_p(\BR, L_q(\dot\Omega))}
+ \|e^{-\gamma t}\CH^{32}_1\|_{L_p(\BR, H^1_q(\dot\Omega))}  \leq CL(L+B)T^{1/2}
E^1_T(\bG_1-\bG_2),
\end{align*}
where we have used $1/p+1/p' = 1$, 
which, combined with \eqref{eq:5.15*}, yields that 
\begin{equation}\label{eq:5.15}\begin{aligned}
&\|e^{-\gamma t}\CH^3_1\|_{H^{1/2}_p(\BR, L_q(\dot\Omega))}
+ \|e^{-\gamma t}\CH^3_1\|_{L_p(\BR, H^1_q(\dot\Omega))} \leq CL(L+B)^2T^{1/p'}
E^1_T(\bG_1-\bG_2),
\end{aligned}\end{equation}

Since 
\begin{align*}
&\|\bar\nabla^2\CE_1[H_{h_1}]\|_{H^{1/2}_p(\BR, L_q(\dot\Omega))}
+ \|\bar\nabla^2\CE_1[H_{h_1}]\|_{L_p(\BR, H^1_q(\dot\Omega))} \\
&\quad \leq C(\|h_1\|_{H^1_p((0, T), W^{2-1/q}_q(\Gamma))} 
+ \|h_1\|_{L_p((0, T), W^{3-1/q}_q(\Gamma))} \\
&\qquad + \|T_h(\cdot)h_0\|_{H^1_p((0, \infty), H^2_q(\dot\Omega))}
+ \|T_h(\cdot)h_0\|_{L_p((0, \infty), H^3_q(\ddot\Omega))} \\
&\quad \leq C(L+\epsilon) \leq 2CL; \\
&\|e^{-\gamma t}\bar\nabla^2(\CE_1[H_{h_1}]-\CE_2[H_{h_2}])
\|_{H^{1/2}_p(\BR, L_q(\dot\Omega))}
+ \|e^{-\gamma t}\bar\nabla^2(\CE_1[H_{h_1}]-\CE_2[H_{h_2}])
\|_{L_p(\BR, H^1_q(\dot\Omega))} \\
&\quad \leq C(\|h_1-h_2\|_{H^1_p((0, T), W^{2-1/q}_q(\Gamma))} 
+ \|h_1-h_2\|_{L_p((0, T), W^{3-1/q}_q(\Gamma))}),
\end{align*}
by Lemma \ref{lem:4.5.3}, \eqref{est:2}, \eqref{eq:5.12}, 
and \eqref{eq:5.13}, we have 
\begin{align*}
&\|e^{-\gamma t}\CH^4_1\|_{H^{1/2}_p(\BR, L_q(\dot\Omega))}
+ \|e^{-\gamma t}\CH^4_1\|_{L_p(\BR, H^1_q(\dot\Omega))} \\
&\quad 
\leq CL^2T^{1/(2p')}(E^2_T(h_1-h_2) + \|\pd_t(h_1-h_2)\|_{L_\infty((0, T), W^{1-1/q}_q(\Gamma))}),
\end{align*}
which, combined with \eqref{eq:5.8}, \eqref{eq:5.11}, 
and \eqref{eq:5.15}, yields that 
\begin{equation}\label{eq:5.17}\begin{aligned}
&\|e^{-\gamma t}\tilde \CH_1\|_{H^{1/2}_p(\BR, L_q(\dot\Omega))}
+ \|e^{-\gamma t}\tilde\CH_1\|_{L_p(\BR, H^1_q(\dot\Omega))}\\
&\quad \leq CL(B+L)^2T^{1/(2p)}
E_T(\bu_1- \bu_2, \bG_1-\bG_2, h_1-h_2),
\end{aligned}\end{equation}
where we have used the fact that $1/p < 1/p'$. 

We now consider $\fg$ and  $\CG$.  In view of  \eqref{4.5.32}, we set
\begin{align*}
\tilde g & = \CG_1(\bar\nabla\CE_2[H_{h_1}])\bar\nabla\CE_2[H_{h_1}]\otimes
\nabla\CE_1[\bu_1] 
-
\CG_1(\bar\nabla\CE_2[H_{h_2}])\bar\nabla\CE_2[H_{h_2}]\otimes
\nabla\CE_1[\bu_2], \\
\tilde \CG & = \CG_2(\bar\nabla\CE_2[H_{h_1}])\bar\nabla \CE_2[H_{h_1}]\otimes
\CE_1[\bu_1] -
\CG_2(\bar\nabla\CE_2[H_{h_2}])\bar\nabla \CE_2[H_{h_2}]\otimes
\CE_1[\bu_2].
\end{align*}
And then, $\tilde\fg$ and $\tilde \CG$ are defined for $t \in \BR$ and 
$\fg = \tilde \fg$ and $\CG=\tilde\CG$ for $t \in (0, T)$.  
Employing the same argument as in proving \eqref{eq:5.8}, we have
\begin{equation}\label{eq:5.20}\begin{aligned}
&\|e^{-\gamma t}\tilde\fg\|_{H^{1/2}_p(\BR, L_q(\dot\Omega))}
+ \|e^{-\gamma t}\tilde \fg\|_{L_p(\BR, H^1_q(\dot\Omega))} \\
&\quad \leq CL(B+L)T^{1/(2p')}
E_T(\bu_1- \bu_2, \bG_1-\bG_2, h_1-h_2). 
\end{aligned}\end{equation}

To estimate $\tilde\CG$, we write $\tilde \CG = G_1 + G_2$ with
\begin{align*}
G_1 & = (\tilde \CG_2(\bar\nabla\CE_2[H_{h_1}[) - \tilde \CG_2[\bar\nabla
\CE_2[H_{h_2}]))\otimes \CE_1[\bu_1], \\
G_2 & = \CG_2(\bar\nabla\CE_2[H_{h_2}])\bar\nabla\CE_2[H_{h_2}]
\otimes(\CE_1[\bu_1]-\CE_2[\bu_2]),
\end{align*}
where we have set $\tilde \CG_2(\bK) = \tilde \CG_2(\bK)\bK$.  
To estimate $\pd_tG_1$, we write
\begin{align*}
\pd_tG_1 & = \int^1_0(d_\bK\tilde\CG_2)(\bar\nabla 
\CE_2[H_{h_2}] + \theta\bar\nabla(\CE_2[H_{h_1}]-\CE_2[H_{h_2}])]\,d\theta
\bar\nabla(\CE_2[H_{h_1}]-\CE_2[H_{h_2}])\otimes \pd_t\CE_1[\bu_1]\\
&+\Bigl(\int^1_0(d_\bK\tilde\CG_2)(\bar\nabla 
\CE_2[H_{h_2}] + \theta\bar\nabla(\CE_2[H_{h_1}]-\CE_2[H_{h_2}])]\,d\theta
\bar\nabla\pd_t(\CE_2[H_{h_1}]-\CE_2[H_{h_2}])\Bigr)\otimes \CE_1[\bu_1]\\
& +\Bigl(\int^1_0(d^2_\bK\tilde\CG_2)(\bar\nabla 
\CE_2[H_{h_2}] + \theta\bar\nabla(\CE_2[H_{h_1}]-\CE_2[H_{h_2}])]
\pd_t((1-\theta)\bar\nabla 
\CE_2[H_{h_2}] + \theta\bar\nabla\CE_2[H_{h_1}])\,d\theta\Bigr)\\
&\hskip6cm
\otimes\bar\nabla(\CE_2[H_{h_1}]-\CE_2[H_{h_2}])\otimes \CE_1[\bu_1].
\end{align*}
By \eqref{4.5.34}, \eqref{4.5.1}, \eqref{1.3.2}, \eqref{est:1},  
 and \eqref{diffh.1}, we have
\begin{align}
&\|e^{-\gamma t}\pd_t G_1\|_{L_p(\BR, L_q(\dot\Omega))} \nonumber \\
&\quad \leq C\{\|h_1-h_2\|_{L_\infty((0, T), W^{2-1/q}_q(\Gamma))}
\|\pd_t\CE_1[\bu_1]\|_{L_p(\BR, L_q(\dot\Omega))} \nonumber \\
&\qquad +T^{1/p}\|\pd_t(h_1-h_2)\|_{L_\infty((0, T), W^{1-1/q}_q(\Gamma)}
\|\CE_1[\bu_1]\|_{L_\infty((0, T), H^1_q(\dot\Omega))}\nonumber \\
&\qquad +T^{1/p}(\|\pd_t\CE_2[H_{h_1}]\|_{L_\infty(\BR, W^{1-1/q}_q(\Gamma))}
+ \|\pd_t\CE_2[H_{h_2}]\|_{L_\infty(\BR, W^{1-1/q}_q(\Gamma))})\nonumber  \\
&\qquad\quad\times \|h_1-h_2\|_{L_\infty((0, T), W^{2-1/q}_q(\Gamma))}
\|\CE_1[\bu_1]\|_{L_\infty((0, T), H^1_q(\dot\Omega))}\nonumber \\
&\quad \leq C\{T^{1/p'}\|\pd_t(h_1-h_2)\|_{L_p((0, T), W^{2-1/q}_q(\Gamma))}(L+B)
\nonumber \\
&\qquad + TL\|\pd_t(h_1-h_2)\|_{L_p((0, T), W^{2-1/q}_q(\Gamma))}(L+B)
\nonumber \\
&\qquad + T^{1/p}\|\pd_t(h_1-h_2)\|_{L_\infty((0, T), W^{1-1/q}_q(\Gamma))}
(L+B)\}
\nonumber \\
&\quad \leq CT^{1/p}(L+B)(\|\pd_t(h_1-h_2)\|_{L_p((0, T), W^{2-1/q}_q(\Gamma))}
+ \|\pd_t(h_1-h_2)\|_{L_\infty((0, T), W^{1-1/q}_q(\Gamma))}),
\label{eq:5.21}
\end{align}
where we have used $T^{1/p'}L \leq 1$. 
  Since $\CE_1[\bu_1]- \CE_2[\bu_2]
=e_T[\bu_1-\bu_2]$, writing 
\begin{align*}
\pd_tG_2 &= \CG_2(\bar\nabla\CE_2[H_{h_2}])\bar\nabla \CE_2[H_{h_2}]\otimes\pd_te_T[\bu_1-\bu_2]
+ \CG_2(\bar\nabla\CE_2[H_{h_2}])\bar\nabla\pd_t\CE_2[H_{h_2}])
\otimes e_T[\bu_1-\bu_2]\\
&+ (d_\bK\CG_2)(\bar\nabla\CE_2[H_{h_2}])\pd_t\bar\nabla\CE_2[H_{h_2}]
\otimes\bar\nabla\CE_2[H_{h_2}]
\otimes e_T[\bu_1-\bu_2],
\end{align*}
by \eqref{4.5.34}, \eqref{4.5.34*}, \eqref{est:1}, and \eqref{est:6}
\begin{align*}
&\|e^{-\gamma t}\pd_tG_2\|_{L_p(\BR, L_q(\dot\Omega))} \\
&\quad\leq C LT^{1/p'}
\|\pd_t(\bu_1-\bu_2)\|_{L_p((0, T), L_q(\dot\Omega))}
+\|\pd_t\CE_2[H_{h_2}]\|_{L_p(\BR, W^{2-1/q}_q(\Gamma))}
\|\bu_1-\bu_2\|_{L_\infty((0, T), L_q(\dot\Omega))}\\
&\quad  \leq  C\{LT^{1/p'}\|\pd_t(\bu_1-\bu_2)\|_{L_p((0, T), L_q(\dot\Omega))}
+T^{\frac{s}{p'(1+s)}}LE^1_T(\bu_1-\bu_2)\}, 
\end{align*}
which, combined with \eqref{eq:5.21}, yields that
\begin{equation}\label{eq:5.22}
\|e^{-\gamma t}\pd_t\tilde \CG\|_{L_p(\BR, L_q(\dot\Omega))}
\leq C(L+B)T^{\frac{s}{p'(1+s)}}
E_T(\bu_1-\bu_2, \bG_1-\bG_2, h_1-h_2).
\end{equation}
Applying Theorem \ref{thm:4.2.1} to equations \eqref{eq:5.1}
and using \eqref{eq:5.3}, \eqref{eq:5.6}, \eqref{eq:5.17},
\eqref{eq:5.20}, and \eqref{eq:5.22}, we have
\begin{equation}\label{eq:5.23}
E_T(\bv_1-\bv_2) + E^2_T(\rho_1-\rho_2)
\leq C(1+\gamma_1^{1/2})e^{\gamma_1 }L^3T^{\frac{s}{p'(1+s)}}E_T(\bu_1-\bu_2, \bG_1-
\bG_2, h_1-h_2),
\end{equation}
provided that $LT^{1/p'} \leq 1$, $0 < T = \kappa=\epsilon < 1$, and 
$L > B \geq 1$.

Moreover, by the third equation of \eqref{eq:5.1}, 
\eqref{eq:5.5}, and \eqref{4.4.3}, we have
\begin{align*}
\|\pd_t(\rho_1-\rho_2)\|_{L_\infty((0, T), W^{1-1/q}_q(\Gamma))}
&\leq C(B\|\rho_1-\rho_2\|_{L_\infty((0, T), W^{2-1/q}_q(\Gamma))}
+ \|\bv_1-\bv_2\|_{L_\infty((0, T), H^1_q(\dot\Omega))} \\
&\quad +T^{1/p'}(L+B)E_T(\bu_1-\bu_2, \bG_1-\bG_2, h_1-h_2))\\
&\leq C(BT^{1/p'}\|\pd_t(\rho_1-\rho_2)\|_{L_p((0, T), W^{2-1/q}_q(\Gamma))}
+ \|\bv_1-\bv_2\|_{L_\infty((0, T), H^1_q(\dot\Omega))} \\
&\quad +T^{1/p'}(L+B)E_T(\bu_1-\bu_2, \bG_1-\bG_2, h_1-h_2))
\end{align*}
which, combined with \eqref{eq:5.23} and $BT^{1/p'} \leq 1$, yields that
\begin{equation}\label{eq:5.24}\begin{aligned}
&E^1_T(\bv_1-\bv_2) + E^2_T(\rho_1-\rho_2) 
+ \|\pd_t(\rho_1-\rho_2)\|_{L_\infty((0, T), W^{1-1/q}_q(\Gamma))}
\\
&\quad
\leq C((1+\gamma_1^{1/2})e^{\gamma_1}L^3T^{\frac{s}{p'(1+s)}}
+ (L+B)T^{1/p'})E_T(\bu_1-\bu_2, \bG_1-\bG_2,
h_1-h_2)\\
&\quad \leq M_2L^3T^{\frac{s}{p'(1+s)}}E_T(\bu_1-\bu_2, \bG_1-\bG_2,
h_1-h_2)
\end{aligned}\end{equation}
with some constant $M_2$ depending on $s \in (0, 1-2/p)$ and $\gamma_1 > 0$ 
provided that $LT^{1/p'} \leq 1$, $1 \leq B \leq L$, and  $0 < T=\kappa=\epsilon < 1$. 

We now consider  $\tilde \bH = \bH_1- \bH_2$. . We first consider $\CF_2$.  
In view of  \eqref{maineq:2*},  we may write 
$$\bff_2(\bu, \bG, H_\rho) = \bV^2_\bh(\bar\nabla H_\rho)\bff_4(\bu, \bG, H_\rho)$$
with
$$\bff_4(\bu, \bG, H_\rho) = (\nabla\bG\otimes\pd_t H_\rho + 
\bar\nabla H_\rho\otimes\nabla^2\bG + \bar\nabla^2 H_\rho\otimes \nabla\bG
+ \nabla\bu\otimes\bG + \bu\otimes\nabla\bG,
$$
where $\bV^2_\bff(\bK)$ is some matrix of smooth funtions of $\bK$ for 
$|\bK| < \delta$. And then, employing the same argument as in proving \eqref{eq:5.3}, 
we have
\begin{equation}\label{eq:5.25}
\|\CF_2\|_{L_p((0, T), L_q(\dot\Omega))}
\leq CT^{1/p}(L+B)E_T(\bu_1-\bu_2, \bG_1-\bG_2,
h_1-h_2)
\end{equation}
provided that $T^{1/p'}L \leq 1$, $1 < B \leq L$, and $0 < T=\epsilon=\kappa < 1$.

Concerning $\CH_2$ and  $\CH_3$, in view of \eqref{eq:4.5.8}, 
we define $\tilde \CH_2$ and $\tilde \CH_3$ by setting 
$\tilde \CH_2 = \CB_1 + b(y)\CB_2 + \CB_3$ with 
\begin{align*}
\CB_1 & = (\tilde \bV^3_\bh(\cdot, \bar\nabla\CE_2[H_{\rho_1}]) - 
 \tilde \bV^3_\bh(\cdot, \bar\nabla\CE_2[H_{\rho_2}]))\otimes\nabla \CE_1[tr[\bu_1]]\\
&+ \bV^3_\bh(\cdot, \bar\nabla\CE_2[H_{\rho_2}])\bar\nabla\CE_{\rho_2}[H_{\rho_2}]
\otimes \nabla(\CE_1[tr[\bu_1]]-\CE_1[tr[\bu_2]]); \\
\CB_2 & = (\CE_1[tr[\bu_1]]-\CE_1[tr[\bu_2]])\otimes\CE_1[tr[\bG_1]]
+ \CE_1[tr[\bu_2]]\otimes(\CE_1[tr[\bG_1]]-\CE_1[tr[\bG_2]]); \\
\CB_3 & = (\tilde\bV^4_\bh(\cdot, \bar\nabla\CE_2[H_{\rho_1}]) - 
\tilde\bV^4_\bh(\cdot, \bar\nabla\CE_2[H_{\rho_2}]))\CE_1[tr[\bu_1]]\otimes
\CE_1[tr[\bG_1]] \\
& + \bV^4_\bh(\cdot, \bar\nabla\CE_2[H_{\rho_2}])
\bar\nabla\CE_2[H_{\rho_2}]\otimes \CB_2,
\end{align*}
where we have set 
$$\tilde\bV^3_\bh(\cdot, \bK) = \bV^3_\bh(\cdot, \bK)\bK, \quad
\tilde\bV^4_\bh(\cdot, \bK) = \bV^4_\bh(\cdot, \bK)\bK;
$$
and 
\begin{align*}\tilde \CH_3 & =-\mu\sum_{j,k=1}^N(\tilde V_{0jk}(\nabla \CE_2[H_{\rho_1}]) - 
\tilde V_{0jk}(\nabla\CE_2[h_{\rho_2}]))
\frac{\pd}{\pd y_k}\CE_1[tr[\bu_1]]_j \\
& - \sum_{j,k=1}^N V_{0jk}(\nabla\CE_2[H_{\rho_2}])\nabla\CE_2[H_{\rho_2}]
\frac{\pd}{\pd y_k}(\CE_1[tr[\bu_1]]_j - 
\CE_1[tr[\bu_2]]_j).
\end{align*}
Obviously, $\tilde\CH_i$  are defined for $t \in \BR$, 
and  $\tilde \CH_i = \CH_i$ for $t \in (0, T)$ for $i=3,4$.  Employing the same 
argument as in proving \eqref{eq:5.8} and \eqref{eq:5.17}, we have
\begin{equation}\label{eq:5.26}\begin{aligned}
&\|e^{-\gamma t}\tilde\CH_3\|_{H^{1/2}_p(\BR, L_q(\dot\Omega))} + 
\|e^{-\gamma t}\tilde \CH_3\|_{L_p(\BR, H^1_q(\dot\Omega))}\\
&\quad 
\leq CL(B+L)T^{1/(2p')}E_T(\bu_1-\bu_2, \bG_1-\bG_2, h_1-h_2); \\
&\|e^{-\gamma t}\tilde\CH_2\|_{H^{1/2}_p(\BR, L_q(\dot\Omega))}
+ \|e^{-\gamma t}\tilde\CH_2\|_{L_p(\BR, H^1_q(\dot\Omega))}\\
&\quad 
\leq CL(B+L)^2T^{1/(2p)}
E_T(\bu_1-\bu_2, \bG_1-\bG_2, h_1-h_2)
\end{aligned}\end{equation}
provided that $T^{1/p'}L \leq 1$ and $0 < \epsilon = T < 1$, $1 \leq L$, and $1 \leq B$. . 

We finally consider $\CK_1$ and $\CK_2$.  As was mentioned in \eqref{4.5.34}, 
we may assume that 
$$\sup_{t\in\BR}\|\CE_2[H_{\rho_i}]\|_{H^1_\infty(\Omega)} \leq \delta
\quad(i=1,2).
$$
In view of \eqref{eq:4.5.9}, we set
$\tilde\CK = \tilde \CK_1 + \tilde \CK_2$ with
\begin{align*}
\tilde \CK_1 & = (\tilde\bV^5_\bk(\cdot\bar\nabla \CE_2[H_{\rho_1}])- 
\tilde\bV^5_\bk(\cdot\bar\nabla \CE_2[H_{\rho_2}]))\otimes
\CE_1[tr[\bG_1]] \\
\tilde \CK_2 & = \bV^5_\bk(\cdot, \bar\nabla\CE_2[H_{\rho_2}])\bar\nabla\CE_2[H_{\rho_2}]
\otimes(\CE_1[tr[\bG_1]]-\CE_2[tr[\bG_2]]),
\end{align*}
where we have set $\tilde\bV^5_\bk(\cdot, \bK) = \bV^5_\bk(\cdot, \bK)\bK$.  
Obviously, $\tilde \CK$ is defined for $t \in \BR$ and 
$\tilde \CK = (\CK_1, \CK_2)$ for $t \in (0, T)$.  To estimate $\tilde \CK_1$, 
we write 
\begin{align*}
&\tilde\bV^5_\bk(\cdot, \bar\nabla \CE_2[H_{\rho_1}]) - 
\tilde\bV^5_\bk(\cdot, \bar\nabla \CE_2[H_{\rho_2}]) \\
&\quad = \int^1_0(d_\bK \tilde\bV^5_\bk)(\cdot, \bar\nabla\CE_2[H_{\rho_2}]
+ \theta \bar\nabla(\CE_2[H_{\rho_1}]- \CE_2[H_{\rho_2}]))\,d\theta
\bar\nabla e_T[H_{\rho_1}- H_{\rho_2}], 
\end{align*}
and then by \eqref{eq:4.5.1} we have 
\begin{align*}
&\|\tilde\CK_1\|_{H^2_q(\dot\Omega)} \\ 
&\quad\leq  C\{\|\CE_1[tr[\bG_1]]\|_{H^2_q(\dot\Omega)}
\|\bar\nabla e_T[H_{\rho_1}-H_{\rho_2}]\|_{H^1_q(\dot\Omega)}
+ \|\CE_1[tr[\bG_1]]\|_{H^1_q(\dot\Omega)}\|\bar\nabla
e_T[H_{\rho_1}-H_{\rho_2}]\|_{H^2_q(\Omega)} \\
&\quad +(\|\bar\nabla\CE_2[H_{\rho_1}]\|_{H^2_q(\Omega)} 
+ \|\bar\nabla\CE_2[H_{\rho_2}]\|_{H^2_q(\Omega)})
(1 + \|\bar\nabla\CE_2[H_{\rho_1}]\|_{H^1_q(\Omega)} 
+ \|\bar\nabla\CE_2[H_{\rho_2}]\|_{H^1_q(\Omega)})\\
&\hskip7cm\times\|\bar\nabla e_T[H_{\rho_1}-H_{\rho_2}]\|_{|H^1_q(\Omega)}
\|\CE_1[tr[\bG_1]]\|_{H^1_q(\dot\Omega)}\}.
\end{align*}
Noting that $e_T[H_{\rho_1}-H_{\rho_2})$ vanishes for $t \not\in (0, 2T)$,
we have
$$\|\bar\nabla(H_{\rho_1}-H_{\rho_2})\|_{L_\infty(\BR, H^1_q(\Omega)}
\leq CT^{1/p'}\|\pd_t(\rho_1-\rho_2)\|_{L_p((0, T), W^{2-1/q}_q(\Gamma))}.
$$
Thus, by \eqref{est:1} and \eqref{4.10***}, 
\begin{align*}
\|e^{-\gamma t}\tilde\CK_1\|_{L_p(\BR, H^2_q(\dot\Omega))} 
&\leq  C\{T^{1/p'}(L+e^{2(\gamma-\gamma_1)}B)\|\pd_t(\rho_1-\rho_2)\|_{L_p((0, T), W^{2-1/q}_q(\Gamma))}\\
&\quad + (Be^{2(\gamma-\gamma_1)} + (L+B)T^{\frac{s}{p'(1+s)}})\|\rho_1-\rho_2\|_{L_p((0, T), W^{3-1/q}_q(\Gamma))}
\\
&\quad + L(L+B)T^{1/p'}\|\pd_t(\rho_1-\rho_2)\|_{L_p((0, T), W^{2-1/q}_q(\Gamma))}\}\\
&\leq C(e^{2(\gamma-\gamma_1)}B + L(L+B)T^{\frac{s}{p'(1+s)}})E^2_T(\rho_1-\rho_2).
\end{align*}
Using \eqref{eq:4.5.1}, we have
\begin{align*}
&\|\tilde \CK_2\|_{H^2_q(\dot\Omega)} \\
&\quad \leq C\{\|\bar\nabla\CE_2[H_{\rho_2}]\|_{H^2_q(\Omega)}
\|e_T[tr[\bG_1]-tr[\bG_2]]\|_{H^1_q(\dot\Omega)}
+ \|\bar\nabla\CE_2[H_{\rho_2}]\|_{H^1_q(\Omega)}
\|e_T[tr[\bG_1]-tr[\bG_2]]\|_{H^2_q(\dot\Omega)} \\
&\quad +\|\bar\nabla\CE_2[H_{\rho_2}]\|_{H^2_q(\Omega)}
(1 + \|\bar\nabla\CE_2[H_{\rho_2}]\|_{H^1_q(\Omega)})
\|\bar\nabla\CE_2[H_{\rho_2}]\|_{H^1_q(\Omega)}
\|e_T[tr[\bG_1]-tr[\bG_2]]\|_{H^1_q(\dot\Omega)}. 
\end{align*}
Employing the same argument as in  \eqref{eq:5.8*}, we have
$$\|tr[\bG_1]-tr[\bG_2]\|_{L_\infty((0, T), H^1_q(\dot\Omega))}
\leq CT^{\frac{s}{p'(1+s)}}E^1_T(\bG_1-\bG_2)
$$
for some $s \in (0, 1-2/p)$, and so we have
\begin{align*}
\|e^{-\gamma t}\tilde\CK_2\|_{L_p(\BR, H^2_q(\dot\Omega))}
& \leq C\{LT^{\frac{s}{p'(1+s)}}E^1_T(\bG_1-\bG_2) 
+ LT^{1/p'}\|\bG_1-\bG_2\|_{L_p((0, T), H^2_q(\dot\Omega))}\}\\
&\leq CLT^{\frac{s}{p'(1+s)}}E^1_T(\bG_1-\bG_2).
\end{align*}

By \eqref{4.5.1} and \eqref{eq:4.5.1}, we have
\begin{align*}
&\|\pd_t\tilde \CK_1\|_{L_q(\dot\Omega)} \\
& \leq C\{(\|\pd_t\bar\nabla\CE_2[H_{\rho_1}]\|_{L_q(\dot\Omega)}
+ \|\pd_t\bar\nabla\CE_2[H_{\rho_2}]\|_{L_q(\dot\Omega)})
\|\bar\nabla e_T[H_{\rho_1}-H_{\rho_2}]\|_{H^1_q(\dot\Omega)}
\|\CE_1[tr[\bG_1]]\|_{H^1_q(\dot\Omega)}\\
&+\|\pd_t\bar\nabla e_T[H_{\rho_1}-H_{\rho_2}]\|_{L_q(\dot\Omega)}
\|\CE_1[tr[\bG_1]]\|_{H^1_q(\dot\Omega)}
+ \|\bar\nabla e_T[H_{\rho_1}-H_{\rho_2}]\|_{H^1_q(\dot\Omega)}
\|\pd_t\CE_1[tr[\bG_1]]\|_{L_q(\dot\Omega)};
\\
&\|\pd_t\tilde \CK_2\|_{L_q(\dot\Omega)} 
\leq C\{\|\pd_t\bar\nabla \CE_2[H_{\rho_2}]\|_{L_q(\dot\Omega)}
\|\bar\nabla \CE_2[H_{\rho_2}]\|_{H^1_q(\dot\Omega)}
\|e_T[tr[\bG_1]-tr[\bG_2]]\|_{H^1_q(\dot\Omega)}\\
&+\|\pd_t\bar\nabla\CE_2[H_{\rho_2}]\|_{L_q(\dot\Omega)}
\|e_T[tr[\bG_1]-tr[\bG_2]]\|_{H^1_q(\dot\Omega)}
+ \|\bar\nabla\CE_2[H_{\rho_2}]\|_{H^1_q(\dot\Omega)}
\|\pd_t e_T[tr[\bG_1]-tr[\bG_2]]\|_{L_q(\dot\Omega)}
\end{align*}
Thus, we have
\begin{align*}
&\|e^{-\gamma t}\pd_t\tilde \CK_1\|_{L_p(\BR, L_q(\dot\Omega))}
+ \|e^{-\gamma t}\pd_t\tilde \CK_2\|_{L_p(\BR, L_q(\dot\Omega))}\\
&\leq C\{TL\|\pd_t(\rho_1-\rho_2)\|_{L_p((0, T), W^{2-1/q}_q(\Gamma))}(B+L)
+ T^{1/p}\|\pd_t(\rho_1-\rho_2)\|_{L_\infty((0, T), W^{1-1/q}_q(\Gamma))}(B+L)\\
&+T^{1/p'}\|\pd_t(\rho_1-\rho_2)\|_{L_p((0, T), W^{2-1/q}_q(\Gamma))}(B+L)
+ T^{1/p}L
T^{\frac{s}{p'(1+s)}}E^1_T(\bG_1-\bG_2)\\
&+ T^{1/p}LT^{\frac{s}{p'(1+s)}}E^1_T(\bG_1-\bG_2)
+  LT^{1/p'}\|\pd_t(\bG_1-\bG_2)\|_{L_p((0, T), L_q(\dot\Omega))}\}\\
& \leq CT^{1/p}(L+B)(\tilde E^2_T(\rho_1-\rho_2) + E^2_T(\bG_1-\bG_2)),
\end{align*}
where we have set $\tilde E^2_T(\rho_1-\rho_2) = E^2_T(\rho_1-\rho_2) 
+ \|\pd_t(\rho_1-\rho_2)\|_{L_\infty((0, T), W^{1-1/q}_q(\Gamma))}$.
Putting these inequalities together gives that
\begin{equation}\label{eq:5.27}\begin{aligned}
\|e^{-\gamma t}\tilde \CK\|_{L_p(\BR, H^2_q(\dot\Omega))}
+ \|e^{-\gamma t}\pd_t \tilde \CK\|_{L_p(\BR, L_q(\dot\Omega))} 
&\leq C(e^{2(\gamma-\gamma_1)}B + L(L+B)T^{\frac{s}{p'(1+s)}})\tilde E^2_T(\rho_1-\rho_2) \\
&\quad 
+ C(L+B)T^{T^{\min(\frac{1}{p'}, \frac1p+\frac{s}{p'(1+s)})}}E^1_T(\bG_1-\bG_2).
\end{aligned}\end{equation}
Applying Theorem \ref{thm:4.2.1} to equations \eqref{eq:5.2} and using
\eqref{eq:5.25}, \eqref{eq:5.26}, and \eqref{eq:5.27}, 
give that
\begin{equation}\label{eq:5.28}\begin{aligned}
&\|e^{-\gamma t}(\bH_1- \bH_2)\|_{L_p((0, T), H^2_q(\dot\Omega))}
+ \|e^{-\gamma t}\pd_t(\bH_1-\bH_2)\|_{L_p((0, T), L_q(\dot\Omega))} \\
&\quad\leq Ce^{\gamma_1}\{(B+L(L+B)T^{\frac{s}{p'(1+s)}})\tilde E^2_T(\rho_1-\rho_2)
+ L(B+L)^2T^{1/(2p)}E_T(\bu_1-\bu_2, \bG_1-\bG_2, h_1-h_2)\}.
\end{aligned}\end{equation}
Combining \eqref{eq:5.24} and \eqref{eq:5.28} yields that
$$E_T(\bv_1-\bv_2, \bH_1-\bH_2, \rho_1-\rho_2) \leq \CN_T(L,B)
E_T(\bu_1-\bu_2, \bG_1-\bG_2, h_1-h_2)
$$
with
$$\CN_T(L,B)=(Ce^{\gamma_1}(B+L(L+B)T^{\frac{s}{p'(1+s)}})+1)M_2L^3T^{\frac{s}{p'(1+s)}}
+ Ce^{\gamma_1}L(B+L)^2T^{1/(2p)}).$$  
Thus, choosing $T$ so small that 
$\CN_T(L,B) \leq 1/2$, 
we see that the $\Phi$ is a contraction map
from $U_T$ into itself, and so there is a unique fixed point
$(\bu, \bG, h) \in U_T$ of the map $\Phi$.  This $(\bu, \bG, h)$
solves equations \eqref{mhd.3} uniquely and possessing the 
properties mentioned in Theorem \ref{thm:main}.  This 
completes the proof of Theorem \ref{thm:main}. 
%\end{document} 

\end{document}